\documentclass [12pt]{amsart}
\usepackage[utf8]{inputenc}
\pdfoutput=1

\usepackage{amsmath,amssymb,amsthm,amsfonts}
\usepackage{mathtools}

\usepackage{comment}
\usepackage[bookmarksdepth=4]{hyperref}
\usepackage{bbm}
\usepackage{mathrsfs}

\usepackage{tikz,graphicx,color}
\usepackage{tikz-cd}
\usepackage{tikz-3dplot}
\usetikzlibrary{calc}
\usetikzlibrary{arrows}
\usetikzlibrary{shapes}
\usetikzlibrary{patterns}
\usetikzlibrary{positioning}
\usetikzlibrary{arrows.meta}
\usetikzlibrary{decorations.markings}
\usepackage{epstopdf}
\pdfpageattr{/Group <</S /Transparency /I true /CS /DeviceRGB>>}

\usepackage[arrow]{xy}
\usepackage{diagbox}
\usepackage{subfig}
\usepackage{arcs}
\usepackage{xcolor}
\usepackage{tabu}
\usepackage{booktabs}

\usepackage[margin=1.01in]{geometry}

\usepackage{soul}
\usepackage{accents}

\usepackage{enumitem}
\usepackage{letltxmacro}
\usepackage{thmtools,etoolbox}

\def\myarabic#1{\normalfont(\roman{#1})}
\newlist{theoremlist}{enumerate}{1}
\setlist[theoremlist]{label=\myarabic{theoremlisti},ref={\myarabic{theoremlisti}},itemindent=0pt,labelindent=0pt,
leftmargin=*,noitemsep}

\makeatletter
\renewcommand{\p@theoremlisti}{\perh@ps{\thetheorem}}
\protected\def\perh@ps#1#2{\textup{#1#2}}
\newcommand{\itemrefperh@ps}[2]{\textup{#2}}
\newcommand{\itemref}[1]{\begingroup\let\perh@ps\itemrefperh@ps\ref{#1}\endgroup}
\makeatother

\usepackage{nameref,hyperref}
\usepackage[capitalize]{cleveref}

\newtheorem{theorem}{Theorem}[section]

\newtheorem{lemma}[theorem]{Lemma}
\newtheorem{proposition}[theorem]{Proposition}
\newtheorem{corollary}[theorem]{Corollary}
\theoremstyle{definition}
\newtheorem{remark}[theorem]{Remark}
\theoremstyle{definition}
\newtheorem{definition}[theorem]{Definition}
\newtheorem{conjecture}[theorem]{Conjecture}

\theoremstyle{definition}
\newtheorem{problem}[theorem]{Problem}
\theoremstyle{definition}
\newtheorem{example}[theorem]{Example}

\def\S{Section~}

\addtotheorempostheadhook[theorem]{\crefalias{theoremlisti}{theorem}}
\addtotheorempostheadhook[lemma]{\crefalias{theoremlisti}{lemma}}
\addtotheorempostheadhook[proposition]{\crefalias{theoremlisti}{proposition}}

\def\Acal{\mathcal{A}}\def\Fcal{\mathcal{F}}\def\Gcal{\mathcal{G}}\def\Ical{\mathcal{I}}\def\Lcal{\mathcal{L}}\def\Pcal{\mathcal{P}}\def\Tcal{\mathcal{T}}\def\Ucal{\mathcal{U}}\def\Xcal{\mathcal{X}}

% ...

\def\one{{\mathbbm{1}}}
\def\C{\mathbb{C}}
\def\R{\mathbb{R}}

\def\Z{\mathbb{Z}}
\def\Q{\mathbb{Q}}
\def\P{\mathbb{P}}
% ...

\newcommand\parr[1]{{({#1})}}
% ...

\def\<{{\langle}}
\def\>{{\rangle}}

\def\la{{\lambda}}
\def\l{{\lambda}}

\def\CP{{\C P}}

\def\det{{ \operatorname{det}}}
\def\tr{{ \operatorname{tr}}}

\def\diag{{ \operatorname{diag}}}
\def\rank{{ \operatorname{rank}}}

\def\codim{ \operatorname{codim}}

\def\op{{ \operatorname{op}}}

\def\Span{ \operatorname{Span}}

\def\Cone{\operatorname{Cone}}

\def\CC{{\mathbb C}}

\def\CP{{\CC\mathbb P}}

\def\GL{\operatorname{GL}}
\def\SL{\operatorname{SL}}

\def\Gr{\operatorname{Gr}}
\def\Grtnn{\Gr_{\ge 0}}

\def\codim{{\rm codim}}
\def\Z{{\mathbb Z}}
\def\R{{\mathbb R}}
\def\Gr{{\rm Gr}}

\def\Cone{{\rm Cone}}

\def\GL{{\rm GL}}
\def\diag{{\rm diag}}

\def\Inv{{\rm Inv}}

\def\Waff{\tW}

\newcommand\pcvar[1]{\Pi_{#1}^\circ}

\def\id{{\operatorname{id}}}

\newcommand{\Le}{\textup{\protect\scalebox{-1}[1]{L}}}

\begin{document}
\numberwithin{equation}{section}
\newcommand\TODOL[1]{\textcolor{blue!20!black!30!green}{TODO-LATER:#1}}
\title[Regularity theorem for totally nonnegative flag varieties]{Regularity theorem for totally nonnegative flag~varieties}
\author{Pavel Galashin}
\address{Department of Mathematics, University of California, Los Angeles, 520 Portola Plaza,
Los Angeles, CA 90025, USA}
\email{\href{mailto:galashin@math.ucla.edu}{galashin@math.ucla.edu}}
\author{Steven N. Karp}
\address{LaCIM, Universit\'{e} du Qu\'{e}bec \`{a} Montr\'{e}al, CP 8888, Succ.\ Centre-ville, Montr\'{e}al, QC H3C 3P8, Canada}
\email{\href{mailto:karp.steven@courrier.uqam.ca}{karp.steven@courrier.uqam.ca}}
\author{Thomas Lam}
\address{Department of Mathematics, University of Michigan, 2074 East Hall, 530 Church Street, Ann Arbor, MI 48109-1043, USA}
\email{\href{mailto:tfylam@umich.edu}{tfylam@umich.edu}}
\thanks{P.G.\ was supported by an Alfred P. Sloan Research Fellowship and by the National Science Foundation under Grants No.~DMS-1954121 and No.~DMS-2046915. S.N.K.\ was supported by the Natural Sciences and Engineering Research Council of Canada under a Postdoctoral Fellowship. T.L.\ was supported by a von Neumann Fellowship from the Institute for Advanced Study and by the National Science Foundation under Grants No.~DMS-1464693 and No.~DMS-1953852.}

\begin{abstract}
We show that the totally nonnegative part of a partial flag variety $G/P$ (in the sense of Lusztig) is a regular CW complex, confirming a conjecture of Williams. In particular, the closure of each positroid cell inside the totally nonnegative Grassmannian is homeomorphic to a ball, confirming a conjecture of Postnikov.
\end{abstract}

\date{\today}


\maketitle
\setcounter{tocdepth}{1}
\tableofcontents
\newcommand{\smat}[1]{\left[\begin{smallmatrix}
      #1
    \end{smallmatrix}\right]}
\newcommand{\pmat}[1]{\begin{pmatrix}
#1
  \end{pmatrix}}

\def\ub{{\bar u}}

\def\xs{\operatorname{xs}}
\def\ys{\operatorname{ys^{-1}}}
\def\ut{{\tilde u}}

\def\Lie{\operatorname{Lie}\;}
\def\U{{\mathcal U}}
\def\B{{\mathcal B}}
\def\X{\Xcal}
\def\G{\Gcal}
\def\U{\Ucal}
\def\Ft{{\tilde F}}

\def\pj{^{\parr J}}
\def\Uj{U\pj}
\def\Cj{C\pj}
\def\Cuj{C\pj_u}
\def\Phij{\Phi\pj}
\def\gj{g\pj}
\let\uu\u
\def\u{h}
\def\hj{h^{\parr J}}
\def\Lie{\operatorname{Lie}}

\def\bs{\backslash}
\let\ii\i
\def\i{{\mathbf{i}}}
\def\j{{\mathbf{j}}}
\def\t{{\mathbf{t}}}

\def\d#1{\dot{#1}}
\def\ds{\d{s}}
\def\dw{\d{w}}
\def\du{\d{u}}
\def\dv{\d{v}}
\def\da{\d{a}}
\def\db{\d{b}}
\def\dc{\d{c}}
\def\dz{\d{z}}
\def\dr{\d{r}}
\def\dq{\d{q}}
\def\gfr{\mathfrak{g}}
\def\alphacheck{\alpha^\vee}
\def\X{\accentset{\circ}{X}}
\def\Xcl{X}

\def\Rich_#1^#2{\accentset{\circ}{R}_{#1,#2}}
\def\Richcl_#1^#2{R_{#1,#2}}
\def\RRich_#1^#2{\accentset{\circ}{R}_{#1,#2}^\R}
\def\RRichcl_#1^#2{R_{#1,#2}^\R}
\def\Rtp_#1^#2{R_{#1,#2}^{>0}}
\def\Rtnn_#1^#2{R_{#1,#2}^{\geq0}}

\def\Gtnn{{G_{\geq0}}}
\def\GBtnn{{(G/B)_{\geq0}}}
\def\GPtnn{{(G/P)_{\geq0}}}

\def\PR_#1^#2{\accentset{\circ}{\Pi}_{#1,#2}}
\def\PRtp_#1^#2{\Pi_{#1,#2}^{>0}}
\def\PRtnn_#1^#2{\Pi_{#1,#2}^{\geq0}}
\def\PRcl_#1^#2{\Pi_{#1,#2}}
\def\PRR_#1^#2{\accentset{\circ}{\Pi}_{#1,#2}^\R}
\def\PRRcl_#1^#2{\Pi_{#1,#2}^\R}

\def\sf{\operatorname{sf}}
\def\Tsf{T_{\sf}}
\def\bw{{\mathbf{w}}}
\def\bv{{\mathbf{v}}}
\def\bu{{\mathbf{u}}}
\def\ba{{\mathbf{a}}}
\def\bb{{\mathbf{b}}}
\def\bc{{\mathbf{c}}}
\def\wnot{{\mathbf{w_0}}}
\def\woj{w_J}
\def\dwoj{\dw_J}

\def\Fcal{\mathcal{F}}
\def\Qsf{\Fcal_{\sf}}
\def\Qsfs{\Fcal^\ast_{\sf}}
\def\QsfN#1{\Fcal_{\sf}^{#1}}
\def\QsfsN#1{(\Fcal^\ast_{\sf})^{#1}}
\def\Tsf{T^{\sf}}
\def\Bsf{B^{\diamond}}
\def\Usf{U_{\sf}}

\def\Utp{U_{>0}}
\def\Utnn{U_{\geq0}}

\def\HKL{{\operatorname{HKL}}}
\def\maxx{{\operatorname{max}}}
\def\minn{{\operatorname{min}}}
\def\leqJ{\preceq}
\def\lessJ{\prec}
\def\geqJ{\succeq}
\def\gessJ{\succ}
\def\QJfilt#1{Q_J^{\succeq#1}}
\def\QJfilter#1#2{Q_J^{\succeq(#1,#2)}}
\def\QJideal#1#2{Q_J^{\preceq(#1,#2)}}

\def\pd#1{_{(#1)}}
\def\pu#1{^{(#1)}}

\def\bx{{\mathbf{x}}}
\def\by{{\mathbf{y}}}
\def\bi{{\mathbf{i}}}
\def\bj{{\mathbf{j}}}

\def\Gomp{G_0^\mp}
\def\Gopm{G_0^\pm}
\def\Deltamp{\Delta^\mp}
\def\Deltapm{\Delta^\pm}
\def\Goj{G_0\pj}
\def\K{{\mathbb{K}}}
\def\k{{\mathbbm{k}}}
\def\kast{\k^\ast}

\def\RedMR{\operatorname{Red}}
\def\Red{\operatorname{Red}}
\def\gMR#1#2{{\mathbf{g}}_{#1,#2}}
\def\gMRvw{\gMR\bv\bw}
\def\GMRtp#1#2{G^{>0}_{#1,#2}}
\def\GMRtpvw{\GMRtp\bv\bw}
\def\weightL{Y(T)}
\def\rootL{Q_\Phi}
\def\kls{{\operatorname{KLS}}}
\def\Qkls{Q_P(W,W_J)}
\def\leqkls{\preceq_{P}}

\def\flag{{\operatorname{flag}}}
\def\minormap{\Delta^\flag}
\def\Proj{{\mathbb{P}}}

\def\Kast{\K^{\ast}}
\def\GMRsf#1#2{G^{\sf}_{#1,#2}}

\def\Hom{\operatorname{Hom}}

\def\wj{w^J}
\def\TR{T(\R)} 
\def\GR{G(\R)}
\def\BR{B(\R)}
\def\BRm{B_-(\R)}
\def\PAR{P(\R)}
\def\GBR{(G/B)_{\R}}
\def\GPR{(G/P)_{\R}}
\def\conjug{\sigma_\C}

\def\hjmap{\kappa}
\def\hjmp_#1{\hjmap_{#1}}
\def\hjx{\hjmap_x}
\def\Ftmap{\eta}
\def\Ft{\Ftmap}
\def\zetamap{\zeta_{u,v}\pj}
\def\Guvbig{{G_{u,v}\pj}}
\def\pidup{\pi_{\du P_-}}
\def\ut{u_\maxx}
\def\ut{{\tilde u}}
\def\dut{\d{\ut}}
\def\Fcalo{{\Fcal^{\diamond}}}
\def\Go{G^{\diamond}}
\def\Uo{U^{\diamond}}
\def\Uom_#1{U^{\diamond,-}_{#1}}

\def\GQT{G}
\def\BQT{B}
\def\UQT{U}
\def\TQT{T}

\def\Fcalb{{\overline{\Fcal}}}

\def\Rsf_#1^#2{R_{#1,#2}^{\sf}}
\def\eval{\operatorname{eval}}

\def\Rtpv{\Rtp_{v''}^{[\ut,w_0]}}
\def\hjh{\hjmp_h}

\def\xrasim{\xrightarrow{\sim}}


\def\chr{\operatorname{char}}

\let \oldlabel \label
\let\oldref\ref
\let\oldcref\cref


\def\Lie{\operatorname{Lie}\;}
\def\U{{\mathcal U}}
\def\B{{\mathcal B}}
\def\Xaff{\accentset{\circ}{\Xcal}}
\def\Xaffcl{\Xcal}
\def\Richaff_#1^#2{\accentset{\circ}{\mathcal{R}}_{#1}^{#2}}
\def\G{\Gcal}
\def\Gmin{\G^{\minn}}
\def\Umin{\U^{\minn}}
\def\Bmin{\B^{\minn}}
\def\U{\Ucal}
\def\T{\Tcal}
\def\Caff{{\mathcal{C}}}
\def\GCMaff{\tilde{A}}

\def\bs{\backslash}
\def\Phiaff{\Delta}
\def\re{\operatorname{re}}
\def\im{\operatorname{im}}
\def\Phire{\Phiaff_{\re}}
\def\Phiim{\Phiaff_{\im}}
\def\Waff{{\tilde{W}}}

\def\FPS{{\mathcal{A}}}
\def\CRL{{\rootL^\vee}}
\def\lch{\l}
\let\oldtau\tau
\def\tau_#1{\oldtau_{#1}}
\def\dtau_#1{\dot{\oldtau}_{#1}}
\def\Iaff{\tilde{I}}
\def\minn{{\operatorname{min}}}
\def\Cast{\C^\ast}
\def\hfr{\mathfrak{h}}
\def\gfr{\mathfrak{g}}
\def\hfrast{\mathfrak{h}^\ast}
\def\rank{\operatorname{rank}}
\def\rt{\tilde{r}}
\def\dh{\dot{h}}
\def\dg{\dot{g}}
\def\df{\dot{f}}

\def\tl{\tau_{\lch}}
\def\dtl{\dtau_{\lch}}
\def\tul{\tau_{u\lch}}
\def\twl{\tau_{w\lch}}
\def\dtul{\dtau_{u\lch}}
\def\vtw{v\tau_\lch w^{-1}}
\def\wji{(\wj)^{-1}}
\def\tmin{\tau_\lch^J}

\def\Ical{\mathcal{I}}
\def\tauk{\tau_k}
\def\dtauk{\dtau_k}
\def\leqop{\leq^\op}
\def\Povar_#1{\accentset{\circ}{\Pi}_{#1}}
\def\Povarcl_#1{\Pi_{#1}}
\def\RPovar_#1{\accentset{\circ}{\Pi}^\R_{#1}}
\def\RPovarcl_#1{\Pi^\R_{#1}}
\def\Povtp_#1{\Pi_{#1}^{>0}}
\def\Povtnn_#1{\Pi_{#1}^{\geq0}}

\def\Boundkn{\operatorname{Bound}(k,n)}

\def\FSpro{\pi}
\def\FStra{\varrho}
\def\FSstr{\operatorname{str}}
\def\FSdil{\vartheta}
\def\FSiso{\bar\nu}

\def\Star_#1{\operatorname{Star}_{#1}}
\def\Startnn_#1{\operatorname{Star}^{\geq0}_{#1}}
\def\SCone{\operatorname{SCone}}

\def\FSstr{\operatorname{str}}
\def\FSdil{\vartheta}

\def\Ofg{{\O_{f,g}}}
\def\Og{\O_{g}}
\def\O{\mathcal{O}}

\def\BAcomb{\psi}
\def\BAgeom{\bar\varphi_u}
\def\BAgeo{\varphi_u}

\def\tht{t}

\def\Ng{{N_g}}
\def\dist{\operatorname{dist}_b}
\def\strix{t_1(x)}
\def\strixn{t_1(x_n)}
\def\Bcl{\overline{B}}
\def\pB{\partial B}
\def\epsb{\eps_B}
\def\Sepsb{S_{g}}

\def\conept{c}

\def\Link{\operatorname{Lk}}
\def\Lkx_#1{\Link_{#1}}
\def\Lkxx_#1^#2{\accentset{\circ}{\Link}_{#1}^{#2}}
\def\Lkg{\Lkx_g}
\def\Lkgh{\Lkxx_g^h}
\def\Lkghp{\Lkxx_g^{h'}}
\def\Lktxx_#1^#2{\Link^{>0}_{#1,#2}}
\def\Starxx_#1^#2{\operatorname{Star}_{#1,#2}}
\def\Startxx_#1^#2{\operatorname{Star}^{\geq0}_{#1,#2}}

\def\sctnn_#1{\sc^{\geq0}_{#1}}
\def\sctnng{\sctnn_g}
\def\sctnngh{\sctnn_{g,h}}
\def\sctp_#1^#2{\sc^{>0}_{#1,#2}}
\def\sctpgh{\sctp_g^h}

\def\eps{\varepsilon}
\def\Seps_#1{S_{#1}}

\def\Lktpe_#1^#2{\Link^{>0}_{#1,#2}}
\def\Lktnne_#1{\Link^{\geq0}_{#1}}
\def\Lktnnge{\Lktnne_g}
\def\Lktnnghe{\Lktnne_{g,h}}
\def\Lkge{\Lkx_g(\eps)}
\def\Lkghe{\Lkxx_g^h(\eps)}
\def\Lktnnghpe{\Lktnne_{g,h'}}

\def\Lktp_#1^#2{\Link^{>0}_{#1,#2}}
\def\Lktnn_#1{\Link^{\geq0}_{#1}}
\def\Lktnng{\Lktnn_g}
\def\Lktnngh{\Lktnn_{g,h}}

\def\axref#1{\ref{#1}}
\def\sc{Z}
\def\scg{\sc_g}
\def\sch{\sc_h}
\def\schp{\sc_{h'}}
\def\sco_#1^#2{\accentset{\circ}{\sc}_{#1,#2}}
\def\scogh{\sco_g^h}
\def\sccl_#1^#2{\sc_{#1}^{#2}}
\def\base{0}
\def\baseg{0}
\def\baseh{0}
\def\basehp{0}
\def\sphere{S_{\base}}

\def\Y{\mathcal{Y}}
\def\Yo_#1{\accentset{\circ}{\Y}_{#1}}
\def\Ycl_#1{\Y_{#1}}
\def\Ytnn{\Y^{\geq0}}
\def\Ytp_#1{\Y_{#1}^{>0}}
\def\YQ{Q}
\def\leqY{\leqJ}
\def\lessY{\lessJ}
\def\gessY{\gessJ}
\def\geqY{\geqJ}
\def\gradY{\dim}
\def\YQhat{\widehat{\YQ}}
\def\FSisox#1{\FSiso^{#1}}
\def\YQmin{\YQ_\minn}

\def\dilg{\FSdil_g}
\def\strg(#1){\normg{#1}}
\def\isog{\FSiso_g}
\def\isoga{\FSiso_{g,1}}
\def\isogb{\FSiso_{g,2}}
\def\isoh{\FSiso_h}
\def\isoha{\FSiso_{h,1}}
\def\isohb{\FSiso_{h,2}}

\def\norm#1{\|#1\|}
\def\normg#1{\|#1\|}

\def\ssan{4}
\def\ssak{2}
\def\ssau{s_3s_2}
\def\ssauk{\{1,4\}}
\def\ssax{\smat{
1 &   &   &   \\
x_{1} & x_{2} & -1 &   \\
x_{3} & x_{4} &   & -1 \\
  & 1 &   &  
}}
\def\ssaGRx{\smat{
1 &   \\
x_{1} & x_{2} \\
x_{3} & x_{4} \\
  & 1
}}
\def\ssaGRgjAu{\smat{
1 &   \\
  & x_{2} \\
  & x_{4} \\
  & 1
}}
\def\ssaGRgjBu{\smat{
1 &   \\
x_{1} &   \\
x_{3} &   \\
  & 1
}}
\def\ssagjA{\smat{
1 &   &   &   \\
  & 1 &   & x_{2} \\
  &   & 1 & x_{4} \\
  &   &   & 1
}}
\def\ssagjB{\smat{
1 &   &   &   \\
x_{1} & 1 &   &   \\
x_{3} &   & 1 &   \\
  &   &   & 1
}}
\def\ssadtul{\smat{
\frac{1}{z} &   &   &   \\
  & 1 &   &   \\
  &   & 1 &   \\
  &   &   & \frac{1}{z}
}}
\def\ssagjBi{\smat{
1 &   &   &   \\
-x_{1} & 1 &   &   \\
-x_{3} &   & 1 &   \\
  &   &   & 1
}}
\def\ssabageom{\smat{
\frac{1}{z} &   &   &   \\
-x_{1} & 1 &   & \frac{x_{2}}{z} \\
-x_{3} &   & 1 & \frac{x_{4}}{z} \\
  &   &   & \frac{1}{z}
}}
\def\ssaBm{\smat{
1 &   & \frac{x_{2}}{{\left(x_{1} x_{4} - x_{2} x_{3}\right)} z} & \frac{x_{4}}{x_{3} z} \\
  & 1 &   &   \\
  & \frac{x_{4}}{x_{2}} & 1 &   \\
  & \frac{1}{x_{2}} & \frac{x_{1}}{x_{1} x_{4} - x_{2} x_{3}} & 1
}}
\def\ssatauk{\smat{
  &   & \frac{1}{z} &   \\
  &   &   & \frac{1}{z} \\
1 &   &   &   \\
  & 1 &   &  
}}
\def\ssaBp{\smat{
\frac{x_{1} x_{4} - x_{2} x_{3}}{x_{2}} & -\frac{x_{4}}{x_{2}} & 1 &   \\
  & \frac{x_{3}}{x_{1} x_{4} - x_{2} x_{3}} & -\frac{x_{1}}{x_{1} x_{4} - x_{2} x_{3}} &   \\
  &   & \frac{1}{x_{3}} &   \\
-x_{1} z & z &   & x_{2}
}}
\def\ssaBmx{\smat{
1 & -\frac{1}{x_{1} z} &   & \frac{x_{4}}{x_{3} z} \\
  & 1 &   &   \\
  &   & 1 &   \\
  & -\frac{x_{3}}{x_{1} x_{4}} & \frac{1}{x_{4}} & 1
}}
\def\ssahx{\smat{
  &   & \frac{1}{z} &   \\
1 &   &   &   \\
  &   &   & \frac{1}{z} \\
  & 1 &   &  
}}
\def\ssaBpx{\smat{
-x_{1} & 1 &   &   \\
  & \frac{x_{3}}{x_{1} x_{4}} & -\frac{1}{x_{4}} &   \\
  &   & \frac{1}{x_{3}} &   \\
-x_{3} z &   & z & x_{4}
}}
\def\ssaBmxx{\smat{
1 & -\frac{1}{x_{1} z} & -\frac{1}{x_{3} z} &   \\
  & 1 &   &   \\
  & \frac{x_{3}}{x_{1}} & 1 &   \\
  &   &   & 1
}}
\def\ssahxx{\smat{
  &   & \frac{1}{z} &   \\
1 &   &   &   \\
  & 1 &   &   \\
  &   &   & \frac{1}{z}
}}
\def\ssaBpxx{\smat{
-x_{1} & 1 &   &   \\
  & -\frac{x_{3}}{x_{1}} & 1 &   \\
  &   & \frac{1}{x_{3}} &   \\
  &   &   & 1
}}
\def\ssagy{\smat{
-x_{1} & 1 &   & \frac{x_{2}}{z} \\
  &   &   & \frac{1}{z} \\
1 &   &   &   \\
-x_{3} z &   & z & x_{4}
}}
\def\ssaBmg{\smat{
1 &   &   & \frac{x_{2}}{x_{4} z} \\
-\frac{x_{3}}{x_{1} x_{4} - x_{2} x_{3}} & 1 &   & \frac{1}{x_{4} z} \\
-\frac{x_{4}}{x_{1} x_{4} - x_{2} x_{3}} & \frac{x_{4}}{x_{3}} & 1 &   \\
  &   &   & 1
}}
\def\ssahg{\smat{
1 &   &   &   \\
  & 1 &   &   \\
  &   & 1 &   \\
  &   &   & 1
}}
\def\ssaBpg{\smat{
-\frac{x_{1} x_{4} - x_{2} x_{3}}{x_{4}} & 1 & -\frac{x_{2}}{x_{4}} &   \\
  & \frac{x_{3}}{x_{1} x_{4} - x_{2} x_{3}} & -\frac{x_{1}}{x_{1} x_{4} - x_{2} x_{3}} &   \\
  &   & \frac{1}{x_{3}} &   \\
-x_{3} z &   & z & x_{4}
}}
\def\ssatrMA{\smat{
1 &   \\
  & x_{2} \\
  & x_{4} \\
  & 1
}}
\def\ssatrMB{\smat{
x_{1} & x_{2} \\
x_{3} & x_{4} \\
  & 1 \\
-1 &  
}}
\def\ssatrMC{\smat{
x_{3} & x_{4} \\
  & 1 \\
-1 &   \\
  &  
}}
\def\ssatrMD{\smat{
  & 1 \\
-1 &   \\
  &   \\
  &  
}}
\def\ssatrminA{x_{4}}
\def\ssatrminB{x_{1} x_{4} - x_{2} x_{3}}
\def\ssatrminC{x_{3}}
\def\ssatrminD{1}

\def\ssagUigi{\smat{
1 & -\frac{x_{4}}{{\left(x_{1} x_{4} - x_{2} x_{3}\right)} z} &   & \frac{x_{4}}{x_{3} z} \\
  & 1 & \frac{x_{2}}{x_{4}} &   \\
  &   & 1 &   \\
  & -\frac{x_{3}}{x_{1} x_{4} - x_{2} x_{3}} & \frac{1}{x_{4}} & 1
}}
\def\ssayA{\smat{
1 &   &   &   \\
  & 1 & -\frac{x_{2}}{x_{4}} &   \\
  &   & 1 &   \\
  &   &   & 1
}}
\def\ssayAi{\smat{
1 &   &   &   \\
  & 1 & \frac{x_{2}}{x_{4}} &   \\
  &   & 1 &   \\
  &   &   & 1
}}
\def\ssaxB{\smat{
1 & -\frac{x_{4}}{{\left(x_{1} x_{4} - x_{2} x_{3}\right)} z} &   & \frac{x_{4}}{x_{3} z} \\
  & 1 &   &   \\
  &   & 1 &   \\
  & -\frac{x_{3}}{x_{1} x_{4} - x_{2} x_{3}} & \frac{1}{x_{4}} & 1
}}

\def\ssayAbageom{\smat{
\frac{1}{z} &   &   &   \\
-\frac{x_{1} x_{4} - x_{2} x_{3}}{x_{4}} & 1 & -\frac{x_{2}}{x_{4}} &   \\
-x_{3} &   & 1 & \frac{x_{4}}{z} \\
  &   &   & \frac{1}{z}
}}
\def\ssayAbageomB{\smat{
\frac{1}{z} &   &   &   \\
-\frac{x_{1} x_{4} - x_{2} x_{3}}{x_{4}} & 1 &   &   \\
-x_{3} &   & 1 & \frac{x_{4}}{z} \\
  &   &   & \frac{1}{z}
}}
\def\ssapigx{\smat{
1 &   \\
\frac{x_{1} x_{4} - x_{2} x_{3}}{x_{4}} &   \\
x_{3} & x_{4} \\
  & 1
}}

\def\ssavikappax{\smat{
1 &   &   &   \\
  & x_{4} &   & -1 \\
  & -x_{2} & 1 &   \\
  & 1 &   &  
}}
\def\ssaetaL{\smat{
1 &   &   &   \\
  & 1 &   &   \\
  & -\frac{x_{2}}{x_{4}} & 1 &   \\
  & \frac{1}{x_{4}} &   & 1
}}
\def\ssaeta{\smat{
1 &   &   &   \\
  & x_{4} &   &   \\
  &   & 1 & -\frac{x_{2}}{x_{4}} \\
  &   &   & \frac{1}{x_{4}}
}}
\def\ssaetaR{\smat{
1 &   &   &   \\
  & 1 &   & -\frac{1}{x_{4}} \\
  &   & 1 &   \\
  &   &   & 1
}}

\def\ssazeta{\smat{
1 &   &   &   \\
x_{1} & \frac{x_{2}}{x_{4}} & -1 & -x_{2} \\
x_{3} & 1 &   & -x_{4} \\
  & \frac{1}{x_{4}} &   &  
}}

\def\ssazetaw{\smat{
  &   & 1 &   \\
-1 & -x_{2} & x_{1} & \frac{x_{2}}{x_{4}} \\
  & -x_{4} & x_{3} & 1 \\
  &   &   & \frac{1}{x_{4}}
}}

\def\ssazetawA{\frac{1}{x_{4}}}
\def\ssazetawB{\frac{x_{3}}{x_{4}}}
\def\ssazetawC{\frac{x_{1} x_{4} - x_{2} x_{3}}{x_{4}}}

\def\ssayBbageomA{\smat{
1 &   &   &   \\
  & 1 & \frac{x_{2}}{x_{4}} &   \\
  &   & 1 &   \\
  &   &   & 1
}}
\def\ssayBbageomAt{\smat{
1 &   &   &   \\
  & 1 & \frac{t x_{2}}{x_{4}} &   \\
  &   & 1 &   \\
  &   &   & 1
}}

\def\ssaGRxt{\smat{
1 &   \\
x_{1}+(t-1)\frac{x_2x_3}{x_4} & tx_{2} \\
x_{3} & x_{4} \\
  & 1
}}


\def\ssbn{5}
\def\ssbk{2}
\def\ssbu{s_2}
\def\ssbuk{\{1,3\}}
\def\ssbwp{s_2s_1s_4s_3s_2}
\def\ssbvp{s_1}
\def\ssbwpk{\{3,5\}}
\def\ssbw{s_2s_1s_4s_3s_2}
\def\ssbv{s_2s_1}
\def\ssbwk{\{3,5\}}
\def\ssbx{\smat{
  & -1 &   &   &   \\
1 &   &   &   &   \\
t_{1} & t_{5} & 1 &   &   \\
  & t_{4} t_{5} & t_{4} & 1 &   \\
  & t_{3} t_{4} t_{5} & t_{3} t_{4} & t_{3} & 1
}}
\def\ssbdeovpwp{y_2(t_{1})\ds_{1}y_4(t_{3})y_3(t_{4})y_2(t_{5})}
\def\ssbGRx{\smat{
  & -1 \\
1 &   \\
t_{1} & t_{5} \\
  & t_{4} t_{5} \\
  & t_{3} t_{4} t_{5}
}}
\def\ssbgju{\smat{
1 &   &   &   &   \\
\frac{t_{5}}{t_{1}} & \frac{1}{t_{1}} & -1 &   &   \\
  & 1 &   &   &   \\
-t_{4} t_{5} &   &   & 1 &   \\
-t_{3} t_{4} t_{5} &   &   &   & 1
}}
\def\ssbGRgju{\smat{
1 &   \\
\frac{t_{5}}{t_{1}} & \frac{1}{t_{1}} \\
  & 1 \\
-t_{4} t_{5} &   \\
-t_{3} t_{4} t_{5} &  
}}
\def\ssbhjx{\smat{
1 &   &   &   &   \\
-\frac{t_{5}}{t_{1}} & 1 &   &   &   \\
  &   & 1 &   &   \\
t_{4} t_{5} &   &   & 1 &   \\
t_{3} t_{4} t_{5} &   &   &   & 1
}}
\def\ssbhjxx{\smat{
  & -1 &   &   &   \\
1 & \frac{t_{5}}{t_{1}} &   &   &   \\
t_{1} & t_{5} & 1 &   &   \\
  &   & t_{4} & 1 &   \\
  &   & t_{3} t_{4} & t_{3} & 1
}}
\def\ssbg{[2, 4, 8, 5, 6]}
\def\ssbh{[3, 4, 7, 5, 6]}
\def\ssbgneck{[\{1,3\},\{2,3\},\{3,4\},\{4,8\},\{5,8\}]}
\def\ssbhneck{[\{1,2\},\{2,3\},\{3,4\},\{4,7\},\{5,7\}]}
\def\ssbpidup{\smat{
  & -1 &   &   &   \\
1 &   & -\frac{1}{t_{1}} &   &   \\
t_{1} & t_{5} &   &   &   \\
  & t_{4} t_{5} & t_{4} & 1 &   \\
  & t_{3} t_{4} t_{5} & t_{3} t_{4} & t_{3} & 1
}}
\def\ssbeta{\smat{
t_{1} & t_{5} &   &   &   \\
  & 1 &   &   &   \\
  &   & \frac{1}{t_{1}} &   &   \\
  &   & t_{4} & 1 &   \\
  &   & t_{3} t_{4} & t_{3} & 1
}}
\def\ssbzeta{\smat{
  & -1 &   &   &   \\
\frac{1}{t_{1}} & -\frac{t_{5}}{t_{1}} & -1 &   &   \\
1 &   &   &   &   \\
  & t_{4} t_{5} &   & 1 &   \\
  & t_{3} t_{4} t_{5} &   &   & 1
}}
\def\ssbetaB{\smat{
  & 1 &   &   &   \\
-1 &   &   &   &   \\
  &   & 1 &   &   \\
  &   &   & 1 &   \\
  &   &   &   & 1
}}
\def\ssbetaBi{\smat{
  & -1 &   &   &   \\
1 &   &   &   &   \\
  &   & 1 &   &   \\
  &   &   & 1 &   \\
  &   &   &   & 1
}}
\def\ssbzetaB{\smat{
  & -1 &   &   &   \\
\frac{1}{t_{1}} & -\frac{t_{5}}{t_{1}} & -1 &   &   \\
1 &   &   &   &   \\
  & t_{4} t_{5} &   & 1 &   \\
  & t_{3} t_{4} t_{5} &   &   & 1
}}
\def\ssbzetaw{\smat{
  &   &   &   & -1 \\
-1 &   & \frac{1}{t_{1}} &   & -\frac{t_{5}}{t_{1}} \\
  &   & 1 &   &   \\
  & 1 &   &   & t_{4} t_{5} \\
  &   &   & -1 & t_{3} t_{4} t_{5}
}}
\def\ssbzetawA{ t_{3} t_{4} t_{5} }
\def\ssbzetawB{ t_{4} t_{5} }
\def\ssbzetawC{ t_{4} t_{5} }
\def\ssbzetawD{ \frac{t_{5}}{t_{1}} }
\def\ssbzetawE{ 1 }

\def\ssbgAneck{13}
\def\ssbtrminA{ 1 }
\def\ssbgBneck{23}
\def\ssbtrminB{ \frac{t_{5}}{t_{1}} }
\def\ssbgCneck{34}
\def\ssbtrminC{ t_{4} t_{5} }
\def\ssbgDneck{48}
\def\ssbtrminD{ t_{4} t_{5} }
\def\ssbgEneck{58}
\def\ssbtrminE{ t_{3} t_{4} t_{5} }

\def\notblack{black}
\def\black{black}
\def\notblue{blue}
\def\blue{blue}
      \def\lw{1.2pt}
      \def\lblscl{0.67}

\def\sqlw{0.4pt}

\newcommand\crosssq[5]{
\pgfmathsetmacro{\xR}{#1+1}
\pgfmathsetmacro{\xL}{#1-1}
\pgfmathsetmacro{\yU}{#2+1}
\pgfmathsetmacro{\yD}{#2-1}
\pgfmathsetmacro{\xl}{#1-0.5}
\pgfmathsetmacro{\xr}{#1+0.5}
\pgfmathsetmacro{\yu}{#2+0.5}
\pgfmathsetmacro{\yd}{#2-0.5}
\draw[line width=\sqlw,draw=#3] (\xR,#2)--(#1,\yU)--(\xL,#2)--(#1,\yD)--cycle;
\draw[line width=\lw,draw=#4] (\xl,\yu)--(\xr,\yd);
\draw[line width=\lw,draw=#4] (\xr,\yu)--(\xl,\yd);
\coordinate (LU#5) at (\xl,\yu);
\coordinate (LD#5) at (\xl,\yd);
\coordinate (RU#5) at (\xr,\yu);
\coordinate (RD#5) at (\xr,\yd);
}

\def\ledot#1#2{\draw[fill=black!30] (#1,#2) circle (4pt);}

\newcommand\uncrssq[5]{
\pgfmathsetmacro{\xR}{#1+1}
\pgfmathsetmacro{\xL}{#1-1}
\pgfmathsetmacro{\yU}{#2+1}
\pgfmathsetmacro{\yD}{#2-1}
\pgfmathsetmacro{\xl}{#1-0.5}
\pgfmathsetmacro{\xr}{#1+0.5}
\pgfmathsetmacro{\yu}{#2+0.5}
\pgfmathsetmacro{\yd}{#2-0.5}
\draw[line width=\sqlw,draw=#3] (\xR,#2)--(#1,\yU)--(\xL,#2)--(#1,\yD)--cycle;
\draw[line width=\lw,draw=#4] (\xl,\yu) to[out=-45, in=45] (\xl,\yd);
\draw[line width=\lw,draw=#4] (\xr,\yu) to[out=-135,in=135] (\xr,\yd);
\ledot{#1}{#2}
\coordinate (LU#5) at (\xl,\yu);
\coordinate (LD#5) at (\xl,\yd);
\coordinate (RU#5) at (\xr,\yu);
\coordinate (RD#5) at (\xr,\yd);
}

    \newcommand\drawLe[3]{
\pgfmathtruncatemacro{\xshift}{(#1)*5}
\begin{scope}[xshift=\xshift cm]
  \foreach \i in {1,2,...,5}
  {
\pgfmathtruncatemacro{\xi}{\i+(#1)*5}

    \node[draw=#2,  scale=0.3,fill=#2] (u\xi) at (\i,6) {};
    \node[draw=#2,  scale=0.3,fill=#2] (mm\xi) at (\i,1) {};
    \node[draw=#2,  scale=0.3,fill=#2] (m\xi) at (\i,-1) {};
    \node[label=below:{\scalebox{\lblscl}{$\xi$}},draw=#2,  scale=0.3,fill=#2] (d\xi) at (\i,-6) {};
    \coordinate (ux\xi) at (\i,7.28);
    \node[anchor=north] (XXX) at (ux\xi.90) {\scalebox{\lblscl}{$\xi$}};
  }
  \foreach \i in {3,4,5}
  {
\pgfmathtruncatemacro{\xi}{\i+(#1)*5}
\draw[line width=\lw,draw=#3] (mm\xi) -- (m\xi);
}
\pgfmathtruncatemacro{\one}{1+(#1)*5}
\pgfmathtruncatemacro{\two}{2+(#1)*5}
\pgfmathtruncatemacro{\three}{3+(#1)*5}
\pgfmathtruncatemacro{\four}{4+(#1)*5}
\pgfmathtruncatemacro{\five}{5+(#1)*5}
\crosssq{1.5}{3.5}{#2}{#3}{A}
\crosssq{2.5}{2.5}{#2}{#3}{B}
\crosssq{2.5}{4.5}{#2}{#3}{C}
\crosssq{3.5}{3.5}{#2}{#3}{D}
\crosssq{4.5}{4.5}{#2}{#3}{E}
\draw[line width=\lw,draw=#3] (u\one)--(LUA);
\draw[line width=\lw,draw=#3] (u\two)--(LUC);
\draw[line width=\lw,draw=#3] (u\three)--(RUC);
\draw[line width=\lw,draw=#3] (u\four)--(LUE);
\draw[line width=\lw,draw=#3] (u\five)--(RUE);

\draw[line width=\lw,draw=#3] (mm\one)--(LDA);
\draw[line width=\lw,draw=#3] (mm\two)--(LDB);
\draw[line width=\lw,draw=#3] (mm\three)--(RDB);
\draw[line width=\lw,draw=#3] (mm\four)--(RDD);
\draw[line width=\lw,draw=#3] (mm\five)--(RDE);

\crosssq{1.5}{-3.5}{#2}{#3}{A}
\uncrssq{2.5}{-2.5}{#2}{#3}{B}
\uncrssq{2.5}{-4.5}{#2}{#3}{C}
\uncrssq{3.5}{-3.5}{#2}{#3}{D}
\uncrssq{4.5}{-4.5}{#2}{#3}{E}

\draw[line width=\lw,draw=#3] (d\one)--(LDA);
\draw[line width=\lw,draw=#3] (d\two)--(LDC);
\draw[line width=\lw,draw=#3] (d\three)--(RDC);
\draw[line width=\lw,draw=#3] (d\four)--(LDE);
\draw[line width=\lw,draw=#3] (d\five)--(RDE);

\draw[line width=\lw,draw=#3] (m\one)--(LUA);
\draw[line width=\lw,draw=#3] (m\two)--(LUB);
\draw[line width=\lw,draw=#3] (m\three)--(RUB);
\draw[line width=\lw,draw=#3] (m\four)--(RUD);
\draw[line width=\lw,draw=#3] (m\five)--(RUE);

\end{scope}

}

\newcommand\drawLine[2]{
\pgfmathtruncatemacro{\xshift}{(#1)*5}
\begin{scope}[xshift=\xshift cm]
  \foreach \i in {1,2}
  {
\pgfmathtruncatemacro{\xi}{\i+(#1)*5}
\pgfmathtruncatemacro{\xxi}{\i+(#1+1)*5}
\draw[line width=\lw,draw=#2] (mm\xi) -- (m\xxi);
  }
\end{scope}
}

\def\sqt{0.5}

\newcommand\Lesq[3]{
\pgfmathsetmacro{\xl}{#1-\sqt}
\pgfmathsetmacro{\xr}{#1+\sqt}
\pgfmathsetmacro{\yu}{#2+\sqt}
\pgfmathsetmacro{\yd}{#2-\sqt}
\draw[line width=\sqlw,draw=#3] (\xl,\yu)--(\xr,\yu)--(\xr,\yd)--(\xl,\yd)--cycle;
}

\newcommand\ledott[2]{\draw[fill=black!30] (#1,#2) circle (2.83pt);}

\newcommand\Lesqd[3]{
\pgfmathsetmacro{\xl}{#1-\sqt}
\pgfmathsetmacro{\xr}{#1+\sqt}
\pgfmathsetmacro{\yu}{#2+\sqt}
\pgfmathsetmacro{\yd}{#2-\sqt}
\draw[line width=\sqlw,draw=#3] (\xl,\yu)--(\xr,\yu)--(\xr,\yd)--(\xl,\yd)--cycle;
\draw[fill=black!30] (#1,#2) circle (2.83pt);
}
\newcommand\Lesql[4]{
\Lesq{#1}{#2}{#3}
\node[scale=1] (SSS) at (#1,#2) {$s_{#4}$};
}

\def\figureLe{

\begin{tabular}{ccc}

\begin{tikzpicture}[scale=0.565,baseline=(zero.base)]
\coordinate (zero) at (0,0.5);
\Lesqd{0}{-2}{\black}
\Lesqd{1}{-2}{\black}
\Lesqd{2}{-2}{\black}
\Lesq{0}{-3}{\black}
\Lesqd{1}{-3}{\black}

\Lesql{0}{3}{\black}{2}
\Lesql{1}{3}{\black}{3}
\Lesql{2}{3}{\black}{4}
\Lesql{0}{2}{\black}{1}
\Lesql{1}{2}{\black}{2}
\end{tikzpicture}
& \qquad\qquad\qquad
&
\begin{tikzpicture}[xscale=0.4,yscale=0.4,baseline=(zero.base)]
\coordinate (zero) at (0,0);
   \begin{scope}[even odd rule]
  \clip (-1.4,-7) rectangle (12.4,7);
\drawLe 0{\black}{\blue}
\drawLe 1{\notblack}{\notblue}
\drawLe 2{\notblack}{\notblue}
\drawLe 3{\notblack}{\notblue}
\drawLe{-1}{\notblack}{\notblue}
\drawLine{-1}{\notblue}
\drawLine{0}{\blue}
\drawLine{1}{\notblue}
\drawLine{2}{\notblue}
      \end{scope}
\node[anchor=west,scale=1.2] (A) at (12.6,3.5) {$w^{-1}$};
\node[anchor=west,scale=1.2] (B) at (12.6,0) {$\tl$};
\node[anchor=west,scale=1.2] (C) at (12.6,-3.5) {$v$};
\node[anchor=west,scale=1.1] (D1) at (12.13,6.62) {$\cdots$};
\node[anchor=east,scale=1.1] (D1) at (-1.13,6.62) {$\cdots$};
\node[anchor=west,scale=1.1] (D1) at (12.13,-6.78) {$\cdots$};
\node[anchor=east,scale=1.1] (D1) at (-1.13,-6.78) {$\cdots$};
\end{tikzpicture}
\end{tabular}

}

\def\figureFlag{
\def\lw{1.2pt}
\def\lww{1.6pt}
\def\lws{0.6pt}
\def\pointscl{0.3}
\def\pointscll{0.3}
\def\redx{red}
\setlength{\tabcolsep}{12pt}
\begin{tabular}{ccc}
\begin{tikzpicture}[scale=2,baseline=(zero.base)]
\coordinate (zero) at (0,0);
\coordinate (ss2) at (-30:1);
\coordinate (ss21) at (150:1);
\draw[dashed, line width=\lw,opacity=0.7] (ss2) to[bend right=-20] (ss21);
\shade [ball color=white,opacity=0.4] (0,0) circle [radius=1];
\draw[line width=\lw] (0,0) circle (1);
\draw[draw=\redx,line width=\lww] (-150:1) arc (-150:-210:1);
\node[label=below:$\id$,draw,circle,fill,scale=\pointscl] (id) at (-90:1) {};
\node[label=-150:$s_1$,draw,circle,fill,scale=\pointscl] (s1) at (-150:1) {};
\node[label=-30:$s_2$,draw,circle,fill,scale=\pointscl] (s2) at (-30:1) {};
\node[label=30:$s_1s_2$,draw,circle,fill,scale=\pointscl] (s12) at (30:1) {};
\node[label=150:$s_2s_1$,draw,circle,fill,scale=\pointscl] (s21) at (150:1) {};
\node[label=above:$w_0$,draw,circle,fill,scale=\pointscl] (w0) at (90:1) {};
\draw[line width=\lw] (s1) to[bend right=20] (s12);
\node[anchor=east,color=\redx] (g) at (180:1) {$\Povtp_g$};
\end{tikzpicture}
  & \scalebox{1.3}{$\overset{\isog}{\longrightarrow}$} &
\begin{tikzpicture}[xscale=2.5,yscale=2,baseline=(zero.base)]
\coordinate (zero) at (0,0.5);
\coordinate (w0) at (-0.2,1);
\coordinate (s2) at (0.2,1.3);
\coordinate (s21) at (0.8,1);
\coordinate (id) at (1.2,1.3);
\coordinate (ss1) at (1,0);
\coordinate (ss12) at (0,0);

\fill[opacity=0.1] (w0)--(s2)--(id)--(ss1)--(ss12)--(w0);
\fill[opacity=0.2] (s2) -- (id)--(ss1)--(ss12)--(s2);
\fill[opacity=0.3] (w0) -- (s21)--(ss1)--(ss12)--(w0);

\node[label=below:$s_1$,draw,circle,fill=white,scale=\pointscll] (s1) at (1,0) {};
\node[label=below:$s_2s_1$,draw,circle,fill=white,scale=\pointscll] (s12) at (0,0) {};
\draw[line width=\lww,draw=\redx] (s1)--(s12);
\node[anchor=north,color=\redx] (g) at (0.5,0) {$\Povtp_g$};
\draw[dashed,line width=\lw,opacity=0.7] (s12) -- (w0);
\draw[dashed,line width=\lw,opacity=0.7] (s12) -- (s2);
\draw[dashed,line width=\lw,opacity=0.7] (s1) -- (s21);
\draw[dashed,line width=\lw,opacity=0.7] (s1) -- (id);
\draw[dashed,line width=\lws,opacity=0.2] (w0) -- (s2)--(id)--(s21)--(w0);

\node[anchor=south] (A) at (w0) {$w_0$};
\node[anchor=south] (A) at (s2) {$s_2$};
\node[anchor=south] (A) at (id) {$\id$};
\node[anchor=south] (A) at (s21) {$s_1s_2$};
\end{tikzpicture}

\end{tabular}
\setlength{\tabcolsep}{6pt}
}

\section{Introduction}\label{sec:introduction}

Let $G$ be a semisimple algebraic group, split over $\R$, and let $P\subset G$ be a parabolic subgroup.  Lusztig~\cite{LusGB} introduced the totally nonnegative part of the partial flag variety $G/P$, denoted $\GPtnn$, which he called a ``remarkable polyhedral subspace''. He conjectured and Rietsch proved~\cite{Rie99} that $\GPtnn$ has a decomposition into open cells. We prove the following conjecture of Williams~\cite{Wil}:

\begin{theorem}\label{thm:main_intro}
The cell decomposition of $\GPtnn$ forms a regular CW complex. Thus the closure
 of each cell is homeomorphic to a closed ball.
\end{theorem}

A special case of particular interest is when $G/P$ is the Grassmannian $\Gr(k,n)$ of $k$-dimensional linear subspaces of $\R^n$. In this case, $\GPtnn$ becomes the \emph{totally nonnegative Grassmannian} $\Grtnn(k,n)$, introduced by  Postnikov~\cite{Pos} as the subset of $\Gr(k,n)$ where all Pl\"ucker coordinates are nonnegative. He gave a stratification of $\Grtnn(k,n)$ into \emph{positroid cells} according to which Pl\"ucker coordinates are zero and which are strictly positive, and conjectured that the closure of each positroid cell is homeomorphic to a closed ball. Postnikov's conjecture follows as a special case of \cref{thm:main_intro}:
\begin{corollary}\label{cor:positroids_intro}
The decomposition of $\Grtnn(k,n)$ into positroid cells forms a regular CW complex. Thus the closure of each positroid cell is homeomorphic to a closed ball.
\end{corollary}

When $k = 1$, $\Grtnn(1,n)$ is the standard $(n-1)$-dimensional simplex $\Delta_{n-1} \subset \P^{n-1}$. Simplices, and more generally convex polytopes, are prototypical examples of regular CW complexes. While the spaces $\GPtnn$ and $\Grtnn(k,n)$ are not themselves homeomorphic to polytopes, our results confirm that they have the simplest possible topology.

\subsection{History and motivation}
A matrix is called \emph{totally nonnegative} if all its minors are nonnegative. The theory of such matrices originated in the 1930's~\cite{Schoenberg,GK37}. Later, Lusztig~\cite{LusGB} was motivated by a question of Kostant to consider connections between totally nonnegative matrices and his theory of canonical bases for quantum groups \cite{LusCan}. This led him to introduce the totally nonnegative part $G_{\ge 0}$ of a split semisimple $G$.  Inspired by a result of Whitney~\cite{Whitney}, he defined $G_{\ge 0}$ to be generated by exponentiated Chevalley generators with positive real parameters, and generalized many classical results for $G = \SL_n$ to this setting. He introduced the totally nonnegative part $\GPtnn$ of a partial flag variety $G/P$, and showed~\cite[\S 4]{Lus_intro} that $G_{\geq0}$ and $\GPtnn$ are contractible.

Fomin and Shapiro~\cite{FS} realized that Lusztig's work may be used to address a longstanding problem in poset topology. Namely, the Bruhat order of the Weyl group $W$ of $G$ had been shown to be shellable by Bj\"orner and Wachs~\cite{BW}, and by general results of Bj\"orner~\cite{Bjo2} it followed that there exists a ``synthetic'' regular CW complex whose face poset coincides with $(W,\leq)$. The motivation of~\cite{FS} was to answer a natural question due to Bernstein and Bj\"orner of whether such a regular CW complex exists ``in nature''. Let $U \subset G$ be the unipotent radical of the standard Borel subgroup, and let $U_{\ge 0} := U\cap G_{\geq0}$ be its totally nonnegative part.  For $G = \SL_n$, $U_{\ge 0}$ is the semigroup of upper-triangular unipotent matrices with all minors nonnegative. The work of Lusztig~\cite{LusGB} implies that $U_{\ge 0}$ has a cell decomposition whose face poset is $(W, \le)$. The space $U_{\ge 0}$ is not compact, but Fomin and Shapiro~\cite{FS} conjectured that taking the link of the identity element in $U_{\ge 0}$, which also has $(W, \le)$ as its face poset, gives the desired regular CW complex. Their conjecture was confirmed by Hersh~\cite{Her}.  Hersh's theorem also follows as a corollary to our proof of \cref{thm:main_intro}; see \cref{rmk:U}. 
\begin{corollary}[{\cite{Her}}] \label{cor:Her}
The link of the identity in $U_{\geq0}$ is a regular CW complex.
\end{corollary}
\noindent
For recent related developments, see~\cite{DHM}.

Meanwhile, Postnikov~\cite{Pos} defined the totally nonnegative Grassmannian $\Grtnn(k,n)$, decomposed it into positroid cells, and showed that each positroid cell is homeomorphic to an open ball. Motivated by work of Fomin and Zelevinsky~\cite{FZ} on double Bruhat cells, he conjectured~\cite[Conjecture 3.6]{Pos} that this decomposition forms a regular CW complex.  It was later realized (see~\eqref{eq:Gr_Grtnn_minors}) that the space $\Gr_{\ge 0}(k,n)$ and its cell decomposition coincide with the one studied by Lusztig and Rietsch in the special case that $G/P = \Gr(k,n)$. Williams~\cite[\S 7]{Wil} extended Postnikov's conjecture from $\Grtnn(k,n)$ to $\GPtnn$.

There has been much progress towards proving these conjectures. Williams~\cite{Wil} showed that the face poset of $(G/P)_{\ge 0}$ (and hence of $\Gr_{\ge 0}(k,n)$) is graded, thin, and shellable, and therefore by~\cite{Bjo2} is the face poset of some regular CW complex. Postnikov, Speyer, and Williams~\cite{PSW} showed that $\Gr_{\ge 0}(k,n)$ is a CW complex, and their result was generalized to $(G/P)_{\ge 0}$ by Rietsch and Williams~\cite{RW08}. Rietsch and Williams~\cite{RW10} also showed that the closure of each cell in $(G/P)_{\ge 0}$ is contractible.  
In previous work~\cite{GKL,GKL2}, we showed that the spaces $\Gr_{\ge 0}(k,n)$ and $(G/P)_{\ge 0}$ are homeomorphic to closed balls,  which is the special case of \cref{thm:main_intro} for the top-dimensional cell of $\GPtnn$.  We remark that our proof of \cref{thm:main_intro} uses different methods than those employed in~\cite{GKL,GKL2}, in which we relied on the existence of a vector field on $G/P$ contracting $\GPtnn$ to a point in its interior. Singularities of lower-dimensional positroid cells give obstructions to the existence of a continuous vector field with analogous properties.

The topology of a regular CW complex is completely determined by the combinatorial structure of its associated cell closure poset, as observed by Bj\"{o}rner \cite{Bjo2}. Therefore one may regard spaces such as $\Utnn$ and $\Grtnn(k,n)$ as canonical topological realizations of natural posets arising in combinatorics. We expect this phenomenon to hold more broadly for other spaces appearing in total positivity, as we discuss in \cref{sec:further}.



Totally positive spaces have also attracted a lot of interest due to their appearances in other contexts such as cluster algebras~\cite{FZclust} and the physics of scattering amplitudes~\cite{ABCGPT}. Our original motivation for studying the topology of spaces arising in total positivity was to better understand the \emph{amplituhedra} of Arkani-Hamed and Trnka~\cite{AT}, and more generally the Grassmann polytopes of the third author~\cite{Lam16}. Faces of these geometric objects are linear projections of closures of positroid cells, and we expect that \cref{cor:positroids_intro} will play an essential role in developing a theory of Grassmann polytopes.


\subsection{Stars, links, and the Fomin--Shapiro atlas}
Rietsch~\cite{Rie99,Rie} defined a certain poset $(Q_J,\leqJ)$, and established the decomposition $\GPtnn=\bigsqcup_{g\in Q_J} \Povtp_g$ into open balls $\Povtp_g$ indexed by $g \in Q_J$.  She showed that for $h\in Q_J$, the closure $\Povtnn_h$ of $\Povtp_h$ is given by $\Povtnn_h=\bigsqcup_{g\leqJ h} \Povtp_h$. When $\GPtnn$ is the totally nonnegative Grassmannian $\Grtnn(k,n)$, this is the positroid cell decomposition of~\cite{Pos}. 

Given $g\in Q_J$, define the \emph{star} of $g$ in $\GPtnn$ by 
\begin{equation}\label{eq:star_tnn_intro}
  \Startnn_g:=\bigsqcup_{h\geqY g} \Povtp_h.
\end{equation}
In \cref{sec:links}, we define another space $\Lktnng$ (the \emph{link} of $g$) stratified as $\Lktnng=\bigsqcup_{h\gessJ g}\Lktp_g^h$. We later show in \cref{thm:Lk_ball} that $\Lktnng$ is a regular CW complex homeomorphic to a closed ball.

At the core of our approach is a collection of (stratification-preserving) homeomorphisms
\begin{equation}\label{eq:isog_intro}
\isog: \Startnn_g\xrasim \Povtp_g\times \Cone(\Lktnng),
\end{equation}
one for each $g \in Q_J$.  Here $\Cone(A):=(A\times \R_{\geq0})/(A\times \{0\})$ denotes the \emph{open cone} over $A$.  The homeomorphisms $\{ \isog \mid g \in Q_J\}$ are part of the data of what we call a \emph{Fomin--Shapiro atlas}; cf.~\cref{dfn:FS_atlas}.  Our construction is inspired by similar maps introduced in~\cite{FS} for the unipotent radical $U_{\geq 0}$.

\begin{figure}
\figureFlag
  \caption{\label{fig:flag} The map $\isog$ for the case $G=\SL_3$ and $P=B$ from \cref{ex:flag_intro}.}
\end{figure}

\def\PRSteven_#1^#2{\Povtp_{(#1,#2)}}
\begin{example}\label{ex:flag_intro}
When $G=\SL_n$ and $P=B$ is the standard Borel subgroup, $G/B$ is the \emph{complete flag variety} consisting of flags in $\R^n$, and the Weyl group $W$ is the group $S_n$ of permutations of $n$ elements. The face poset $Q_J$ of $\GBtnn$ is the set  $\{(v,w)\in S_n\times S_n\mid v\leq w\}$ of Bruhat intervals in $S_n$, and the cell $\PRSteven_v^w\subset \GBtnn$ indexed by $(v,w)\in Q_J$ has dimension $\ell(w) - \ell(v)$. For example, when $n=3$, this gives a cell decomposition of a $3$-dimensional ball; see Figure~\ref{fig:flag} (left). For $g:=(s_1,s_2s_1)$, $\Povtp_g$ is an open line segment, and $\Startnn_g$ consists of $4$ cells: a line segment $\Povtp_g=\PRSteven_{s_1}^{s_2s_1}$, two open square faces $\PRSteven_{s_1}^{w_0}$ and $\PRSteven_{\id}^{s_2s_1}$, and an open $3$-dimensional ball $\PRSteven_{\id}^{w_0}$. This union is indeed homeomorphic to $\Povtp_g\times \Cone(\Lktnng)$ shown in Figure~\ref{fig:flag} (right). Here $\Lktnng$ is a closed line segment whose endpoints are $\Lktp_g^{(s_1,w_0)}$ and $\Lktp_g^{(\id,s_2s_1)}$, and whose interior is $\Lktp_g^{(\id,w_0)}$.
\end{example}

In \cref{dfn:TNNspace}, we introduce the abstract notion of a \emph{(shellable) totally nonnegative space}, which captures several known combinatorial and geometric properties of $\GPtnn$ used in our proof. This includes the shellability of $Q_J$ due to Williams~\cite{Wil}, and some topological results~\cite{Rie,KLS} on Richardson varieties.

In \cref{sec:topological-results}, we prove (\cref{thm:FS_intro}) that every shellable totally nonnegative space that admits a Fomin--Shapiro atlas is a regular CW complex.  Our argument proceeds by induction on the dimension of $\Lktp_g^h$, and depends on a delicate interplay between objects in smooth and topological categories.  We use crucially that the maps \eqref{eq:isog_intro} in a Fomin--Shapiro atlas are restrictions of smooth maps.  On the topological level, we rely on the generalized Poincar\'e conjecture~\cite{Sma,Fre,Per1} combined with some general results on poset topology.

The bulk of the paper is devoted to the construction of the Fomin--Shapiro atlas.  For each $g\in Q_J$ we give an isomorphism $\BAgeom$ between an open dense subset $\Og\subset G/P$ and a certain subset of the \emph{affine flag variety} $\G/\B$ of the loop group $\G$ associated to $G$. The map $\BAgeom$, which we call an \emph{affine Bruhat atlas}, sends the projected Richardson stratification~\cite{KLS} of $G/P$ to the affine Richardson stratification of its image inside $\G/\B$. The hardest part of the proof consists of showing that the subset $\Og\subset G/P$ contains $\Startnn_g$. See \cref{sec:plan-proof} for a more in-depth overview of the construction of $\BAgeom$.

\begin{remark}\label{rmk:HKL}
The map $\BAgeom$ generalizes the map of Snider~\cite{Sni} from $\Gr(k,n)$ to all $G/P$; see \cref{rmk:Snider}. A different approach to give such a  generalization is due to He, Knutson, and Lu~\cite{HKL}, which led them to the notion of a \emph{Bruhat atlas}. See~\cite{Elek} for the definition. We call our map $\BAgeom$ an affine Bruhat atlas since its target space is always an affine flag variety, while the Bruhat atlases of~\cite{HKL} necessarily involve more general \emph{Kac--Moody flag varieties}. A similar map has been independently constructed by Huang~\cite{Huang}.
\end{remark}

\begin{remark}
The method of link induction that we use in \cref{sec:link_induction} has appeared before in e.g.~\cite{GLMS,Hersh_poinc}. When applied to the problem at hand, this method immediately runs into the difficulty of showing that the closure of each cell is a topological manifold. Our strategy for overcoming this issue is based on combining technical topological results in \cref{sec:topological-results} with the approach of~\cite{FS}. The crucial new algebraic ingredient is that the factorizations of~\cite{FS} happen inside the unipotent group $U$, while we utilize an embedding into the affine flag variety for that purpose. This embedding is defined on an open dense subset of $G/P$, but surprisingly, this subset turns out to contain the whole totally nonnegative part of the star of the corresponding cell. In order to show this result, we develop a toolbox of \emph{subtraction-free parametrizations} in \cref{sec:subtraction-free}. This machinery also reveals intriguing properties of $\GPtnn$ such as \cref{prop:truncation}, which may be interesting to explore further in their own right.
\end{remark}

\subsection{Outline}
In \cref{sec:overview}, we introduce   totally nonnegative spaces and define Fomin--Shapiro atlases. We state in \cref{thm:FS_intro} that every shellable totally nonnegative space that admits a Fomin--Shapiro atlas is a regular CW complex, and prove it in \cref{sec:topological-results}. We give background on $G/P$ in \cref{sec:gp:-basic}, and study subtraction-free Marsh--Rietsch parametrizations in \cref{sec:subtraction-free}. We then apply our results on such parametrizations to prove \cref{thm:zeta}, that will later imply that the above open subset $\Og$ contains $\Startnn_g$. We introduce affine Bruhat atlases in \cref{sec:kac-moody-groups} and use them to construct a Fomin--Shapiro atlas for $G/P$ in \cref{sec:BA_implies_FS}. \Cref{thm:main_intro2} (which implies our main result \cref{thm:main_intro}) is proved in \cref{sec:FS_proof}. \Cref{sec:type_A} is devoted to specializing our construction to type $A$ (when $G=\SL_n$), with a special focus on the totally nonnegative Grassmannian $\Grtnn(k,n)$. We illustrate many of our constructions by examples in  \cref{sec:type_A}, and we encourage the reader to consult this section while studying other parts of the paper. We discuss some conjectures and further directions in \cref{sec:further}. Finally, we give additional background on Kac--Moody flag varieties in \cref{sec:KM}.

\subsection*{Acknowledgments} We thank Sergey Fomin, Patricia Hersh, Alex Postnikov, and Lauren Williams for stimulating discussions. We are also grateful to George Lusztig and Konni Rietsch for their comments on the first version of this manuscript. We thank the anonymous referees for their help with improving the presentation of the paper.

\section{Overview of the proof}\label{sec:overview}
We formulate our results in the abstract language of \emph{totally nonnegative spaces}, since we expect that they can be applied in other contexts; see~\cref{sec:further}.

\subsection{Totally nonnegative spaces}\label{sec:totally-nonn-spac}
We refer the reader to \cref{top:background} for background on posets and regular CW complexes. For a finite poset $(\YQ,\leqY)$, we denote by $\YQhat:=\YQ\sqcup \{\hat0\}$ the poset obtained from $\YQ$ by adjoining a minimum $\hat0$. Bj\"orner showed~\cite[Proposition~4.5(a)]{Bjo2} that if $\YQhat$ is \emph{graded}, \emph{thin}, and \emph{shellable}, then $\YQ$ is isomorphic to the face poset of some regular CW complex. If $\YQhat$ is a graded poset, we let $\gradY:\YQ\to \Z_{\geq0}$ denote the rank function of $Q$.

\begin{definition}\label{dfn:TNNspace}
We say that a triple $(\Y,\Ytnn,\YQ)$ is a \emph{totally nonnegative space} (or \emph{TNN space} for short) if the following conditions are satisfied.
\begin{enumerate}[label={(TNN\arabic*)},leftmargin=*]
\item\label{TNN:graded_max} The poset $(\YQhat,\leqY)$ is graded and contains a unique maximal element $\hat1$.
\item\label{TNN:smooth} $\Y$ is a smooth manifold, stratified into embedded submanifolds $\Y=\bigsqcup_{g\in \YQ} \Yo_g$, and for each $h\in \YQ$,  $\Yo_h$ has dimension $\gradY(h)$ and closure $\Ycl_h:=\bigsqcup_{g\leqY h } \Yo_g$.
\item\label{TNN:Ytnn_compact} $\Ytnn$ is a compact subset of $\Y$.
\item \label{TNN:Ytp} For $g\in\YQ$, $\Ytp_g:=\Yo_g\cap \Ytnn$ is a connected  component of $\Yo_g$ diffeomorphic to~$\R_{>0}^{\gradY(g)}$.
\item \label{TNN:Ytp_cl} The closure of $\Ytp_h$ inside $\Y$ equals $\Ytnn_h:=\bigsqcup_{g\leqY h} \Ytp_g$.
\end{enumerate}
We say that a TNN space $(\Y,\Ytnn,\YQ)$ is \emph{shellable} if it additionally satisfies the following.
\begin{enumerate}[label={(TNN\arabic*')},leftmargin=*]
\item\label{TNN:thin_shell} The poset $(\YQhat,\leqY)$ is thin and shellable.
\end{enumerate}
\end{definition}

\noindent For the case $\Y=G/P$, the smooth submanifolds $\Yo_g$ are the \emph{open projected Richardson varieties} of~\cite{KLS}. 

\begin{definition}\label{dfn:sc}
Let $N\geq0$, and denote by $\norm\cdot$ the Euclidean norm on $\R^N$.  We say that a pair $(\sc,\FSdil)$ is a \emph{smooth cone} if $\sc\subset \R^N$ is a closed embedded submanifold and  $\FSdil:\R_{>0}\times\R^N\to\R^N$ a smooth map such that
\begin{enumerate}[label={\textup{(SC\arabic*)}},leftmargin=*]
\item\label{SC:R_+_action}\label{SC:base_lim} $\FSdil$ gives an $(\R_{>0},\cdot)$-action on $\R^N$ that restricts to an $(\R_{>0},\cdot)$-action on $\sc$.
\item\label{SC:str_derivative} $\frac{\partial}{\partial t} \norm{\FSdil(t,x)}>0$ for all $t\in\R_{>0}$ and $x\in\R^N\setminus\{\base\}$.
\end{enumerate}
\end{definition}
\noindent The map $\FSdil$ is a smooth analog of a \emph{contractive flow} of~\cite{GKL}; see~\cref{lemma:cflow}.

For $g\in \YQ$, define $\Star_g:=\bigsqcup_{h\geqY g} \Yo_h$ and $\Startnn_g := \Star_g \cap \Ytnn = \bigsqcup_{h\geqY g} \Ytp_h$; cf.~\eqref{eq:star_tnn_intro}.

\begin{definition}\label{dfn:FS_atlas}
We say that a TNN space $(\Y,\Ytnn,\YQ)$ admits a \emph{Fomin--Shapiro atlas} if for each $g\in \YQ$, there exists an open subset $\Og\subset \Star_g$, a smooth cone $(\scg,\dilg)$, and a diffeomorphism
\begin{equation}\label{eq:FSiso_intro}
\isog:\Og\xrasim (\Yo_g\cap \Og)\times \scg
\end{equation}
satisfying the following conditions.
\begin{enumerate}[label={(FS\arabic*)},leftmargin=*]
\item\label{FS:mnf_strat} For all $h\geqY g$, we are given $\sco_g^h\subset \scg$ such that $\scg=\bigsqcup_{h\geqY g} \sco_g^h$ and $\sco_g^g=\{\baseg\}$.
\item\label{FS:mnf_dil} For all $h\geqY g$ and $t\in\R_{>0}$, we have $\dilg(t,\sco_g^h)=\sco_g^h$.
\item\label{FS:image_of_Yh} For all $h\geqY g$, we have $\isog(\Yo_h\cap\Og)=(\Yo_g\cap \Og)\times \sco_g^h$. 
\item\label{FS:Yg_id} For all $y\in \Yo_g\cap \Og$, we have   $\isog(y)=(y,\baseg)$.
\item\label{FS:Startnn} $\Startnn_g\subset \Og$.
\end{enumerate}
\end{definition}
We will prove the following result in~\cref{sec:link_induction}, using \emph{link induction}.
\begin{theorem}\label{thm:FS_intro}
Suppose that $(\Y,\Ytnn,\YQ)$ is a shellable TNN space that admits a Fomin--Shapiro atlas. Then $\Ytnn=\bigsqcup_{h\in \YQ} \Ytp_h$ is a regular CW complex. In particular, for each $h\in \YQ$, $\Ytnn_h$ is homeomorphic to a closed ball of dimension $\gradY(h)$.
\end{theorem}
\noindent Thus  \cref{thm:main_intro} follows as a corollary of \cref{thm:FS_intro} and the following result:
\begin{theorem}\label{thm:main_intro2}
  $(G/P,\GPtnn,Q_J)$ is a shellable TNN space that admits a Fomin--Shapiro atlas.
\end{theorem}

\def\faff{{\tilde{f}}}
\def\gaff{{\tilde{g}}}
\def\haff{{\tilde{h}}}
\def\fbac{{\BAcomb(f)}}
\def\gbac{{\BAcomb(g)}}
\def\hbac{{\BAcomb(h)}}

\subsection{Plan of the proof}\label{sec:plan-proof}
We give a brief outline of the proof of \cref{thm:main_intro2}. See \cref{sec:gp:-basic} for background on $G/P$, and see \cref{sec:kac-moody-groups,sec:KM} for background on $\G/\B$. We deduce that $(G/P,\GPtnn,Q_J)$ is a shellable TNN space from known results in \cref{cor:GP_TNN_space}. In order to construct a Fomin--Shapiro atlas, we consider the (infinite-dimensional) \emph{affine flag variety} $\G/\B$ associated to $G$. It is stratified into (finite-dimensional) affine Richardson varieties $\G/\B=\bigsqcup_{\haff\leq \faff\in \Waff}\Richaff_\haff^\faff$, where $\Waff$ is the affine Weyl group and $\leq$ denotes its Bruhat order. There exists an order-reversing injective map $\BAcomb:Q_J\to \Waff$, defined in~\cite{HL}; see~\eqref{eq:BA_comb}. The set of minimal elements of $Q_J$ equals $\{(u,u)\mid u\in W^J\}$, where $W^J$ is the set of minimal length parabolic coset representatives of the Weyl group; see \cref{sec:parab-subgr-w_j}. For each minimal element $f:=(u,u)\in Q_J$, $\BAcomb$ identifies the interval $[f,\hat1]$ of $Q_J$ with (the dual of) a certain interval $[\tmin,\tul]\subset \Waff$. For the case $G/P=\Gr(k,n)$, elements of $Q_J$ are in bijection with \emph{\Le -diagrams} of~\cite{Pos}, and $\BAcomb$ sends a \Le-diagram indexing a positroid cell to the corresponding \emph{bounded affine permutation} of~\cite{KLS}; see \cref{ex:Le}.

In \cref{sec:bruh-atlas-proj}, we lift $\BAcomb$ to the geometric level: given a minimal element $f:=(u,u)\in Q_J$, we introduce a map $\BAgeom:\Cuj\to \G/\B$ defined on an open dense subset $\Cuj\subset G/P$. We show in \cref{thm:BA} that for $g\in Q_J$ such that $g\geqJ f$, $\BAgeom$ sends $\Cuj\cap \Povar_g$ isomorphically to the affine Richardson cell $\Richaff_{\gbac}^{\fbac}$.

For every $\gaff\in \Waff$, we consider an open dense subset $\Caff_{\gaff}\subset \G/\B$ defined by $\Caff_\gaff:=\gaff\cdot \B_-\cdot \B/\B$, as well as affine Schubert and opposite Schubert cells $\Xaff^\gaff=\bigsqcup_{\haff\leq\gaff}\Richaff_\haff^\gaff$, $\Xaff_\gaff=\bigsqcup_{\gaff\leq\faff}\Richaff_\gaff^\faff$. In \cref{prop:FS_iso}, we give a natural isomorphism
\begin{equation}\label{eq:intro_KWY_aff}
\Caff_\gaff\xrasim \Xaff_\gaff\times\Xaff^\gaff,\quad\text{which restricts to}\quad (\Caff_\gaff\cap \Richaff_\haff^\faff)\xrasim  \Richaff_\gaff^\faff\times \Richaff_\haff^\gaff\quad \text{for all $\haff\leq \gaff\leq \faff$.}
\end{equation}
A finite-dimensional analog of this map is due to~\cite{KWY}, and similar maps have been considered in~\cite{KL,FS}. The action of $\FSdil$ on $\Xaff^\gaff$ essentially amounts to multiplying by an element of the affine torus, and thus preserves $\Richaff_\haff^\gaff$ for all $\haff\leq \gaff$.

Let us now fix $g\in Q_J$, and choose some minimal element $f:=(u,u)\in Q_J$ such that $f\leqJ g$. Then the map $\BAgeom$ is defined on an open dense subset $\Cuj\subset G/P$, and let us denote by $\Og\subset \Cuj$ the preimage of $\Caff_\gbac$ under $\BAgeom$. The diffeomorphism~\eqref{eq:FSiso_intro} is obtained by conjugating  the isomorphism~\eqref{eq:intro_KWY_aff} by  the map $\BAgeom$. The smooth cone $(\scg,\dilg)$ is extracted from the corresponding structure on $\Xaff^\gbac$. As we have already mentioned, the hardest step in the proof consists of showing~\axref{FS:Startnn}. To achieve this, we study subtraction-free parametrizations of partial flag varieties in \cref{sec:subtraction-free}, and then use them to show that some generalized minors of a particular group element $\zetamap(x)$ from \cref{sec:KWY_TP} do not vanish for all $x\in \Startnn_g$. The definition of $\zetamap(x)$ is quite technical, but we conjecture in \cref{sec:type_A} that in the Grassmannian case, these generalized minors specialize to simple functions on $\Gr(k,n)$ that we call \emph{$u$-truncated minors}. We complete the proof of \cref{thm:main_intro2} in \cref{sec:FS_proof}.

\section{Topological results}\label{sec:topological-results}
Throughout this section, we assume that $(\Y,\Ytnn,\YQ)$ is a TNN space that admits a Fomin--Shapiro atlas.  Thus for each $g\in \YQ$, we have the objects $\Og$, $\scg$, $\dilg$,  and $\isog$ from \cref{dfn:FS_atlas}.  Additionally, we assume some familiarity with basic theory of smooth manifolds; see e.g.~\cite{Lee}.

\subsection{Links}\label{sec:links}
Throughout, we denote the two components of the map $\isog$ from~\eqref{eq:FSiso_intro} by $\isog=(\isoga,\isogb)$, where $\isoga:\Og\to \Yo_g\cap \Og$ and $\isogb:\Og\to\scg$. We set $\Startnn_{g,h}:=\Ytnn_h\cap \Startnn_g=\bigsqcup_{g\leqY g'\leqY h}\Ytp_{g'}$. Let $\Ng$ be the integer from \cref{dfn:sc} such that $\scg\subset \R^\Ng$.

\begin{definition}
Let  $g\leqY h\in\YQ$. Denote
\begin{align*}
  \sctnng&:=\isogb \left(\Startnn_g\right),&\sctnngh&:=\isogb \left(\Startnn_{g,h}\right),  &\sctpgh&:=\sctnng\cap \sco_g^h,\\
\Seps_g&:=\{x\in \R^\Ng: \strg(x)=1\},&\Lktnnghe&:=\sctnngh\cap \Seps_g,&\Lktpe_g^h&:=\sctpgh\cap \Seps_g.
\end{align*}
\end{definition}
\noindent Note that by~\axref{FS:image_of_Yh}, we have
\begin{align}\label{eq:links_stratif}
\sctnngh = \bigsqcup_{g\leqY g'\leqY h}\sctp_g^{g'}, \qquad \Lktnnghe=\bigsqcup_{g\lessY g'\leqY h}\Lktpe_g^{g'}.
\end{align}
In the latter disjoint union, we have $\Lktpe_g^g=\emptyset$ since $\sco_g^g=\{\baseg\}$ by~\axref{FS:mnf_strat}.

\begin{lemma}\label{lemma:links} Let $g\lessY h\in\YQ$.
\begin{theoremlist}
\item\label{links:sctnn} For all $x\in \Og$, we have $x\in\Ytp_h$ if and only if  $\isog(x)\in \Ytp_g\times \sctpgh$.
\item\label{links:smooth} $\sctpgh$ is an embedded submanifold of $\scg$ of dimension $\dim(h) - \dim(g)$ that intersects $\Seps_g$  transversely. For all $t\in\R_{>0}$ and $x\in\sctpgh$, we have $\FSdil(t,x) \in \sctpgh$.
\item\label{links:contractible} $\Lktpe_g^h$ is a contractible smooth manifold of dimension $\gradY(h)-\gradY(g)-1$.
\item\label{links:compact} $\Lktnnghe$ is a compact subset of $\scg$.
\end{theoremlist}
\end{lemma}

Before we prove these properties, let us state some preliminary results on smooth manifolds. Given smooth manifolds $A,B$ and a smooth map $f:A\to B$, a point $a\in A$ is called a \emph{regular point} of $f$ if the differential of $f$ at $a$ is surjective. Similarly, $b\in B$ is called a \emph{regular value} of $f$ if $f^{-1}(b)$ consists of regular points. In this case $f^{-1}(b)$ is a closed embedded submanifold of $A$ of dimension $\dim(A)-\dim(B)$~\cite[Corollary 5.14]{Lee}.

\begin{lemma}\label{lemma:smooth_manifolds_embedded}
Suppose that $A,B$ are smooth manifolds and $B'\subset B$ is such that $A\times B'$ is an embedded submanifold of $A\times B$. Then $B'$ is an embedded submanifold of $B$.
\end{lemma}
\begin{proof}
Choose $a\in A$.  Clearly $a$ is a regular value of the projection $A\times B'\to A$, so $\{a\}\times B'$ is an embedded submanifold of $A\times B'$, and hence of $\{a\}\times B$.
\end{proof}

We also recall some facts about $\FSdil$ from \cite{GKL}.
\begin{lemma}\label{lemma:cflow}
Let $\FSdil:\R_{>0}\times\R^N\to\R^N$ be a smooth map satisfying~\axref{SC:base_lim} and~\axref{SC:str_derivative}.
\begin{theoremlist}
\item\label{cflow:0} We have $\FSdil(t,0) = 0$ for all $t\in\R_{>0}$.
\item\label{cflow:limit} We have $\lim_{t\to 0+}\FSdil(t,x) = 0$ for all $x\in\R^N$.
\item\label{cflow:t1} For all $x\in\R^N\setminus\{0\}$, there exists a unique $t\in\R_{>0}$ such that $\|\FSdil(t,x)\| = 1$, which we denote by $\strix$. The function $t_1 : \R^N\setminus\{0\}\to\R_{>0}$ is continuous.
\end{theoremlist}
\end{lemma}

\begin{proof}
The function $f : \R \times \R^N \to \R^N$ defined by $f(t,x) = \FSdil(e^{-t},x)$ is a contractive flow, as defined in~\cite[Definition~2.1]{GKL}. Therefore the statements follow from~\cite[Lemma~2.2]{GKL} and the claim in the proof of \cite[Lemma~2.3]{GKL}.
\end{proof}

\begin{proof}[Proof of \cref{lemma:links}.]
\itemref{links:sctnn}: We prove this more generally for $g\leqY h$. The set $\Startnn_g$ is connected since it contains a connected dense subset $\Ytp_{\hat1}$.  Therefore $\isoga(\Startnn_g)$ is a connected subset of $\Yo_g\cap \Og$. By~\axref{FS:Yg_id}, it contains $\Ytp_g$, and therefore $\isoga(\Startnn_g)=\Ytp_g$ by~\axref{TNN:Ytp}. By definition, $\isogb(\Startnn_{g,h})=\sctnngh$, and thus $\isog(\Startnn_{g,h})\subset \Ytp_g\times \sctnngh$. By~\axref{FS:image_of_Yh}, we get $\isog(\Ytp_h)\subset \Ytp_g\times \sctpgh$. In particular, $\sctpgh=\isogb(\Ytp_h)$ is a connected subset of $\sco_g^h$. Let $C$ be the connected component of $\sco_g^h$ containing $\sctpgh$. By~\axref{FS:image_of_Yh}, $\isog^{-1}(\Ytp_g\times C)$ is a connected subset of $\Yo_h\cap\Og$, which contains $\Ytp_h$ as we have just shown. Therefore we must have $\isog^{-1}(\Ytp_g\times C)=\Ytp_h$ by~\axref{TNN:Ytp}, which shows that $\sctpgh=C$ is a connected component of $\sco_g^h$. Thus indeed $\isog(\Ytp_h)=\Ytp_g\times \sctpgh$. 

\itemref{links:smooth}: By~\axref{TNN:Ytp} and~\axref{TNN:smooth}, $\Ytp_h$ is an embedded submanifold of $\Y$. Applying $\isog$ and using~\itemref{links:sctnn}, we get that $\Ytp_g\times\sctpgh$ is an embedded submanifold of $\Ytp_g\times\scg$, of dimension $\dim(h) - \dim(g)$. By \cref{lemma:smooth_manifolds_embedded}, $\sctpgh$ is an embedded submanifold of $\scg$. Moreover, it follows from~\axref{FS:mnf_dil} that $\dilg(t,\sctpgh)=\sctpgh$ for all $t\in\R_{>0}$, since $\sctpgh$ is a connected component of $\sco_g^h$. Thus $1$ is a regular value of the restriction $\strg(\cdot):\sctpgh\to \R_{>0}$, so the manifolds $\Seps_g$ and $\sctpgh$ intersect transversely inside~$\R^\Ng$.

\itemref{links:contractible}: By~\itemref{links:smooth}, $\Lktpe_g^h=\sctpgh\cap \Seps_g$ is an embedded submanifold of $\scg$ of dimension $\gradY(h)-\gradY(g)-1$. To show that it is contractible, we use the fact that a retract of a contractible space is contractible~\cite[Exercise~0.9]{Hat}. Since $\Ytp_h$ is contractible (by~\axref{TNN:Ytp}), so is $\isog(\Ytp_h)=\Ytp_g\times \sctpgh$. Then $\{x\}\times \sctpgh$ is a retract of $\Ytp_g\times \sctpgh$ for any $x\in\Ytp_g$, so $\sctpgh$ is contractible. Finally, by~\itemref{links:smooth} and~\cref{cflow:t1}, the map $x\mapsto \dilg(\strix,x)$ gives a retraction $\sctpgh\to \Lktpe_g^h$.

\itemref{links:compact}: By~\axref{FS:Startnn}, $\Startnn_{g,h}=\Ytnn_h\cap \Startnn_g=\Ytnn_h\cap \Og$ is a closed subset of $\Og$. Thus $\isog(\Startnn_{g,h})$ is a closed subset of $\Ytp_g\times \scg$. Since $\isog(\Startnn_{g,h})=\Ytp_g\times \sctnngh$ (by~\itemref{links:sctnn} and~\eqref{eq:links_stratif}), we get that $\sctnngh$ is a closed subset of $\scg$. It follows that $\Lktnnghe=\sctnngh\cap \Seps_g$ is a closed and bounded subset of $\scg$, which is closed in $\R^\Ng$ by \cref{dfn:sc}.
\end{proof}

Recall that $\Cone(A):=(A\times \R_{\geq0})/(A\times \{0\})$ is the open cone over $A$.  We denote by $\conept:=(\ast,0)\in\Cone(A)$ its \emph{cone point}.
\begin{proposition} Let $g\lessY h\in\YQ$.
\begin{theoremlist}
\item\label{links:cone} We have a homeomorphism $\sctnngh\xrasim \Cone(\Lktnnghe)$ sending $\baseg$ to the cone point~$\conept$, and sending $\sctp_g^{g'}$ to $\Lktpe_g^{g'}\times \R_{>0}$ for all $g\lessY g'\leqY h$.
\item\label{links:TNN_homeo} We have a homeomorphism $\Startnn_{g,h}\xrasim \Ytp_g\times \Cone(\Lktnnghe)$ sending $\Ytp_g$ to $\Ytp_g \times \{\conept\}$.
\end{theoremlist}
\end{proposition}
\begin{proof}

\def\conehomeo{\xi}
\itemref{links:cone}: Define a map $\conehomeo:\sctnngh\to \Cone(\Lktnnghe)$ sending $\baseg$ to $\conept$ and  $x$ to $\left(\dilg(\strix,x),\frac1{\strix}\right)$ for $x\in \sctnngh\setminus\{\baseg\}$, where $\strix$ is defined in~\cref{cflow:t1} and $\dilg(\strix,x)\in\Lktnnghe$ by~\cref{links:smooth}. We claim that $\conehomeo$ is a homeomorphism. Note that $\conehomeo$ has an inverse $\conehomeo^{-1}$, which sends $\conept$ to $\baseg$ and $(y,t)$ to $\dilg(t,y)$ for $(y,t)\in \Cone(\Lktnnghe)\setminus\{\conept\}= \Lktnnghe\times \R_{>0}$. By~\cref{cflow:t1}, $\conehomeo$ is continuous on $\sctnngh\setminus \{\baseg\}$ and $\conehomeo^{-1}$ is continuous on $\Lktnnghe\times \R_{>0}$. It remains to show that $\conehomeo$ is continuous at $0$ and that $\conehomeo^{-1}$ is continuous at $c$.

Suppose that $(x_n)_{n\geq0}$ is a sequence in $\sctnngh\setminus \{\baseg\}$ converging to $\baseg$. We claim that $\strixn\to\infty$ as $n\to\infty$. Otherwise, after passing to a subsequence, we may assume that there exists $R\in\R_{>0}$ such that $\strixn\leq R$ for all $n\geq0$. Then~\axref{SC:str_derivative} implies that $\normg{\dilg(R,x_n)}\geq \normg{\dilg(\strixn,x_n)}=1$ for all $n\geq0$. Taking $n\to\infty$ gives $\normg{\dilg(R,0)}\geq 1$, contradicting~\cref{cflow:0}. This shows that $\conehomeo$ is continuous at $\baseg$.

Suppose now that $((y_n,t_n))_{n\geq0}$ is a sequence in $\Lktnnghe\times \R_{>0}$ converging to $\conept$, i.e., $t_n\to 0$. The function $D(t):=\max_{x\in\Seps_g} \normg{\dilg(t,x)}$ is increasing in $t$, by compactness of $\Seps_g$ and~\axref{SC:str_derivative}.  We have $\lim_{t\to0+}D(t)=0$ by~\cref{cflow:limit} and compactness of $\Seps_g$ (more precisely, by \emph{Dini's theorem}). Therefore $\conehomeo^{-1}(y_n,t_n)=\dilg(t_n,y_n)$ converges to $0$ as $n\to\infty$, showing that $\conehomeo^{-1}$ is continuous at $\conept$.

\itemref{links:TNN_homeo}: By \cref{links:sctnn}, $\isog$ restricts to a homeomorphism $\Startnn_{g,h}\xrasim \Ytp_g\times \sctnngh$, which by~\axref{FS:Yg_id} sends $\Ytp_g$ to $\Ytp_g\times\{0\}$. The result follows from~\itemref{links:cone}.
\end{proof}

Our next aim is to analyze the local structure of the space $\Lktnnghe$. For two topological spaces $A$ and $B$ and $a\in A$, $b\in B$, a \emph{local homeomorphism between $(A,a)$ and $(B,b)$} is a homeomorphism from an open neighborhood of $a$ in $A$ to an open neighborhood of $b$ in $B$ which sends $a$ to $b$.

\def\hh{h}
\def\gg{p}
\def\scgg{\sc_\gg}
\def\basegg{0}
\def\Ogg{\O_\gg}
\def\Hgg{H_\gg}
\def\sctpggg{\sctp_g^\gg}
\def\isogga{\FSiso_{\gg,1}}
\def\isoggb{\FSiso_{\gg,2}}
\def\Lktnnghhe{\Lktnne_{g,\hh}}

\begin{lemma}\label{lemma:links_local_homeo}
Let $g\lessY \gg\leqY \hh\in \YQ$, $x_\gg\in \Lktpe_g^\gg$, and set $d:=\gradY(\gg)-\gradY(g)-1$. Then there exists a local homeomorphism between $\left(\Lktnnghhe,x_\gg\right)$ and $\left(\sctnn_{\gg,\hh}\times \R^d,(\basegg,0)\right).$
\end{lemma}
\begin{proof}
\def\mapg{\theta_{g,\gg}}
\def\Diff#1{\operatorname{D}{#1}}
\def\Dmapg{\Diff{\mapg}}
\def\H{H_\gg}
\def\preim{F}
\def\mapdif{\beta}
\def\xgg{\bar x_{\gg}}
Choose some $x_g\in \Ytp_g$ and consider the open subset $\H\subset \scg$ defined by $\H:=\{x\in\scg\mid \isog^{-1}(x_g,x)\in \Ogg\}$. Introduce a map 
\begin{align*}
\mapg: \H\cap \Seps_g \to \scgg,\quad x\mapsto\isoggb (\isog^{-1}(x_g,x)).
\end{align*}
Since $x_\gg\in \Lktpe_g^\gg\subset \sctpggg$ and $x_g\in \Ytp_g$, we get $\xgg:=\isog^{-1}(x_g,x_\gg)\in \Ytp_\gg$ by \cref{links:sctnn}. By \axref{FS:Startnn}, we have $\Ytp_\gg\subset \Startnn_\gg\subset \Ogg$, and thus $x_\gg\in\H$. Since $\H$ is open in $\scg$, $\H\cap \Seps_g$ is an open subset of $\scg\cap \Seps_g$, which is nonempty since it contains $x_\gg$. We have $\mapg(x_\gg)=\basegg$ by~\axref{FS:Yg_id}.

We claim that $x_\gg$ is a regular point of $\mapg$. By~\axref{FS:Yg_id}, the differential of $\isoggb:\Ogg\to \scgg$ is surjective at $\xgg$, and its kernel is the tangent space of $\Yo_\gg$ at $\xgg$. By~\axref{TNN:Ytp} and~\axref{FS:Startnn}, $\Ytp_\gg$ is a connected component of $\Yo_\gg\cap \Ogg$, and it contains $\xgg=\isog^{-1}(x_g,x_\gg)$ as we have shown above. Therefore $x_\gg$ is a regular point of $\mapg$ if and only if the manifolds $\Ytp_\gg$ and $\preim:=\isog^{-1}\left(\{x_g\}\times (\H\cap \Seps_g)\right)$ intersect transversely at $\xgg$. By \cref{links:sctnn}, we have $\isog(\Ytp_\gg)=\Ytp_g\times \sctpggg$, and clearly $\isog(\preim)=\{x_g\}\times (\H\cap \Seps_g)$. These two manifolds intersect transversely at $(x_g,x_\gg)$ by \cref{links:smooth}. We have shown that $x_\gg$ is a regular point of $\mapg$.

By the \emph{submersion theorem} (see e.g.~\cite[Corollary~A(1.3)]{Kos}), there exist local coordinates centered at $x_\gg\in \H\cap \Seps_g$ and at $\basegg\in \scgg$ in which $\mapg$ is just the canonical projection $\R^{\dim(\H\cap\Seps_g)}\to\R^{\dim(\scgg)}$. Recall that $\YQ$ contains a unique  maximal element $\hat1$, and by~\eqref{eq:FSiso_intro} we have $\dim(\scg)=\codim(g):=\gradY(\hat1)-\gradY(g)$. Thus $\dim(\H\cap\Seps_g)=\codim(g)-1$, $\dim(\scgg)=\codim(\gg)$, and $\dim(\H\cap\Seps_g)-\dim(\scgg)=d$. We have shown that there exist open neighborhoods $U$ of $x_\gg$ in $\H\cap \Seps_g$ and $V$ of $\basegg$ in $\scgg$ and a diffeomorphism $\mapdif:U\xrasim V\times \R^d$ sending $x_\gg$ to $(\basegg,0)$ such that the first component of $\mapdif$ coincides with the restriction $\mapg:U\to V$.

\def\rr{r}
In order to complete the proof, we need to show that the image $\mapdif(U\cap \Lktnnghhe)$ equals $(V\cap \sctnn_{\gg,\hh})\times \R^d$. We may assume that $U$ is connected. Suppose we are given $x\in U$, and let $\rr\in \YQ$ be such that $x':=\isog^{-1}(x_g,x)\in \Yo_{\rr}$. Since $U\subset \H$, $x'$ belongs to $\Ogg\subset \Star_\gg$ by \cref{dfn:FS_atlas}, and therefore $\gg\leqY \rr$. By \cref{links:sctnn}, we have $x\in U\cap \Lktpe_g^{\rr}$ if and only if $x'\in \Ytp_{\rr}$. On the other hand, $\isogga(\isog^{-1}(\{x_g\}\times U))$ is a connected subset of $\Yo_\gg\cap \Ogg$ that contains $\isogga(\xgg)\in \Ytp_\gg$. Thus $\isogga(\isog^{-1}(x_g,U))\subset \Ytp_\gg$ by \axref{TNN:Ytp}. It follows that $x'\in \Ytp_{\rr}$ if and only if $\mapg(x)=\isoggb(x')$ belongs to $\sctp_\gg^{\rr}$. The result follows by taking the union over all $\gg\leqY \rr\leqY \hh$, using~\eqref{eq:links_stratif}.
\end{proof}

\subsection{Topological background}\label{top:background}

\subsubsection{Regular CW complexes}\label{top:regular_CW}

We refer to \cite{Hat,LW} for background on CW complexes.~
\begin{definition}\label{dfn:cw}
Let $X$ be a Hausdorff space. We call a finite disjoint union $X = \bigsqcup_{\alpha\in \YQ}X_\alpha$
a \emph{regular CW complex} if it satisfies the following two properties.
\begin{enumerate}[label={(CW\arabic*)},leftmargin=*]
\item\label{cond:cw:regularity} For each $\alpha\in \YQ$, there exists a homeomorphism from the closure $\overline{X_\alpha}$ to a closed ball $\Bcl$ which sends $X_\alpha$ to the interior of $\Bcl$.
\item\label{cond:cw:normality} For each $\alpha\in \YQ$, there exists $\YQ'\subset \YQ$ such that $\overline{X_\alpha} = \bigsqcup_{\beta\in \YQ'} X_\beta$.
\end{enumerate}
The \emph{face poset} of $X$ is the poset $(\YQ, \leqY)$, where $\beta \leqY \alpha$ if and only if $X_\beta \subset \overline{X_\alpha}$.
\end{definition}

The condition~\axref{cond:cw:normality} is often omitted from the definition of a regular CW complex, but is necessary in order to apply the arguments of~\cite{Bjo2}. We remark that the cell decomposition of $\Ytnn$ satisfies~\axref{cond:cw:normality} by~\axref{TNN:Ytp_cl}. 

\subsubsection{Posets}\label{top:posets}

We review the definitions of \emph{graded}, \emph{thin}, and \emph{shellable} for finite posets, though we will not need to work with them in our arguments. We refer to~\cite{Bjo1,Sta} for background.

A finite poset $(\YQ,\leqY)$ is called \emph{graded} if every maximal chain in $\YQ$ has the same length $\ell$, in which case we denote $\rank(\YQ):=\ell$. For $x\leq z\in\YQ$, we denote by $[x,z]:=\{y\in \YQ\mid x\leq y\leq z\}$ the corresponding \emph{interval}. Note that the intervals in a graded poset $\YQ$ are also graded, and we call $\YQ$ \emph{thin} if every interval of rank $2$ has exactly $4$ elements.

The \emph{order complex} of a graded poset $\YQ$ is the pure $(\rank(\YQ) - 1)$-dimensional simplicial complex whose vertices are the elements of $\YQ$, and whose faces are the chains in $\YQ$. We say that $\YQ$ is \emph{shellable} if its order complex is shellable, i.e., its maximal faces can be ordered as $F_1,\dots, F_n$ so that for $2\leq k\leq n$,  $F_k\cap \left(\bigcup_{1\leq i<k}F_i\right)$ is a nonempty union of $(\rank(\YQ)-2)$-dimensional faces of $F_k$.
\begin{proposition}[{\cite[Proposition~4.2]{Bjo1}}]\label{prop:shellable_int}
If a graded poset is shellable, then so are each of its intervals.
\end{proposition}
 
See \cite[Sections~2 and 3]{Bjo2} for the proof of the following result.
\begin{theorem}[{\cite{LW,DK,Bjo2}}]\label{thm:bjorner}
Suppose that $X$ is a regular CW complex with face poset $\YQ$. If $\YQ\sqcup \{\hat{0},\hat{1}\}$ (obtained by adjoining a minimum $\hat{0}$ and a maximum $\hat{1}$ to $Q$) is graded, thin, and shellable, then $X$ is homeomorphic to a sphere of dimension $\rank(\YQ)-1$.
\end{theorem}

\subsubsection{Poincar\'e conjecture}\label{sec:poincare}
Recall that an \emph{$n$-dimensional topological manifold with boundary} is a Hausdorff space $C$ such that every point $x\in C$ has an open neighborhood homeomorphic either to $\R^n$, or to $\R_{\geq0}\times \R^{n-1}$ via a homeomorphism which takes $x$ to a point in $\{0\} \times \R^{n-1}$. In the latter case, we say that $x$ belongs to the \emph{boundary} of $C$, denoted $\partial C$.

The following is a well-known consequence of the \emph{(generalized) Poincar\'e conjecture} due to Smale~\cite{Sma}, Freedman~\cite{Fre}, and Perelman~\cite{Per1}. We refer to~\cite[Theorem~10.3.3(ii)]{Dav} for this formulation.
\begin{theorem}[{\cite{Sma,Fre,Per1}}]\label{thm:Poincare}
Let $C$ be a compact contractible $n$-dimensional topological manifold with boundary, such that its boundary $\partial C$ is homeomorphic to an $(n-1)$-dimensional sphere. Then $C$ is homeomorphic to an $n$-dimensional closed ball.
\end{theorem}
\noindent For $n\geq6$, \cref{thm:Poincare} can be proved using the \emph{topological $h$-cobordism theorem}~\cite{Milnor,KS}. We sketch another standard argument for deducing \cref{thm:Poincare} from the Poincar\'e conjecture when $n$ is arbitrary. The boundary of $C$ is \emph{collared} by~\cite[Theorem~2]{Brown_Collar}, i.e., there exists a homeomorphism from an open neighborhood of $\partial C$ in $C$ to $\partial C \times [0,1)$, which takes $\partial C$ to $\partial C \times \{0\}$. Thus we can attach the (collared) boundary of an $n$-dimensional closed ball to the (collared) boundary of $C$, obtaining a topological manifold $C'$. By van Kampen's theorem, $C'$ is simply connected. It is easy to see from the Mayer–Vietoris sequence that $C'$ is a homology sphere. Thus $C'$ must be homeomorphic to a sphere by the Poincar\'e conjecture. Therefore $C$ is homeomorphic to a closed ball by Brown's Schoenflies theorem~\cite{Brown_Schoenflies}.

The following is also a consequence of Brown's collar theorem~\cite[Theorem~2]{Brown_Collar}.
\begin{proposition}\label{prop:Cinterior}
Suppose that $C$ is a topological manifold with boundary $\partial C$.  Then $C$ is homotopy equivalent to its interior $C\setminus \partial C$.
\end{proposition}

\subsection{Link induction}\label{sec:link_induction}
\begin{theorem}\label{thm:Lk_ball}
Suppose that $(\Y,\Ytnn,\YQ)$ is a shellable TNN space that admits a Fomin--Shapiro atlas, and let $g\lessY h\in\YQ$. Then $\Lktnnghe = \bigsqcup_{g\lessY g'\leqY h} \Lktpe_g^{g'}$ (cf.~\eqref{eq:links_stratif}) is a regular CW complex homeomorphic to a closed ball of dimension $\gradY(h)-\gradY(g)-1$.
\end{theorem}
\begin{proof}
We proceed by induction on $d := \gradY(h)-\gradY(g)-1$. For the base case $d=0$, we have $\Lktnnghe=\Lktpe_g^h$, which is a $0$-dimensional contractible manifold by \cref{links:contractible}. Thus $\Lktnnghe$ is a point, and we are done with the base case. Assume now that $d>0$ and that the result holds for all $d' < d$. We need to verify~\axref{cond:cw:regularity} and~\axref{cond:cw:normality} when $X_\alpha = \Lktpe_g^h$ (the other cases follow from the induction hypothesis).

We claim that $\Lktnnghe$ is a topological manifold with boundary $\partial\Lktnnghe$, where
\begin{align}\label{eq:Lktnn_boundary}
\partial \Lktnnghe=\bigsqcup_{g\lessY g'\lessY h} \Lktpe_g^{g'}.
\end{align}
Let $x\in \Lktnnghe$. By~\eqref{eq:links_stratif}, we have $x\in \Lktpe_g^{g'}$ for a unique $g\lessY g'\leqY h$. If $g'=h$, then $x$ has an open neighborhood in $\Lktnnghe$ homeomorphic to $\R^{d}$ by \cref{links:contractible}. If $g'\lessY h$, then by \cref{lemma:links_local_homeo} we have a local homeomorphism $(\Lktnnghe, x)\xrasim(\sctnn_{g',h}\times \R^{d'}, (0,0))$, where $d' := \gradY(g')-\gradY(g)-1$. By \cref{links:cone}, we have a homeomorphism $\sctnn_{g',h}\xrasim\Cone(\Lktnne_{g',h})$ which sends $0$ to the cone point $c$. By the induction hypothesis, $\Lktnne_{g',h}$ is homeomorphic to a $(d-d'-1)$-dimensional closed ball, and so we have a homeomorphism $\Cone(\Lktnne_{g',h}) \xrasim \R_{\geq0}\times\R^{d-d'-1}$ which sends $c$ to $(0,0)$. Composing gives a local homeomorphism $(\Lktnnghe, x) \xrasim (\R_{\geq0}\times \R^{d-d'-1}\times \R^{d'}, (0,0,0))$.  Thus indeed $\Lktnnghe$ is a topological manifold with boundary given by~\eqref{eq:Lktnn_boundary}.

By \cref{links:compact}, $\Lktnnghe$ is compact. By \cref{links:contractible} and \cref{prop:Cinterior}, $\Lktnnghe$ is contractible. Thus $\Lktnnghe$ is a compact contractible topological manifold with boundary. In addition, the boundary $\partial \Lktnnghe$ is a regular CW complex by the induction hypothesis. Its face poset is the interval $(g,h):=[g,h]\setminus\{g,h\}$ in $\YQ$. The interval $[g,h]$ is graded, thin, and shellable by~\axref{TNN:graded_max}, \axref{TNN:thin_shell}, and \cref{prop:shellable_int}, and thus $\partial \Lktnnghe$ is homeomorphic to a $(d-1)$-dimensional sphere by \cref{thm:bjorner}. By \cref{thm:Poincare}, we get a homeomorphism from $\Lktnnghe$ to a $d$-dimensional closed ball $\Bcl$. By~\eqref{eq:Lktnn_boundary}, it sends the interior $\Lktpe_g^h$ to the interior of $\Bcl$. This proves~\axref{cond:cw:regularity}, and~\axref{cond:cw:normality} follows from~\eqref{eq:Lktnn_boundary}. This completes the induction.
\end{proof}

\begin{proof}[Proof of \cref{thm:FS_intro}.]
The proof follows the same structure as the proof of \cref{thm:Lk_ball}. We proceed by induction on $\gradY(h)$. If $\gradY(h)=0$, then $\Ytnn_h=\Ytp_h$ is a point by~\axref{TNN:Ytp}, which finishes the base case.

Let $\gradY(h)>0$. We want to show that $\Ytnn_h$ is a topological manifold with boundary
\begin{equation}\label{eq:Ytnn_boundary}
\partial \Ytnn_h=\bigsqcup_{g\lessY h} \Ytp_g.
\end{equation}
Let $x\in \Ytnn_h$. By~\axref{TNN:Ytp_cl}, we have $x\in \Ytp_g$ for a unique $g\leqY h$. If $g=h$, then $x$ has an open neighborhood in $\Ytnn_h$ homeomorphic to $\R^{\gradY(h)}$ by~\axref{TNN:Ytp}. If $g\lessY h$, then $\Startnn_g$ is an open subset of $\Ytnn$ (its complement is $\bigcup_{g'\not\geqY g} \Ytnn_{g'}$, which is closed by~\axref{TNN:Ytp_cl}). Thus $\Startnn_{g,h}$ is an open neighborhood of $x$ in $\Ytnn_h$.  By \cref{links:TNN_homeo}, \axref{TNN:Ytp}, and \cref{thm:Lk_ball}, we get a homeomorphism $\Startnn_{g,h}\xrasim\R_{\geq0}\times \R^{\gradY(h)-1}$ whose first component sends $x\in \Ytp_g$ to $0\in \R_{\geq0}$. This shows that $\Ytnn_h$ is a topological manifold with boundary given by~\eqref{eq:Ytnn_boundary}.

By \axref{TNN:Ytnn_compact} and~\axref{TNN:Ytp_cl}, $\Ytnn_h$ is compact. By~\axref{TNN:Ytp} and \cref{prop:Cinterior}, $\Ytnn_h$ is contractible. Thus $\Ytnn_h$ is a compact contractible topological manifold with boundary. In addition, the boundary $\partial \Ytnn_h$ is a regular CW complex by the induction hypothesis. Its face poset is the interval $(\hat0,h)$ in $\YQhat$. The interval $[\hat0,h]$ is graded, thin, and shellable by~\axref{TNN:graded_max},~\axref{TNN:thin_shell}, and \cref{prop:shellable_int}, and thus $\partial \Ytnn_h$ is homeomorphic to a $(d-1)$-dimensional sphere by \cref{thm:bjorner}. We are done by \cref{thm:Poincare}, as in the proof of \cref{thm:Lk_ball}.
\end{proof}

\begin{remark}\label{rmk:U}
We note that \cref{thm:main_intro2,thm:Lk_ball} imply the result of Hersh \cite{Her} (see \cref{cor:Her}) that the link of the identity in the Bruhat decomposition of $U_{\ge 0}$ is a regular CW complex. (Recall that $U$ is the unipotent radical of the standard Borel subgroup $B \subset G$.) Indeed, let $B_- \subset G$ denote the opposite Borel subgroup. Then the natural inclusion $U \hookrightarrow G/B_-$ sends $U$ to the opposite Schubert cell $\Star_{(\id,\id)}$ indexed by $\id \in W$, and the composition of this map with $\FSiso_{(\id,\id)}$ sends the link of the identity in $\overline{U^w_{>0}}$ homeomorphically to $\Lktnn_{(\id,\id),(\id,w)}$ for all $w \in W$.
\end{remark}

\section{\texorpdfstring{$G/P$}{G/P}: preliminaries}\label{sec:gp:-basic}

We give some background on partial flag varieties. Throughout, $\K$ denotes an algebraically closed field of characteristic $0$, and $\Kast:=\K\setminus\{0\}$ denotes its multiplicative group.

\subsection{Pinnings}\label{sec:pinnings}
We recall some standard notions that can be found in e.g.~\cite[\S1]{LusGB}. Suppose that $G$ is a simple and simply connected algebraic group over $\K$, with $T\subset G$ a maximal torus. Let $B,B_-$ be opposite Borel subgroups satisfying $T=B\cap B_-$. We identify $G$ with its set of $\K$-valued points.  When $\K=\C$, we assume that $G$ and $T$ are split over $\R$, and denote by $\GR\subset G$ and $\TR\subset T$ the sets of their $\R$-valued points. (Thus what was denoted by $G$ in \cref{sec:introduction} is from now on denoted by $\GR$. We are also assuming that $G$ is a simple algebraic group, rather than semisimple; our results for the case of a general semisimple group reduce to the simple case by taking products.)

 Let $X(T):=\Hom(T,\Kast)$ be the \emph{weight lattice}, and for a weight $\gamma\in X(T)$ and $a\in T$, we denote the value of $\gamma$ on $a$ by $a^\gamma$. Let $\Phi\subset X(T)$ be the set of \emph{roots}. We have a decomposition $\Phi=\Phi^+\sqcup\Phi^-$ of $\Phi$ as a union of positive and negative roots corresponding to the choice of $B$; see~\cite[\S27.3]{Hum}. For $\alpha\in \Phi$, we write $\alpha>0$ if $\alpha\in\Phi^+$ and $\alpha<0$ if $\alpha\in\Phi^-$. Let $\{\alpha_i\}_{i\in I}$ be the simple roots corresponding to the choice of $\Phi^+$. For every $i\in I$, we have a homomorphism $\phi_i:\SL_2\to G$, and denote
\begin{equation}\label{eq:x_i(t)}
x_i(t):=\phi_i \begin{pmatrix}
1&t\\
0&1
\end{pmatrix}, \quad y_i(t):=\phi_i \begin{pmatrix}
1&0\\
t&1
\end{pmatrix},\quad  \ds_i:=\phi_i \begin{pmatrix}
0&-1\\
1&0
\end{pmatrix}=y_i(1)x_i(-1)y_i(1).
\end{equation}

The data $(T,B, B_-,x_i,y_i;i\in I)$ is called a \emph{pinning} for $G$. Let $W:=N_G(T)/T$ be the Weyl group, and for $i\in I$, let $s_i\in W$ be represented by $\ds_i$ above. Then $W$ is generated by $\{s_i\}_{i\in I}$, and $(W,\{s_i\}_{i\in I})$ is a finite Coxeter group. For $w\in W$, the \emph{length} $\ell(w)$ is the minimal $n$ such that $w=s_{i_1}\cdots s_{i_n}$ for some $i_1,\dots,i_n\in I$. When $n=\ell(w)$, we call $\bi:=(i_1,\dots,i_n)$ a \emph{reduced word} for $w$. The representatives $\{\ds_i\}_{i\in I}$ satisfy the braid relations~\cite[Proposition 9.3.2]{Spr}, so we set $\dw:=\ds_{i_1}\cdots \ds_{i_n}\in G$, and this representative does not depend on the choice of $\bi$.

Let $Y(T):=\Hom(\Kast,T)$ be the \emph{coweight lattice}. For $i\in I$, we have a simple coroot $\alphacheck_i(t):=\phi_i\begin{pmatrix}
t&0\\0&t^{-1}
\end{pmatrix}\in Y(T)$. Denote by $\<\cdot,\cdot\>:Y(T)\times X(T)\to \Z$ the natural pairing so that for $\gamma\in X(T)$, $\beta\in Y(T)$, and $t\in \Kast$, we have $(\beta(t))^\gamma=t^{\<\beta,\gamma\>}$. Let $\{\omega_i\}_{i\in I}\subset X(T)$ be the \emph{fundamental weights}. They form a dual basis to $\{\alphacheck_i\}_{i\in I}$: $\<\alphacheck_j,\omega_i\>=\delta_{ij}$ for $i,j\in I$.

The Weyl group $W$ acts on $T$ by conjugation, which induces an action on $Y(T)$, $X(T)$, and $\Phi$. For $\gamma\in X(T)$, $t\in \Kast$, $a\in T$, and $w\in W$, we have~\cite[(1.2) and~(2.5)]{FZ}
\begin{equation}\label{eq:torus_conj}
(\dw^{-1} a\dw)^\gamma=a^{w\gamma},\quad ax_i(t)a^{-1}=x_i(a^{\alpha_i}t),\quad ay_i(t)a^{-1}=y_i(a^{-\alpha_i}t).
\end{equation}

Following~\cite[(1.6) and~(1.7)]{BZ} (see also~\cite[(2.1) and~(2.2)]{FZ}), we define two involutive anti-automorphisms $x\mapsto x^T$ and $x\mapsto x^\iota$ of $G$ by 
\begin{align}\label{eq:transpose}
 (x_i(t))^T&=y_i(t),& (y_i(t))^T&=x_i(t),& \dw^T&= \dw^{-1},& a^T&= a,\\
\label{eq:iota}
 (x_i(t))^\iota&=x_i(t),& (y_i(t))^\iota&=y_i(t),& \dw^\iota&= \dz,& a^\iota&= a^{-1}
\end{align}
for all $i\in I$, $t\in \Kast$, $a\in T$, and $w\in W$, where $z:=w^{-1}$. We note that when $z=w^{-1}\in W$ and $\bi=(i_1,\dots, i_n)$ is a reduced word for $w$ then $\dw^{-1}=\ds_{i_n}^{-1}\cdots \ds_{i_1}^{-1}$ while $\dz=\ds_{i_n}\cdots \ds_{i_1}$.

\subsection{Subgroups of \texorpdfstring{$U$}{U}}\label{sec:subgroups-u}
We say that a subset $\Theta\subset \Phi$ is \emph{bracket closed} if whenever $\alpha,\beta\in \Theta$ are such that $\alpha+\beta\in\Phi$, we have $\alpha+\beta\in \Theta$.  For a bracket closed subset $\Theta\subset \Phi^+$, define $U(\Theta)\subset U$ to be the subgroup generated by $\{U_\alpha\mid \alpha\in \Theta\}$, where $U_\alpha$ is a \emph{root subgroup} of $G$; see~\cite[Theorem~26.3]{Hum}. For a bracket closed subset $\Theta\subset \Phi^-$, let $U_-(\Theta):=U(-\Theta)^T\subset U_-$. 

Given closed subgroups $H_1,\dots,H_n$ of an algebraic group $H$, we say that $H_1,\cdots, H_n$ \emph{directly span} $H$ if the multiplication map $H_1\times\dots\times H_n\to H$ is a biregular isomorphism.

\begin{lemma}[{\cite[Proposition~28.1]{Hum}}]Let $\Theta\subset \Phi^+$ be a bracket closed subset.
\begin{theoremlist} 
\item \label{U(R)_generated} If $\Theta=\bigsqcup_{i=1}^n \Theta_i$ and $\Theta,\Theta_1,\dots,\Theta_n\subset \Phi^+$ are bracket closed then $U(\Theta)$ is directly spanned by $U(\Theta_1),\dots,U(\Theta_n)$.
\item \label{U(alpha)_iso} In particular, $U(\Theta)$ is directly spanned by $\{U_\alpha\mid \alpha\in\Theta\}$ in any order, and therefore $U(\Theta)\cong \K^{|\Theta|}$.
\end{theoremlist}
\end{lemma}

For $\alpha\in \Phi$ and $w\in W$, we have $\dw U_\alpha \dw^{-1}= U_{w\alpha}$. For $w\in W$, define $\Inv(w):=(w^{-1}\Phi^-)\cap \Phi^+$. The subsets $\Inv(w)$ and $\Phi^+\setminus \Inv(w)$ are bracket closed~\cite[\S28.1]{Hum}, and 
\begin{equation}\label{eq:U(Inv)}
U(\Inv(w))=\dw^{-1} U_-\dw \cap U.
\end{equation}

\subsection{Bruhat projections}
Let $\Theta\subset \Phi^+$ be bracket closed, and let $w\in W$. Define $\Theta_1:=\Theta\cap \Inv(w)$ and $\Theta_2:=\Theta\setminus \Inv(w)$. Thus both sets are bracket closed and 
\[\dw U(\Theta)\dw^{-1}\cap U_-=U_-(w\Theta_1),\quad \dw U(\Theta)\dw^{-1}\cap U=U(w\Theta_2).\]
Denote $U_1:=U_-(w\Theta_1)$ and $U_2:=U(w\Theta_2)$. Then by \cref{U(R)_generated}, the multiplication map gives isomorphisms $\mu_{12}:U_1\times U_2\to \dw U(\Theta) \dw^{-1}$ and $\mu_{21}:U_2\times U_1\to \dw U(\Theta) \dw^{-1}$. Denote by $\nu_1:\dw U(\Theta) \dw^{-1}\to U_1$ and $\nu_2:\dw U(\Theta) \dw^{-1}\to U_2$ the second component of $\mu_{21}^{-1}$ and $\mu_{12}^{-1}$, respectively. In other words, given $g\in \dw U(\Theta) \dw^{-1}$, $\nu_1(g)$ is the unique element in $U_1\cap U_2g$ and $\nu_2(g)$ is the unique element in $U_2\cap U_1g$. 

\begin{lemma}[{\cite[Lemma~2.2]{KWY}}]\label{lemma:KWY} 
The map $(\nu_1,\nu_2):\dw U(\Theta) \dw^{-1}\to U_1\times U_2$ is a biregular isomorphism.
\end{lemma}

\subsection{Commutation relations}
Let $a,b\in W$ be such that $\ell(ab)=\ell(a)+\ell(b)$. Then
\begin{equation}\label{eq:Inv(ab)_sqcup}
  \Inv(b)\subset \Inv(ab),\quad  b^{-1}\Inv(a)\subset\Phi^+,\quad\text{and}\quad\Inv(ab)=\left(b^{-1}\Inv(a)\right)\sqcup \Inv(b).
\end{equation}
Thus by \cref{U(R)_generated}, the multiplication map gives an isomorphism
\begin{equation}\label{eq:Inv(ab)_iso}
  \db^{-1}U(\Inv(a))\db\times U(\Inv(b))\xrightarrow{\sim}U(\Inv(ab)).
\end{equation}
We will later need the following consequences of~\eqref{eq:Inv(ab)_iso}: if $\ell(ab)=\ell(a)+\ell(b)$ then 
\begin{align}\label{eq:Inv(ab)_commute}
\db^{-1} \cdot (U_-\cap \da^{-1}U\da) &\subset (U_-\cap (\da\db)^{-1} U\da\db) \cdot \db^{-1},\\
\label{eq:Inv(ab)_commute2}
(U\cap \da^{-1}U_-\da)\cdot \db & \subset \db \cdot (U\cap (\da\db)^{-1}U_-\da\db).
\end{align}
Multiplying both sides of~\eqref{eq:Inv(ab)_commute2} by $\db^{-1}$ on the left, we get $\db^{-1}U(\Inv(a))\db\subset U(\Inv(ab))$, which follows from~\eqref{eq:Inv(ab)_sqcup}. We obtain~\eqref{eq:Inv(ab)_commute} from~\eqref{eq:Inv(ab)_commute2} by applying the map $x\mapsto x^T$; see~\eqref{eq:transpose}.

\begin{lemma}\label{commute_alpha_beta}
Let $\alpha\in\Phi^+$ and $i\in I$ be such that $\alpha\neq\alpha_i$. Let $\Psi\subset \Phi$ denote the set of all roots of the form $m\alpha-r\alpha_i$ for integers $m>0$, $r\geq0$.  Then $\Psi$ is a bracket closed subset of $\Phi^+$, and for all $t\in \K$ we have $y_i(t)U_\alpha y_i(-t)\subset U(\Psi)$.
\end{lemma}
\begin{proof}
Let $x\in U_{\alpha}$ and $x':=\ds_i^{-1} x\ds_i\in U_{s_i\alpha}$. By~\cite[Lemma~4.4.3]{BB}, $s_i$ permutes $\Phi^+\setminus \{\alpha_i\}$ (in particular, $s_i\alpha>0$). Write
\[y_i(t)\cdot x\cdot  y_i(-t)=\ds_ix_i(-t)\ds_i^{-1} \cdot x \cdot \ds_i x_i(t)\ds_i^{-1}= \ds_i x_i(-t)\cdot x'\cdot  x_i(t)\ds_i^{-1}.\]
Denote by $\Psi'\subset \Phi$ the set of all roots of the form $ms_i\alpha+r\alpha_i$ for integers $m,r>0$. It is clear that $\Psi'\subset \Phi^+\setminus \{\alpha_i,s_i\alpha\}$ is a bracket closed subset.  By~\cite[Lemma~32.5]{Hum}, we have $x_i(-t) x' x_i(t) x'^{-1}\in U(\Psi')$, so $x_i(-t) x' x_i(t) \in U(\Psi')x'$. Thus $\Psi'':=s_i\Psi'$ is also a bracket closed subset of $\Phi^+\setminus \{\alpha_i,\alpha\}$, and we have $\ds_i U(\Psi')x'\ds_i^{-1}= U(\Psi'')x$. Clearly, $\Psi=\Psi''\sqcup\{\alpha\}$. We thus have $y_i(t)U_\alpha y_i(-t)\subset U(\Psi'')U_\alpha = U(\Psi)$.
\end{proof}

\subsection{Flag variety and Bruhat decomposition}\label{sec:flag-variety-bruhat}
Let $G/B$ be the \emph{flag variety} of $G$ (over $\K$). We recall some standard properties of the Bruhat decomposition that can be found in e.g.~\cite[\S28]{Hum}. Define open Schubert, opposite Schubert, and Richardson varieties:
\begin{equation}\label{eq:Schubert_intro}
\X^w:=B \dw B/B,\quad  \X_v:=B_- \dv B/B,\quad \Rich_v^w:=\X_v\cap\X^w\quad (\text{for $v\leq w\in W$}).
\end{equation}

Recall the Bruhat and Birkhoff decompositions:
\begin{align}
\label{eq:bruhat}
  &G=\bigsqcup_{w\in W} B\dw B=\bigsqcup_{v\in W} B_-\dv B,\quad\text{where}\\
\label{eq:BvB_cap_BwB_empty}
  &B_-\dv B\cap B\dw B=\emptyset \quad \text{and}\quad \X_v\cap \X^w=\emptyset \quad\text{for $v\not\leq w\in W$}.
\end{align}

 Let $\Xcl_v$ denote the (Zariski) closure of $\X_v$. Similarly, let $\Xcl^w$ denote the closure of $\X^w$, and then $\Richcl_v^w=\Xcl_v\cap \Xcl^w$ is the closure of $\Rich_v^w$ in $G/B$. We have
\begin{align}
\label{eq:Schubert_closure}
&\Xcl_v=\bigsqcup_{v\leq v'} \X_{v'},\quad &&\Xcl^w=\bigsqcup_{w'\leq w} \X^{w'},\\
\label{eq:GB_closure}
&G/B=\bigsqcup_{v\leq w} \Rich_v^w,\quad  &&\Richcl_v^w=\bigsqcup_{v\leq v'\leq w'\leq w} \Rich_{v'}^{w'}.
\end{align}

For any $w\in W$, $i\in I$, and $t\in \K^\ast$, we have 
\begin{align}\label{eq:y_i_Bs_iB}
&x_i(t)\in  B_-\ds_iB_-,\quad y_i(t)\in B\ds_iB,\\
\label{eq:Bruhat_s_i_BwB}
B\ds_iB\cdot B\dw B&\subset 
\begin{cases}
  B\ds_i\dw B, &\text{if $s_iw>w$,}\\
B\ds_i\dw B\sqcup B\dw B,&\text{if $s_iw<w$,}
\end{cases}\\
\label{eq:Bruhat_s_i_B-wB}
B_-\ds_iB_-\cdot B_-\dw B&\subset 
\begin{cases}
  B_-\ds_i\dw B, &\text{if $s_iw<w$,}\\
B_-\ds_i\dw B\sqcup B_-\dw B,&\text{if $s_iw>w$,}
\end{cases}\\
\label{eq:Bruhat_length_add} B\dv B\cdot B\dw B&\subset B\dv \dw B\quad \text{for $v\in W$ such that  $\ell(vw)=\ell(v)+\ell(w)$.}
\end{align}
For $\t=(t_1,\dots,t_n)\in(\K^\ast)^n$ and a reduced word $\bi=(i_1,\dots,i_n)$  for $w\in W$, define 
\begin{equation}\label{eq:bx_bi_by_bi}
\bx_\bi(\t):=x_{i_1}(t_1)\cdots x_{i_n}(t_n),\quad\text{and}\quad\by_\bi(\t):=y_{i_1}(t_1)\cdots y_{i_n}(t_n).
\end{equation}
 It follows from~\eqref{eq:y_i_Bs_iB}, \eqref{eq:Bruhat_s_i_BwB}, and~\eqref{eq:transpose} that
\begin{equation}\label{eq:U_BwB_over_K}
\bx_\bi(\t)\in B_-\dw B_-,\quad \by_\bi(\t)\in B\dw B.
\end{equation}

\subsection{Parabolic subgroup \texorpdfstring{$W_J$}{W\textunderscore J} of \texorpdfstring{$W$}{W}}\label{sec:parab-subgr-w_j}
We give a description of the poset $Q_J$ studied in~\cite{Rie,GY,KLS,HL} in a form adapted to our needs in this paper.

Let $J\subset I$, and denote by $W_J\subset W$ the subgroup generated by $\{s_i\}_{i\in J}$. Let $W^J$ be the set of minimal-length coset representatives of $W/W_J$; see~\cite[\S2.4]{BB}. Let $\woj$ be the longest element of $W_J$, and $\wj:=w_0\woj$ be the maximal element of $W^J$. Let $\Phi_J\subset \Phi$ consist of roots that are linear combinations of $\{\alpha_i\}_{i\in J}$. Define 
\[\Phi_J^+:=\Phi_J\cap \Phi^+,\quad \Phi_J^-:=\Phi_J\cap \Phi^-,\quad \Phij_+:=\Phi^+\setminus \Phi_J^+,\quad\Phij_-:=\Phi^-\setminus \Phi_J^-.\]
The sets $\Phi_J^+$, $\Phij_+$, $\Phi_J^-$, $\Phij_-$  are clearly bracket closed, so consider subgroups
\[U_J=U(\Phi_J^+),\quad U_J^-=U_-(\Phi_J^-),\quad \Uj=U(\Phij_+),\quad \Uj_-=U_-(\Phij_-).\]
In fact, we have
\begin{equation}\label{eq:dwoj_U_J}
\Phi_J^+=\Inv(\woj),\quad \Phij_+=\Inv(\wj),\quad \dwoj U_J^-\dwoj^{-1}=U_J.
\end{equation}

Let $W^J_\maxx:=\{w\woj\mid w\in W^J\}$. By~\cite[Proposition~2.4.4]{BB}, every $w\in W$ admits a unique \emph{parabolic factorization} $w=w_1w_2$ for $w_1\in W^J$ and $w_2\in W_J$, and this factorization is length-additive. We state some standard facts on parabolic factorizations for later use.
\begin{lemma}\leavevmode
\begin{theoremlist}
\item \label{s_iW^J_still_W^J} If $u\in W^J$ and $s_iu<u$, then $s_iu\in W^J$.
\item \label{ur_leq_ur'} Given $u\in W^J$ and $r,r'\in W_J$, we have $ur\leq ur'$ if and only if $r\leq r'$.
\end{theoremlist}
\end{lemma}
\begin{proof}
For~\itemref{s_iW^J_still_W^J} suppose that $s_iu\notin W^J$, so that $s_ius_j<s_iu$ for some $j\in J$. Then $s_ius_j<s_iu<u<us_j$, which contradicts $\ell(us_j)=\ell(s_ius_j)\pm1$. For~\itemref{ur_leq_ur'}, see~\cite[Exercise~2.21]{BB}. 
\end{proof}

\begin{lemma} 
\label{Inv_W^J} For any $w\in W^J$, we have $\Inv(w)\subset\Phij_+$. In particular, $w\Phi_J^+\subset\Phi^+$, $\dw U_J\dw^{-1}\subset U$, and $\dw U_J^- \dw^{-1}\subset U^-$. 
\end{lemma}
\begin{proof}
Let $\alpha\in\Phi^+$ be a positive root. Then it can be written as $\alpha=\sum_{i\in I} c_i\alpha_i$ for $c_i\in \Z_{\geq0}$. Since $w\in W^J$, we have $w\alpha_i>0$ for all $i\in J$. Thus if $w\alpha<0$, we must have $c_i\neq0$ for some $i\notin J$, so $\alpha\in \Phij_+$.
\end{proof}

\begin{lemma}[{\cite{He}}]\label{lemma:x*y}
Let $x,y\in W$.
\begin{theoremlist}
\item\label{x*y:dfn} The set $\{uv\mid u\leq x,\  v\leq y\}$ contains a unique maximal element, denoted $x*y$. The set $\{xv \mid v \le y\}$ contains a unique minimal element, denoted $x \triangleleft y$.
\item\label{x*y:u_v} There exist elements $u'\leq x$ and $v'\leq y$ such that $x*y=xv'=u'y$, and these factorizations are both length-additive.
\item\label{x*y:compare} If $x' \le x$, then $x'*y \le x*y$ and $x' \triangleleft y \le x \triangleleft y$.
\item\label{x*y:length_add} If $xy$ is length-additive, then $x*y=xy$ and $(xy)\triangleleft y^{-1} = x$.
\end{theoremlist}
\end{lemma}
\noindent The operations $*$ and $\triangleleft$ are called the \emph{Demazure product} and \emph{downwards Demazure product}.
\begin{proof}
The first three parts were shown in~\cite[\S 1.3]{He}, with the caveat that our $\triangleleft$ is the `mirror image' of He's $\triangleright$. Part~\itemref{x*y:length_add} follows from the definitions of $*$ and $\triangleleft$.
\end{proof}

\begin{definition}
\label{Q_J_relation} Let $ Q_J=\{(v,w)\in W\times W^J\mid v\leq w\}$. We define an order relation $\leqJ$ on $Q_J$ as follows: for $(v,w),(v',w')\in Q_J$, we write $(v,w)\leqJ (v',w')$ if and only if there exists $r\in W_J$ such that $vr$ is length-additive and $v'\leq vr\leq wr\leq w'$. For $(v,w)\in Q_J$, define
\[
\QJfilter vw:=\{(v',w')\in Q_J\mid (v,w)\leqJ (v',w')\}, \quad \QJideal vw:=\{(v',w')\in Q_J\mid (v',w')\leqJ (v,w)\}.
\]
\end{definition}

\begin{lemma} \label{lemma:from_KLS}\ 
  \begin{theoremlist} 
\item \label{Q_J_omit_length_add} Let $(v,w),(v',w')\in Q_J$, $r\in W_J$, and $v'\leq vr\leq wr\leq w'$. Then ${(v,w)\leqJ(v',w')}$.
\item \label{Q_J_charact} Let $(u,u), (v,w), (v',w') \in Q_J$. Then $(u,u)\leqJ (v,w)\leqJ(v',w')$ if and only if
\begin{align}\label{Q_J_charact_eqn}
v'\leq vr'\leq ur\leq wr'\leq w'\quad \text{for some $r,r'\in W_J$ such that $vr'$ is length-additive}.
\end{align}
\end{theoremlist}
\end{lemma}
\begin{proof}

\itemref{Q_J_omit_length_add}: By \cref{lemma:x*y}, there exists $r'\leq r$ such that $v*r=vr' \ge vr$, and $vr'$ is length-additive. We have $vr'\leq wr'$ by \cref{x*y:compare}, and $wr'\leq wr$ by \cref{ur_leq_ur'}. Therefore $v' \le vr \le vr' \le wr' \le wr \le w'$, so $(v,w) \leqJ (v',w')$.

\itemref{Q_J_charact} ($\Rightarrow$): Suppose that $(u,u)\leqJ (v,w) \leqJ (v',w')$.  Then by \cref{Q_J_relation}, there exist $r',r'' \in W_J$ such that $vr'$ is length-additive, $v' \leq vr' \leq wr' \leq w'$, and $v\leq ur''\leq w$. Define $r\in W_J$ by the equality $(ur'')*r' = ur$. Then applying ${}*r'$ on the right to $v \leq ur'' \leq w$, by \cref{x*y:compare}--\itemref{x*y:length_add}, we obtain $vr' \le ur \le wr'$. Therefore \eqref{Q_J_charact_eqn} holds.

\itemref{Q_J_charact} ($\Leftarrow$): Suppose that \eqref{Q_J_charact_eqn} holds. Then $(v,w) \leqJ (v',w')$. Define $r'' \in W_J$ by the equality $(ur) \triangleleft r'^{-1} = ur''$. Then applying ${}\triangleleft (r')^{-1}$ on the right to $vr' \le ur \le wr'$, by \cref{x*y:compare}--\itemref{x*y:length_add}, we obtain $v \le ur'' \le w$. Therefore $(u,u) \leqJ (v,w)$.
\end{proof}

\begin{remark}\label{rmk:Q_J_omit_ell_vr}
By \Cref{Q_J_omit_length_add}, \cref{Q_J_relation} remains unchanged if we omit the condition that $vr$ is length-additive. It follows that $Q_J$ coincides with the poset studied in~\cite[\S2.4]{HL}. Therefore by~\cite[Appendix]{HL}, $Q_J$ is also isomorphic to the posets studied in~\cite{Rie,GY,KLS}.
\end{remark}

\subsection{Partial flag variety \texorpdfstring{$G/P$}{G/P}}\label{sec:partial-flag-variety}
Fix $J\subset I$ as before. Let $P\subset G$ be the subgroup generated by $B$ and $\{y_i(t)\mid t\in\Kast,\ i\in J\}$. We denote by $G/P$ the \emph{partial flag variety} corresponding to $J$, and let $\pi_J:G/B\to G/P$ be the natural projection map. Let $L_J\subset P$ be the \emph{Levi subgroup} of $P$. It is generated by $T$ and $\{x_i(t),y_i(t)\mid i\in J,\ t\in\Kast\}$. Let $P_-$ be the parabolic subgroup opposite to $P$, with $L_J=P\cap P_-$.

For $(v,w)\in Q_J$ we introduce $\PR_v^w:=\pi_J(\Rich_v^w)\subset G/P$, and define the \emph{projected Richardson variety} $\PRcl_v^w\subset G/P$ to be the closure of $\PR_v^w$ in the Zariski topology. By~\cite[Proposition~3.6]{KLS}, we have
\begin{equation}
\label{eq:GP_closure}
G/P=\bigsqcup_{(v,w)\in Q_J}\PR_v^w,\quad\text{and}\quad \PRcl_v^w=\bigsqcup_{(v',w')\in \QJideal vw} \PR_{v'}^{w'}.
\end{equation}

Now let $\K=\C$. The varieties $\X^w$, $\X_v$, $X^w$, $X_v$, $\Rich_v^w$, and $\Richcl_v^w$ are defined over $\R$.  The map $\pi_J$ is defined over $\R$ as well, and thus so are  $\PR_v^w$ and $\PRcl_v^w$. We let
\begin{align*}
\GBR &:=\{gB\mid g\in\GR\}\subset G/B,& \RRich_v^w &:=\GBR\cap \Rich_v^w,& \RRichcl_v^w &:=\GBR\cap \Richcl_v^w, \\
\GPR &:=\{xP\mid x\in \GR\}\subset G/P,& \PRR_v^w &:=\PR_v^w\cap \GPR,& \PRRcl_v^w &:=\PRcl_v^w\cap \GPR.
\end{align*} 
It follows that the  decomposition~\eqref{eq:GP_closure} is valid over $\R$:
\begin{align}
\label{eq:GP_closure_intro}
&\GPR=\bigsqcup_{(v,w)\in Q_J}\PRR_v^w,\quad &&\PRRcl_v^w=\bigsqcup_{(v',w')\in \QJideal vw} \PRR_{v'}^{w'}.
\end{align}

\subsection{Total positivity}\label{sec:total-posit-intro}
Assume $\K=\C$ in this section. Recall from \cref{sec:pinnings} that for each $i\in I$, we have elements $x_i(t)$, $y_i(t)$ (for $t\in\C$) and $\alphacheck_i(t)$ (for $t\in \Cast$).

\begin{definition}[{\cite{LusGB}}]
\label{Gtnn_dfn} Let $\Gtnn\subset \GR$ be the submonoid generated by $x_i(t)$, $y_i(t)$, and $\alphacheck_i(t)$ for $t\in \R_{>0}$. Define $\GBtnn$ to be the closure of $(\Gtnn/B)\subset\GBR$ in the analytic topology. For $v\leq w\in W$, let $\Rtnn_v^w$ denote the closure of $\Rtp_v^w:=\Rich_v^w\cap \GBtnn$ inside $\GBtnn$.
\end{definition}

\begin{definition}[{\cite{LusGP,Rie99}}]\label{dfn:Rtnn_PRtnn}
Set $\GPtnn:=\pi_J(\GBtnn)$. For $(v,w)\in Q_J$, let $\PRtnn_v^w$ denote the closure of $\PRtp_v^w:=\pi_J(\Rtp_v^w)$ inside $\GPtnn$.
\end{definition}
\noindent Thus we denote by $\PRtp_v^w$ what was denoted by $\Povtp_{(v,w)}$ in \cref{ex:flag_intro}. We have decompositions
\begin{align}
\label{eq:GPtnn_closure_intro}
\quad 
&\GPtnn=\bigsqcup_{(v,w)\in Q_J}\PRtp_v^w,\quad &&\PRtnn_v^w=\bigsqcup_{(v',w')\in \QJideal vw} \PRtp_{v'}^{w'}.
\end{align}
\noindent As a special case of~\eqref{eq:GPtnn_closure_intro}, for $J=\emptyset$ we have
\begin{align} \label{eq:GBtnn_closure}
 &\GBtnn=\bigsqcup_{v\leq w}\Rtp_v^w,\quad &&\Rtnn_v^w=\bigsqcup_{v\leq v'\leq w'\leq w} \Rtp_{v'}^{w'}.
\end{align}

\begin{lemma}\label{KLS_pi_J_length_add} (Assume $\K=\C$.) Let $(v,w)\in Q_J$ and $r\in W_J$ be such that $vr$ is length-additive.
Then 
\begin{align}\label{eq:KLS_pi_J_length_add}
&\PR_v^w=\pi_J(\Rich_v^w)=\pi_J(\Rich_{vr}^{wr}), &&\PRtp_v^w=\pi_J(\Rtp_v^w)=\pi_J(\Rtp_{vr}^{wr}),\\
\label{eq:KLS_pi_J_length_add_closed}
&\PRcl_v^w=\pi_J(\Richcl_v^w)=\pi_J(\Richcl_{vr}^{wr}), &&\PRtnn_v^w=\pi_J(\Rtnn_v^w)=\pi_J(\Rtnn_{vr}^{wr}).
\end{align}
\end{lemma}
\begin{proof}
By \cite[Lemma~3.1]{KLS}, we have $\pi_J(\Rich_v^w)=\pi_J(\Rich_{vr}^{wr}) = \PR_v^w$, and $\pi_J$ restricts to isomorphisms $\Rich_v^w \xrasim \PR_v^w$, $\Rich_{vr}^{wr} \xrasim \PR_v^w$. Thus $\pi_J(\Rtp_v^w)=\pi_J(\Rtp_{vr}^{wr}) = \PRtp_v^w$ follows from the equality $\pi_J(\GBtnn)=\GPtnn$, proving~\eqref{eq:KLS_pi_J_length_add}. To show~\eqref{eq:KLS_pi_J_length_add_closed}, note that $\Richcl_a^b$ and $\Rtnn_a^b$ are compact for any $a\leq b$, and therefore their images under $\pi_J$ are closed.
\end{proof}

Recall the definition of $\bx_\bi(\t)$ and $\by_\bi(\t)$ from~\eqref{eq:bx_bi_by_bi}. Choose a reduced word $\bi=(i_1,\dots,i_n)$ for $w\in W$ and define
\[\Utp(w):=\{\bx_\bi(\t)\mid \t\in\R_{>0}^n\},\quad \Utp^-(w):=\{\by_\bi(\t)\mid \t\in\R_{>0}^n\}.\]  
Let $\Utnn\subset U(\R)$ (respectively, $\Utnn^-\subset U_-(\R)$) be the submonoid generated by $x_i(t)$ (respectively, by $y_i(t)$) for $t\in\R_{>0}$. Then $\Utnn=\bigsqcup_{w\in W} \Utp(w)$ and  $\Utnn^-=\bigsqcup_{w\in W} \Utp^-(w)$. We have $\Utp(w)=\Utnn\cap B_-\dw B_-$ and $\Utp^-(w)=\Utnn^-\cap B\dw B$, and these sets do not depend on the choice of the reduced word $\bi$ for $w$; see~\cite[Proposition~2.7]{LusGB}.

\subsection{Marsh--Rietsch parametrizations}\label{sec:MR_param}
Assume that $\K$ is algebraically closed. Given $w\in W$, an \emph{expression} $\bw$ for $w$ is a sequence $\bw=(w\pd0,\dots,w\pd n)$ such that $w\pd0=\id$, $w\pd n=w$, and for $j=1,\dots,n$, either $w\pd j=w\pd{j-1}$ or $w\pd j=w\pd{j-1}s_{i_j}$ for some $i_j\in I$. In the latter case we require $w\pd{j-1}<w\pd j$, unlike in~\cite{MR}. We define $J_\bw^+:=\{1\leq j\leq n\mid w\pd {j-1}<w\pd{j}\}$ and $J_\bw^\circ:=\{1\leq j\leq n\mid w\pd {j-1}=w\pd{j}\}$ so that $J_\bw^+\sqcup J_\bw^\circ=\{1,2,\dots,n\}$. Every reduced word $\bi=(i_1,\dots,i_n)$ for $w$ gives rise to a \emph{reduced expression} $\bw=\bw(\bi)=(w\pd0,\dots,w\pd n)$ with $w\pd j=w\pd{j-1}s_{i_j}$ for $j=1,\dots, n$.
\begin{lemma}[{\cite[Lemma~3.5]{MR}}]
\label{lemma:positive_subexpression} Let $v\leq w\in W$, and consider a reduced expression $\bw=(w\pd0,\dots,w\pd n)$ for $w$ corresponding to a reduced word $\bi=(i_1,\dots,i_n)$. Then there exists a unique \emph{positive subexpression $\bv$ for $v$ inside} $\bw$, i.e., an expression $\bv=(v\pd0,\dots,v\pd n)$ for $v$ such that for $j=1,\dots,n$, we have $v\pd{j-1}<v\pd{j-1}s_{i_j}$. This positive subexpression can be constructed inductively by setting $v\pd n:=v$ and
\begin{equation}\label{eq:pos_subexpr_inductively}
v\pd{j-1}:=
  \begin{cases}
    v\pd j s_{i_j}, &\text{if $v\pd j s_{i_j}<v\pd j$,}\\ 
    v\pd j,&\text{otherwise},
  \end{cases}\quad \text{ for $j=n,\dots,1$.} 
\end{equation}
\end{lemma}
\begin{corollary}
\label{cor:positive_sub_v_not_leq_w} In the setting above, if $v\pd 1=s_i$ for some $i\in I$ then $v\not\leq s_iw$.
\end{corollary}
\begin{proof}
Indeed, if $v\leq s_iw<w$ then there exists a positive subexpression $\bv'=(v'\pd{0},\dots,v'\pd{n-1})$ for $v$ inside $\bw(\i')$, where $\i'=(i_2,\dots,i_n)$. By~\eqref{eq:pos_subexpr_inductively}, we have $v'\pd j=v\pd {j+1}$ for $j=0,1,\dots,n-1$, which contradicts the fact that $v'\pd 0=1$ while $v\pd 1=s_i$.
\end{proof}

 For $w\in W$, let $\Red(w):=\{\bw\mid \text{$\bw$ is a reduced expression for $w$}\}$. For $v\leq w\in W$, let
\[\RedMR(v,w):=\{(\bv,\bw)\mid \bw\in\Red(w),\ \text{$\bv$ is a positive subexpression for $v$ inside $\bw$}\}.\]
Thus for all $v\leq w$, the sets $\Red(w)$ and $\Red(v,w)$ have the same cardinality. Let $v\leq w\in W$ and $(\bv,\bw)\in\RedMR(v,w)$. Given a collection $\t=(t_k)_{k\in J_\bv^\circ}\in(\K^\ast)^{J_\bv^\circ}$, define 
\begin{align}\label{eq:gMRvw}
\gMRvw(\t):=g_1\cdots g_n,\quad\text{where}\quad g_k:=
  \begin{cases}
y_{i_k}(t_k),    & \text{if $k\in J_\bv^\circ$,} \\
\ds_{i_k},     & \text{if $k\in J_\bv^+$.} 
  \end{cases}
\end{align}

\subsubsection{Marsh--Rietsch parametrizations of \texorpdfstring{$\GBtnn$}{G/B}}\label{sec:mr-param-gbtnn}
In this section, we assume $\K=\C$. Let $v$, $w$, $\bv$, and $\bw$ be as above. Define a subset $\GMRtpvw\subset \GR$ by
\[\GMRtpvw:=\{\gMRvw(\t)\mid \t\in \R_{>0}^{J_\bv^\circ}\}.\]

\begin{theorem}[{\cite[Theorem~11.3]{MR}}]\leavevmode
\label{thm:GMR_param} The map $\GR\to \GBR$ sending $g$ to $gB$ restricts to an isomorphism of real semialgebraic varieties
\[\GMRtpvw\xrightarrow{\sim} \Rtp_v^w.\]
\end{theorem}

\begin{proposition}[{\cite[Proposition~8.12]{LusGB}}]
\label{lemma:Gtnn_times_GBtnn} We have $\Gtnn\cdot \GBtnn\subset \GBtnn$.
\end{proposition}

\begin{lemma}
\label{lemma:Demazure} Suppose that $g\in \Gtnn$ and $x\in G$ are such that $xB\in \Rtp_v^w$ for some $v\leq w\in W$. Then $gxB\in \Rtp_{v'}^{w'}$ for some $v'\leq v\leq w\leq w'$.
\end{lemma}

\begin{proof}
  By \cref{lemma:Gtnn_times_GBtnn}, we have $gxB\in \GBtnn$, so it suffices to show that $gx\in B\dw'B\cap B_-\dv'B$ for some $v'\leq v\leq w\leq w'$. Note that we have $x\in B\dw B\cap B_-\dv B$. By \cref{Gtnn_dfn}, it is enough to consider the cases $g=x_i(t)$ and $g=y_i(t)$ for $i\in I$ and $t\in\R_{>0}$. 

  Suppose that $g=y_i(t)$.  We clearly have $gx\in B_-\dv B$. If $s_iw>w$ then  by~\eqref{eq:Bruhat_s_i_BwB} we have $gx\in B\ds_i\dw B$. Thus we may assume that $s_iw<w$. By \cref{thm:GMR_param}, we can also assume $x=\gMR\bv\bw(\t)=g_1\cdots g_n$ for $\t\in\R_{>0}^{J_\bv^\circ}$  and some choice of $(\bv,\bw)\in \RedMR(v,w)$ such that $\bw=(w\pd 0,\dots, w\pd n)$ satisfies $w\pd 1=s_i$. Let $\bv=(v\pd 0,\dots,v\pd n)$. If $v\pd 1\neq s_i$ then $g_1=y_i(t')$, so $gx\in \GMRtpvw$ and we are done. If $v\pd1=s_i$ then by \cref{cor:positive_sub_v_not_leq_w} we have $v\not\leq s_iw$. Recall that $gx\in B_-\dv B$ and by~\eqref{eq:Bruhat_s_i_BwB}, $gx\in B\ds_i\dw B\sqcup B\dw B$. But $B_-\dv B\cap B\ds_i\dw B=\emptyset$ by~\eqref{eq:BvB_cap_BwB_empty}. Therefore we must have $gx\in B\dw B$, finishing the proof in this case.

The case $g=x_i(t)$ follows similarly using a ``dual'' Marsh--Rietsch parametrization~\cite[\S 3.4]{Rie}, where for $(\bv,\bw)\in\RedMR(v,w)$, every element of $\Rtp_{ww_0}^{vw_0}$ is parametrized as 
\[g_1\cdots g_n\dw_0B,\quad\text{where}\quad g_k:=
  \begin{cases}
x_{i_k}(t_k),    & \text{if $k\in J_\bv^\circ$,} \\
\ds_{i_k}^{-1},     & \text{if $k\in J_\bv^+$.} 
  \end{cases}\qedhere\]
\end{proof}

\noindent We will use the following consequence of \cref{thm:GMR_param} in \cref{sec:Gr_total-positivity}.
\begin{corollary}[{cf.~\cite[Proposition~3.3]{KLS}}] 
 \label{cor:MR_projection} 
Let $u\in W^J$, $r\in W_J$, and $v\in W$ be such that $v\leq ur$. Then 
\[\pi_J(\Rtp_v^{ur})=\pi_J(\Rtp_{v\hspace*{0.5pt}\triangleleft\hspace*{0.5pt} r^{-1}}^u)=\PRtp_{v\hspace*{0.5pt}\triangleleft\hspace*{0.5pt} r^{-1}}^u.\]
\end{corollary}
\begin{proof}
Let $\bi=(i_1,\dots,i_n)$ be a reduced word for $w:=ur$, such that $(i_{\ell(u)+1},\dots,i_n)$ is a reduced word for $r$.  Let $(\bv,\bw)\in\Red(v,w)$ be such that $\bw$ corresponds to $\bi$. Then it is clear from \cref{lemma:positive_subexpression} that after setting $\bv':=(v\pd0,\dots,v\pd{\ell(u)})$ and $\bu:=(w\pd0,\dots,w\pd{\ell(u)})$, we get $(\bv',\bu)\in\Red(v\triangleleft r^{-1},u)$. Moreover, the indices $i_{\ell(u)+1},\dots,i_n$ clearly belong to $J$, so if $g_1\cdots g_n\in \GMRtpvw$ then $g_1\cdots g_{\ell(u)}\in\GMRtp{\bv'}\bu$ and $\pi_J(g_1\cdots g_nB)=\pi_J(g_1\cdots g_{\ell(u)}B)$. We are done by \cref{thm:GMR_param}.
\end{proof}

\subsection{\texorpdfstring{$G/P$}{G/P} is a shellable TNN space}

We show that the triple $(\GPR,\GPtnn,Q_J)$  is a shellable TNN space in the sense of \cref{dfn:TNNspace}. We start by recalling several known results.
\begin{theorem}\leavevmode\label{prop:smooth_conn_CW_shellable}
\begin{theoremlist}
\item \label{thin_shellable} The poset $\widehat{Q}_J:=Q_J\sqcup \{\hat0\}$ is graded, thin, and shellable.
\item \label{smooth_manif} $\GPR$ is a smooth manifold. Each $\PRR_v^w$ is a smooth embedded locally closed submanifold of~$\GPR$.
\item \label{conn_comps} For $(v,w)\in Q_J$, $\PRtp_v^w$ is a connected component of $\PRR_v^w$.
\end{theoremlist}
\end{theorem}
\begin{proof}
Part~\itemref{thin_shellable} is due to Williams~\cite{Wil}. For~\itemref{smooth_manif}, $\GPR$ is a smooth manifold because it is a homogeneous space of a real Lie group. Each $\PRR_v^w$ is a smooth embedded manifold because it is the set of real points of a smooth algebraic subvariety $\PR_v^w$ of $G/P$; see~\cite[Corollary~3.2]{KLS} or~\cite{LusGP,Rie}. Part~\itemref{conn_comps} is due to Rietsch~\cite{Rie99}.
\end{proof}

\begin{corollary}\label{cor:GP_TNN_space}
$(\GPR,\GPtnn,Q_J)$ is a shellable TNN space.
\end{corollary}
\begin{proof}
Let us check each part of \cref{dfn:TNNspace}.

\axref{TNN:graded_max} and \axref{TNN:thin_shell}: These follow from \cref{thin_shellable}. The maximal element $\hat1\in Q_J$ is given by $(\id,\wj)$; see~\cref{sec:parab-subgr-w_j}.

\axref{TNN:smooth}: This follows from \cref{smooth_manif} and~\eqref{eq:GP_closure_intro}.

\axref{TNN:Ytnn_compact}: This holds since $\GPR$ is compact and $\PRtnn_v^w\subset G/P$ is closed.

\axref{TNN:Ytp}: This follows from \cref{conn_comps} combined with \cref{thm:GMR_param}.

\axref{TNN:Ytp_cl}: This result is due to Rietsch~\cite{Rie}; see~\eqref{eq:GPtnn_closure_intro}.
\end{proof}

\subsection{Gaussian decomposition}
Assume $\K$ is algebraically closed. Let us define 
\[\Gomp:=B_-B,\quad \Gopm:=BB_-.\]
For $i\in I$, let $\Deltamp_i:\Gomp\to\K$ and $\Deltapm_i:\Gopm\to\K$ be defined as follows. Given $(x_-,x_0,x_+)\in U_-\times T\times U$,  we have $x_-x_0x_+\in \Gomp$ and $x_+x_0x_-\in\Gopm$, and we set $\Deltamp_i(x_-x_0x_+):=x_0^{\omega_i}$, $\Deltapm_i(x_+x_0x_-):=x_0^{w_0\omega_i}$. For a finite set $A$, let $\Proj^A$ denote the $(|A|-1)$-dimensional projective space over $\K$, with coordinates indexed by elements of $A$.
\begin{lemma}\leavevmode
\begin{theoremlist}
\item \label{G_0_uniquely} The multiplication map gives biregular isomorphisms
\[  U_-\times T\times U\xrightarrow{\sim} \Gomp, \quad U\times T\times U_-\xrightarrow{\sim} \Gopm.\]
\item \label{gen_minors_regular} The maps $\Deltamp_i$ and $\Deltapm_i$ extend to regular functions $G\to \K$.
\item \label{lemma:G_0} $\Gomp=\{x\in G\mid \Deltamp_i(x)\neq0\ \text{ for all } i\in I\}$, $\Gopm=\{x\in G\mid \Deltapm_i(x)\neq0\ \text{ for all } i\in I\}.$
\item \label{gen_minors_flag} Fix $i\in I$ and let $W\omega_i:=\{w\omega_i\mid w\in W\}$ denote the $W$-orbit of the corresponding fundamental weight. Then there exists a regular map $\minormap_i:G/B\to \Proj^{W\omega_i}$ such that for $w\in W$ and $x\in G$, the $w\omega_i$-th coordinate of $\minormap_i(xB)$ equals $\Deltamp_i(\dw^{-1}x)$.
\end{theoremlist}
\end{lemma}
\begin{proof}
For~\itemref{G_0_uniquely}, see~\cite[Proposition~28.5]{Hum}. Parts~\itemref{gen_minors_regular} and~\itemref{lemma:G_0} are well known when $\K=\C$; see~\cite[Proposition~2.4 and Corollary~2.5]{FZ}. We give a proof for arbitrary algebraically closed $\K$, using a standard argument that relies on representation theory. We refer to~\cite[\S31]{Hum} for the necessary notation and background. 

We have $\Gopm=\dw_0^{-1} \Gomp \dw_0$ and $\Deltapm_i(\dw_0^{-1}g\dw_0)=\Deltamp_i(g)$ for all $g\in \Gomp$. Thus it suffices to give a proof for $\Deltamp_i$ and $\Gomp$. For $i\in I$, there exists a regular function $c_{\omega_i}:G\to \K$ that coincides with $\Deltamp_i$ on $\Gomp$; see~\cite[\S31.4]{Hum}. This shows~\itemref{gen_minors_regular}. Explicitly, $c_{\omega_i}$ is given as follows: consider the highest weight module $V_{\omega_i}$ for $G$, and let $v_+\in V_{\omega_i}$ be its highest weight vector. We have a direct sum of vector spaces $V_{\omega_i}=\K v_+\oplus V'$, where $V'$ is spanned by weight vectors of weights other than $\omega_i$. Letting $r^+:V_{\omega_i}\to \K$ denote the linear function such that $r^+(v_+)= 1$ and $r^+(V')=\{0\}$, we have $c_{\omega_i}(g):=r^+(gv_+)$ for all $g\in G$. The decomposition $V_{\omega_i}=\K v_+\oplus V'$ is such that for $(x_-,x_0,x_+)\in U_-\times T\times U$ and $w\in W$, we have $x_+v_+=v_+$, $x_0v_+=Mv_+$ for some $M\in\Kast$, $x_-v_+\in v_++V'$, $x_-V'\subset V'$, and $\dw v_+\in V'$ if $w\omega_i\neq\omega_i$.  Thus if $g\in \Gomp$ then $c_{\omega_i}(g)\neq0$ for all $i\in I$. 
Conversely, if $g\notin \Gomp$ then by~\eqref{eq:bruhat}, there exists a unique $w\neq \id\in W$ such that $g\in U_-\dw T U$. For $i\in I$ such that $w\omega_i\neq\omega_i$, we get $c_{\omega_i}(g)=0$. This proves~\itemref{lemma:G_0}. For~\itemref{gen_minors_flag}, let $V_{\omega_i}=V_1\oplus V_2$ where $V_1$ is spanned by all weight vectors of weights in $W\omega_i$, and $V_2$ is spanned by the remaining weight vectors. Let $\pi_1:V_{\omega_i}\to V_1$ denote the projection along $V_2$. It follows that for all $g\in G$, $\pi_1(gv_+)\neq0$. Then $\minormap_i$ is the natural morphism $G/B \to \Proj(V_1)$, sending $gB$ to $[\pi_1(gv_+)]$.
\end{proof}

\begin{lemma} Define  $\Goj:=P_-P$ (with notation as in \cref{sec:partial-flag-variety}).
  \begin{theoremlist}
\item \label{G_0^J}We have $\Goj=P_-B$ and $P=\bigsqcup_{r\in W_J}B \dr B$.
\item \label{unip_radical} For $p\in P$, we have $p\Uj p^{-1}= \Uj$. Similarly, for $p\in P_-$, we have $p\Uj_-p^{-1}=\Uj_-$. In particular, for $p\in L_J$, we have $p\Uj p^{-1}= \Uj$ and $p \Uj_-p^{-1}= \Uj_-$.
\item \label{lemma:bracketsJ} The multiplication map gives a biregular isomorphism $\Uj_-\times  L_J\times  \Uj\xrightarrow{\sim} \Goj$. In particular, every element $x\in \Goj$ can be uniquely factorized as $[x]_-\pj \cdot [x]_J\cdot  [x]_+\pj\in U_-\pj\cdot  L_J \cdot U\pj$. The map $\Goj\to L_J$ sending $x$ to $[x]_J$ satisfies $[p_-xp_+]_J=[p_-]_J[x]_J[p_+]_J$ for all $x\in \Goj$, $p_-\in P_-$, and $p_+\in P$.
\item \label{bracketsJ_homo} The map $b\mapsto [b]_J$ gives group homomorphisms $U\to U_J$ and $U_-\to U_J^-$, such that
\[x_i(t)\mapsto [x_i(t)]_J=
  \begin{cases}
    x_i(t), &\text{if $i\in J$,}\\
    1,&\text{otherwise,}
  \end{cases}\quad y_i(t)\mapsto [y_i(t)]_J=
  \begin{cases}
    y_i(t), &\text{if $i\in J$,}\\
    1,&\text{otherwise.}
  \end{cases}\]
\end{theoremlist}
\end{lemma}
\begin{proof}
By~\cite[\S30.2]{Hum}, $\Uj$ is the \emph{unipotent radical} (in particular, a normal subgroup) of $P$ and $\Uj_-$ is the unipotent radical of $P_-$. This shows~\itemref{unip_radical}. It follows that $P=L_J\Uj=L_JB$, and therefore $\Goj=P_-B$. By~\cite[\S30.1]{Hum} and~\eqref{eq:bruhat}, $P=\bigsqcup_{r\in W_J}B \dr B$, which proves~\itemref{G_0^J}.

By~\cite[Proposition~14.21(iii)]{Bor}, the multiplication map gives a biregular isomorphism $\Uj_-\times P\to \Goj$. By~\cite[\S30.2]{Hum}, the multiplication map gives a biregular isomorphism $L_J\times \Uj \to P$. Thus we get a biregular isomorphism $\Uj_-\times  L_J\times  \Uj\xrightarrow{\sim} \Goj$. It is clear from the definition that $[p_-xp_+]_J=[p_-]_J[x]_J[p_+]_J$, since we can factorize $p_-=[p_-]_-\pj [p_-]_J$ and $p_+=[p_+]_J[p_+]_+\pj$. Thus we are done with~\itemref{lemma:bracketsJ}, and~\itemref{bracketsJ_homo} follows by repeatedly applying~\itemref{lemma:bracketsJ}.
\end{proof}

\subsection{Affine charts}\label{sec:isomorphisms}
For $u\in W^J$, define $\Cuj:=\du \Goj/P\subset G/P$.  The following maps are biregular isomorphisms for $u\in W^J$ and $v,w\in W$ (see~\cite[Proposition~14.21(iii)]{Bor}, \cite[Proposition~8.5.1(ii)]{Spr}, and~\cite[Corollary~23.60]{FH}):
\begin{align}
\label{eq:Cuj_to_Uj}
  \du\Uj_-\du^{-1}&\xrightarrow{\sim} \Cuj,  &&\gj\mapsto \gj\du P,\\
\label{eq:X_v_coinv(w)}
\dv U_-\dv^{-1}\cap U_-&\xrightarrow{\sim} \X_v, &&g\mapsto g\dv B,\\
\label{eq:X^w_inv(w)}
\dw U_-\dw^{-1}\cap U&\xrightarrow{\sim} \X^w, &&g\mapsto g\dw B.
\end{align}
As a consequence of~\eqref{eq:X_v_coinv(w)} and~\eqref{eq:X^w_inv(w)}, we get
\begin{equation}\label{eq:BuB_uU_cap_Uu}
B_-\dv B=(\dv U_-\cap U_-\dv)\cdot B,\quad B\dw B=(\dw U_-\cap U\dw) \cdot B.
\end{equation}

The isomorphism in~\eqref{eq:Cuj_to_Uj} identifies an open dense subset $\Cuj$ of $G/P$ with the group $\du \Uj_-\du^{-1}$. We now combine this with \cref{lemma:KWY}.

\begin{definition}
\label{dfn:gj_and_stuff} Let $\Uj_1:=\du\Uj_-\du^{-1}\cap U$ and $\Uj_2:=\du\Uj_-\du^{-1}\cap U_-$. For $x\in\du\Goj$, consider the element $\gj\in\du\Uj_-\du^{-1}$ such that $\gj\du\in xP\cap \du\Uj_-$, which is unique by~\eqref{eq:Cuj_to_Uj}. Further, let $\hj_1,\gj_1\in\Uj_1$ and $\hj_2,\gj_2\in\Uj_2$ be the elements such that $\hj_2\gj=\gj_1$ and $\hj_1\gj=\gj_2$. By~\eqref{eq:Cuj_to_Uj}, the map $x\mapsto \gj$ is regular, and the map $\gj\mapsto (\gj_1,\gj_2,\hj_1,\hj_2)$ is regular by \cref{lemma:KWY}. Let us denote by $\hjmap:\du\Goj\to \Uj_2$ the map $x\mapsto \hjmp_x:=\hj_2$. It descends to a regular map $\hjmap:\Cuj\to \Uj_2$ sending $xP$ to $\hjmp_x$.
\end{definition}

\section{Subtraction-free parametrizations}\label{sec:subtraction-free}
We study subtraction-free analogs of Marsh--Rietsch parametrizations~\cite{MR} of $\GBtnn$.
\subsection{Subtraction-free subsets}
Given some fixed collection $\t$ of variables of size $|\t|$, let $\R[\t]$ be the ring of polynomials in $\t$, and $\R_{>0}[\t]\subset\R[\t]$ be the semiring of nonzero polynomials in $\t$ with positive real coefficients. Let $\Fcal:=\R(\t)$ be the field of rational functions in $\t$. Define
\[\Qsfs:=\{R(\t)/Q(\t)\mid R(\t),Q(\t)\in\R_{>0}[\t]\},\quad \Qsf:=\{0\}\sqcup\Qsfs,\] 
\[\Fcalo:=\{R(\t)/Q(\t)\mid R(\t)\in\R[\t],\ Q(\t)\in\R_{>0}[\t]\}.\]
We call elements of $\Qsf$ \emph{subtraction-free rational expressions in $\t$}. In this section, we assume that $\K=\Fcalb$ is the algebraic closure of $\Fcal$.

\begin{definition}
\label{dfn:Uo} Let $\Tsf\subset \TQT$ be the subgroup generated by $\alphacheck_i(t)$ for $i\in I$ and $t\in \Qsfs$. Let $\Go\subset G$ be the subgroup generated by 
\[\{x_i(t), y_i(t)\mid i\in I,\  t\in \Fcalo\}\cup\{\dw\mid w\in W\}\cup \Tsf.\]
We define subgroups $\Uo:=U\cap \Go$, $\Uo_-:=U_-\cap \Go$, $\Bsf:=\Tsf \Uo=\Uo\Tsf$ and $\Bsf_-:=\Tsf \Uo_-=\Uo_-\Tsf$ (cf.~\cref{Tsf_commutes} below). We also put $\Uo(\Theta):=\Uo\cap U(\Theta)$ (respectively, $\Uo_-(\Theta):=\Uo_-\cap U_-(\Theta)$) for a bracket closed subset $\Theta$ of $\Phi^+$ (respectively, of $\Phi^-$).  Given a reduced word $\i$ for $w\in W$, define
\begin{equation}\label{eq:Usf_dfn}
\Usf(w):=\{\bx_\bi(\t')\mid \t'\in\QsfsN n\},\quad \Usf^-(w):=\{\by_\bi(\t')\mid \t'\in\QsfsN n\}.
\end{equation}
\end{definition}
\noindent These subsets do not depend on the choice of $\i$; see~\cite[\S3]{BZ}.

For two subsets $H_1$ and $H_2$ of $G$, we say that $H_1$ \emph{commutes setwise with} $H_2$ if $H_1\cdot H_2=H_2\cdot H_1$. We say that $H_1$ \emph{commutes setwise with} $g\in G$ if $H_1\cdot g=g\cdot H_1$.

\begin{lemma}\label{Tsf_commutes} $\Tsf$ commutes setwise with $\Bsf$, $\UQT$, $\UQT_-$, $\Uo(\Theta)$, $\Uo_-(\Theta)$, $\Usf(w)$, $\Usf^-(w)$, and $\dw$.
\end{lemma}
\begin{proof}
It follows from~\eqref{eq:torus_conj} that $\Tsf$ commutes setwise with $\Bsf$, $\UQT$, $\UQT_-$, $\Usf(w)$, $\Usf^-(w)$, and $\dw$. For $\Uo(\Theta)$, $\Uo_-(\Theta)$, we use a generalization of~\eqref{eq:torus_conj}: for $\alpha\in\Phi^+$, $i\in I$, and $w\in W$ such that $w\alpha_i=\alpha$, write $x_\alpha(t):=\dw x_i(t)\dw^{-1}\in \Uo(\{\alpha\})$ and $y_\alpha(t):=\dw y_i(t)\dw^{-1}\in \Uo_-(\{-\alpha\})$ for $t\in\Fcalo$. Then~\eqref{eq:torus_conj} implies $ax_\alpha(t)a^{-1}=x_\alpha(a^{\alpha}t)$ and $ay_\alpha(t)a^{-1}=y_\alpha(a^{-\alpha}t)$.
\end{proof}

Let us now introduce subtraction-free analogs of Marsh--Rietsch parametrizations. Let $v\leq w\in W$ and $(\bv,\bw)\in\RedMR(v,w)$. Recall that for $\t'=(t'_k)_{k\in J_\bv^\circ}\in(\Kast)^{J_\bv^\circ}$,  $\gMRvw(\t')=g_1\cdots g_n$ is defined in~\eqref{eq:gMRvw}. Define $\GMRsf\bv\bw:=\{\gMRvw(\t')\mid \t'\in \QsfsN{J_\bv^\circ}\}\subset \Go$. The following result is closely related to~\cite[Lemma~11.8]{MR}.

\begin{lemma}
\label{lemma:MR_move_right} Let $v\leq w\in W$ and $(\bv,\bw)\in\RedMR(v,w)$. Let $\gMRvw(\t')$ be as in~\eqref{eq:gMRvw} for $\t'\in\QsfsN{J_\bv^\circ}$. Then for each $k=0,1,\dots,n$ and for all $x\in \Uo\cap \dv\pd {k}^{-1}U_-\dv\pd {k}$, we have
\begin{equation}\label{eq:MR_move_right}
g_1\cdots g_k \cdot x\cdot  g_{k+1}\cdots g_n\in g_1\cdots g_n \cdot \Uo.
\end{equation}
\end{lemma}

\begin{proof}
We prove this by induction on $k$. For $k=n$, the result is trivial, so suppose that $k<n$. Let $x\in \Uo\cap \dv\pd k^{-1}U_-\dv\pd k$. If $g_{k+1}= \ds_i$ for some $i\in I$ then $\ell(v\pd{k+1})=\ell(v\pd k)+\ell(s_i)$, so we use~\eqref{eq:Inv(ab)_commute2} to show that $x\cdot g_{k+1}=g_{k+1}\cdot x'$ for some $x'\in U\cap \dv\pd {k+1}^{-1}U_-\dv\pd {k+1}$. Since $x'=\ds_i^{-1} x\ds_i$ and each term belongs to $\Go$, we see that $x'\in \Uo\cap \dv\pd {k+1}^{-1}U_-\dv\pd {k+1}$, so we are done by induction.

Suppose now that $g_{k+1}=y_i(t)$ for some $i\in I$ and $t\in \Qsfs$. Write 
\[x\cdot g_{k+1}=g_{k+1}\cdot g_{k+1}^{-1}xg_{k+1}=g_{k+1}\cdot y_i(-t)xy_i(t).\] 
By~\eqref{eq:U(Inv)}, $\Uo\cap \dv\pd k^{-1}U_-\dv\pd k=\Uo(\Inv(v\pd k))$. Clearly again $y_i(-t)xy_i(t)\in\Go$, and we claim that $y_i(-t) xy_i(t)\in U(\Inv (v\pd k))$ for all $x\in U(\Inv (v\pd k))$. First, using \cref{U(alpha)_iso}, we can assume that $x\in U_\alpha$ for some $\alpha\in \Inv (v\pd k)$. Since $v\pd k s_i>v\pd k$, we have $\alpha_i\notin \Inv(v\pd k)$, so $\alpha\neq\alpha_i$. Let $\Psi=\{m\alpha-r\alpha_i\}\subset \Phi^+$ be the set of roots as in \cref{commute_alpha_beta}. Our goal is to show that $\Psi\subset \Inv(v\pd k)$. Let $\gamma:=m\alpha-r\alpha_i\in \Psi$ for some $m>0$ and $r\geq0$. We now show that $\gamma\in \Inv (v\pd k)$, which is equivalent to saying that $v\pd k \gamma<0$. Indeed, $v\pd k\gamma=mv\pd k\alpha-rv\pd k\alpha_i$. Since $\alpha\in \Inv(v\pd k)$,  $v\pd k\alpha<0$. Since $\alpha_i\notin \Inv(v\pd k)$,  $v\pd k\alpha_i>0$. Thus $v\pd k\gamma<0$, because $-v\pd k\gamma$ is a positive linear combination of positive roots. We have shown that $\Psi\subset \Inv(v\pd k)$, and thus by  \cref{commute_alpha_beta}, we find $y_i(-t) xy_i(t)\in U(\Inv(v\pd k))$. Since $v\pd k=v\pd{k+1}$, we get
\[y_i(-t) xy_i(t)\in \Uo(\Inv(v\pd k))=\Uo\cap \dv\pd k^{-1}U_-\dv\pd k=\Uo\cap \dv\pd {k+1}^{-1}U_-\dv\pd {k+1},\] 
and we are done by induction.
\end{proof}
\begin{proposition}
\label{prop:GMR_sf} For $v\leq w\in W$, the set $\GMRsf\bv\bw \cdot \Uo\subset \Go$ does not depend on the choice of $(\bv,\bw)\in\RedMR(v,w)$. In other words: let $(\bv_0,\bw_0),(\bv_1,\bw_1)\in\RedMR(v,w)$. Then for any $\t_0\in \QsfsN{J_{\bv_0}^\circ}$ there exist $\t_1\in   \QsfsN{J_{\bv_1}^\circ}$ and $x\in \Uo$ such that $\gMR{\bv_0}{\bw_0}(\t_0)=\gMR{\bv_1}{\bw_1}(\t_1)\cdot x$.
\end{proposition}

\begin{proof}
Recall that for each $\bw_0\in\Red(w)$ there exists a unique positive subexpression $\bv_0$ for $v$ such that $(\bv_0,\bw_0)\in\RedMR(v,w)$. We need to show that choosing a different reduced expression $\bw_1$ for $w$ results in a subtraction-free coordinate change $\t_0\mapsto \t_1$ of the parameters in \cref{thm:GMR_param}. Any two reduced expressions for $w$ are related by a sequence of braid moves, so it suffices to assume that $\bw_0$ and $\bw_1$ differ by a single braid move.

 The explicit formulae for the corresponding coordinate transformations can be found in the proof of~\cite[Proposition~7.2]{Rie2}; however, an extra step is needed to show that those formulae indeed give the correct coordinate transformations. More precisely, suppose that $\Phi'$ is a root subsystem of $\Phi$ of rank $2$, and let $W'$ be its Weyl group. Then it was checked in the proof of~\cite[Proposition~7.2]{Rie2} that for any $v'\leq w'\in W'$, any  $(\bv'_0,\bw'_0),(\bv'_1,\bw'_1)\in \RedMR(v',w')$, and any $\t'_0\in \QsfsN{J_{\bv'_0}^\circ}$, there exist $\t'_1\in   \QsfsN{J_{\bv'_1}^\circ}$ and $x\in U$ such that $\gMR{\bv'_0}{\bw'_0}(\t'_0)=\gMR{\bv'_1}{\bw'_1}(\t'_1)\cdot x$.

Let us now complete the proof of \cref{prop:GMR_sf} (as well as of~\cite[Proposition~7.2]{Rie2}).\footnote{Alternatively, the proof of~\cite[Proposition~7.2]{Rie2} can be completed using~\cite[Theorem~7.1]{MR}. We thank Konni Rietsch for pointing this out to us.} Suppose that $\bw_0$ and $\bw_1$ differ by a braid move along a subword $g_{k+1}\cdots g_{k+m}$ of $g_1\cdots g_n$. Here $g_{k+1}\cdots g_{k+m}=\gMR{\bv'_0}{\bw'_0}(\t'_0)$ as above. Applying a move from~\cite{Rie2}, we transform $g_{k+1}\cdots g_{k+m}$ into $g_{k+1}'\cdots g_{k+m}' x$ for some $x\in U$ and $g_{k+1}'\cdots g_{k+m}'=\gMR{\bv'_1}{\bw'_1}(\t'_1)$. Thus 
\[g_1\cdots g_n=g_1\cdots g_k\cdot g_{k+1}'\cdots g_{k+m}'\cdot x\cdot g_{k+m+1}\cdots g_n.\]
By~\cite[Proposition~5.2]{MR}, the elements $h:=g_1\cdots g_{k+m}$ and $h':=g_1\cdots g_k\cdot g_{k+1}'\cdots g_{k+m}'$ belong to $U_-\dv\pd{k+m}$. Since $h=h'x$, we get $x\in \dv\pd{k+m}^{-1}U_-\dv\pd{k+m}$. Moreover, since $h,h'\in \Go$ and $x\in U$, we must have $x\in \Uo$. Thus by \cref{lemma:MR_move_right}, we have 
\[g_1\cdots g_n\in g_1\cdots g_{k}\cdot g_{k+1}'\cdots g_{k+m}'\cdot g_{k+m+1}\cdots g_n\cdot \Uo.\qedhere\]
\end{proof}
\begin{definition}
\label{Rsf} From now on we denote $\Rsf_v^w:=\GMRsf\bv\bw \Bsf\subset \Go$.  By \cref{prop:GMR_sf}, the set $\Rsf_v^w$ does not depend on the choice of $(\bv,\bw)\in\RedMR(v,w)$. As we discuss in \cref{eval}, $\Rsf_v^w$ is the ``subtraction-free'' analog of $\Rtp_v^w$.
\end{definition}

\subsection{Collision moves} Assume $\K=\Fcalb$. By~\cite[(2.13)]{FZ}, for each $t\in \Qsfs$ there exist  $t_+\in\Qsfs$, $a_+\in \Tsf$, and $t_-\in \Fcalo$ satisfying 
\begin{equation}\label{eq:collision_x}
\ds_i x_i(t)=a_+x_i(t_-)y_i(t_+),\quad x_i(t)\ds_i=y_i(t_+)x_i(t_-)a_+,
\end{equation}
\begin{equation}\label{eq:collision_y}
\ds_i^{-1}y_i(t)=a_+y_i(t_-)x_i(t_+),\quad y_i(t)\ds_i^{-1}=x_i(t_+)y_i(t_-)a_+.
\end{equation}
(Here, each of the four moves yields different $t_+,a_+,t_-$.) By~\cite[(2.11)]{FZ}, for each $t,t'\in \Qsfs$ there exist $t_+,t'_+\in\Qsfs$ and $a_+\in \Tsf$ satisfying
\begin{equation}\label{eq:collision_xy}
x_i(t)y_i(t')=y_i(t'_+)x_i(t_+) a_+,\quad y_i(t')x_i(t)=x_i(t_+)y_i(t'_+) a_+.
\end{equation}
By~\cite[(2.9)]{FZ}, we have
\begin{equation}\label{eq:xy_commute_ineqj}
x_i(t)y_j(t')=y_j(t')x_i(t),\quad \text{for $i\neq j$}.
\end{equation}

As a direct consequence of~\eqref{eq:collision_xy}, \eqref{eq:xy_commute_ineqj}, and \cref{Tsf_commutes}, for any $v,w\in W$ we get
\begin{equation}\label{eq:Usf_commute}
\Usf(v)\cdot\Usf^-(w)\cdot\Tsf=\Usf^-(w)\cdot\Usf(v)\cdot\Tsf.
\end{equation}

\begin{lemma}\leavevmode
\begin{theoremlist}
\item Let $w\in W$. Then
\begin{align}\label{eq:collision1}
 \Bsf_- \cdot \dw^{-1} \cdot \Usf^-(w)=\Bsf_- \cdot \Usf(w^{-1})\quad\text{and}\quad \Usf^-(w) \cdot \dw^{-1} \cdot \Bsf_-=\Usf(w^{-1}) \cdot \Bsf_-.
\end{align}
\item If $v,w\in W$ are such that $\ell(vw)=\ell(v)+\ell(w)$, then 
\begin{equation}\label{eq:collision_length_add}
\dw^{-1}\dv^{-1} \cdot \Usf^-(v)\subset \Bsf_-\cdot \dw^{-1}\cdot  \Usf(v^{-1}).
\end{equation}
\item Let $w_1,\dots,w_k\in W$ be such that $\ell(w_1\cdots w_k)=\ell(w_1)+\dots+\ell(w_k)$. Then for any $h\in \Usf^-(w_1\cdots w_k)$ there exist $b_1\in \Usf(w_1^{-1}),\dots, b_k\in \Usf(w_k^{-1})$ such that for each $1\leq i\leq k$, we have 
\begin{equation}\label{eq:collision_multi}
\dw_i^{-1}\cdots \dw_1^{-1} \cdot h\in \Bsf_-\cdot b_i\cdots b_1.
\end{equation}

\item Let $v\leq w\in W$. Then 
\begin{equation}\label{eq:collision2}
\dv^{-1} \cdot \Usf^-(w)\subset \Bsf_-\cdot \Usf(v^{-1}).
\end{equation}
\end{theoremlist}
\end{lemma}
\begin{proof}
Let us prove the following claim: if $vv_1= w$ and $\ell(w)=\ell(v)+\ell(v_1)$, then
\begin{equation}\label{eq:collision_master}
\dv^{-1} \Usf^-(w)\subset \Tsf\cdot (\Uo_-\cap \dv^{-1} U \dv) \cdot \Usf^-(v_1)\cdot \Usf(v^{-1}).
\end{equation}
We prove this by induction on $\ell(v)$. If $\ell(v)=0$ then $v=\id$ and~\eqref{eq:collision_master} is trivial. Otherwise there exists an $i\in I$ such that $v':=s_iv<v$ and thus $w':=s_iw<w$. Let $\by_\bi(\t')\in \Usf^-(w)$. Using~\eqref{eq:collision_y}, we see that for some $t_1\in\Qsfs$, $t_+\in\Qsfs$ and $t_-\in \Fcalo$,
\[\dv^{-1} \cdot \by_\bi(\t')\in \dv'^{-1} \cdot \ds_i^{-1} y_i(t_1)\cdot \Usf^-(w')\subset \Tsf \dv'^{-1} \cdot y_i(t_-) x_i(t_+)\cdot \Usf^-(w').\]
By~\eqref{eq:Usf_commute}, $x_i(t_+)\cdot \Usf^-(w')\subset \Tsf \cdot \Usf^-(w')\cdot \Usf(s_i)$. Clearly $s_iv'>v'$, so $y':=\dv'^{-1}  y_i(t_-)\dv'\in U_-$. On the other hand,  $\dv y' \dv^{-1}= \ds_i^{-1}  y_i(t_-) \ds_i=x_i(-t_-)\in U$. Thus $y'\in U_-\cap \dv^{-1}U\dv$, and it is also clear that $y'\in\Go$. We have shown that 
\begin{equation}\label{eq:tmp_collision}
\dv^{-1} \cdot \by_\bi(\t')\in  \Tsf\cdot  y'\cdot  \dv'^{-1} \cdot \Usf^-(w')\cdot \Usf(s_i)\subset \Tsf \cdot  (\Uo_-\cap \dv^{-1}U\dv)\cdot  \dv'^{-1} \cdot \Usf^-(w')\cdot \Usf(s_i).
\end{equation}
We have $v'v_1= w'$, so by induction, 
\[\dv'^{-1}\cdot  \Usf^-(w')\subset \Tsf\cdot (\Uo_-\cap \dv'^{-1} U \dv') \cdot \Usf^-(v_1)\cdot \Usf(v'^{-1}). \]
Since $\Usf(v'^{-1})\cdot \Usf(s_i)=\Usf(v^{-1})$, we have shown that
\[\dv^{-1} \by_\bi(\t')\in  \Tsf \cdot  (\Uo_-\cap \dv^{-1}U\dv)\cdot(\Uo_-\cap \dv'^{-1}U\dv')\cdot  \Usf^-(v_1)\cdot \Usf(v^{-1}).\]
By~\eqref{eq:Inv(ab)_sqcup} applied to $a=s_i$, $b=v'$, $ab=v$, we get $\Inv(v')\subset \Inv(v)$, so $(\Uo_-\cap \dv'^{-1}U\dv')\subset (\Uo_-\cap \dv^{-1}U\dv)$, and we have finished the proof of~\eqref{eq:collision_master}.

Combining~\eqref{eq:collision_master} with~\eqref{eq:Inv(ab)_commute}, we obtain~\eqref{eq:collision_length_add}. 
Next, \eqref{eq:collision_multi} can be shown by induction: the case $k=0$ is trivial. For $k\geq1$, we can write $h=h_1\cdots h_k\in \Usf^-(w_1)\cdots\Usf^-(w_k)$. By~\eqref{eq:collision_length_add}, we have \[\dw_i^{-1}\cdots \dw_1^{-1}\cdot h_1\cdots h_k\in \Bsf_-\cdot \dw_i^{-1}\cdots \dw_2^{-1}\cdot b_1'\cdot h_2\cdots h_k\]
for some $b_1'\in\Usf(w_1)$ that does not depend on $i$. Using~\eqref{eq:Usf_commute}, we write $b_1'\cdot h_2\cdots h_k=h_2'\cdots h_k'\cdot b_1\in \Usf^-(w_2)\cdots \Usf^-(w_k)\cdot \Usf(w_1)$, and then proceed by induction.

Let us state several further corollaries of~\eqref{eq:collision_master}:
\begin{align}\label{eq:collision1_ref}
\dw^{-1} \cdot \Usf^-(w)&\subset  \Tsf\cdot (\Uo_-\cap \dw^{-1} U \dw) \cdot \Usf(w^{-1}),\\
\label{eq:collision1dual_ref}
\Usf^-(w)\cdot \dw^{-1}&\subset  \Usf(w^{-1})\cdot (\Uo_-\cap \dw U \dw^{-1}) \cdot\Tsf,\\
\label{eq:collision1+}
\dw\cdot \Usf(w^{-1})&\subset   (\Uo\cap \dw U_- \dw^{-1})\cdot \Usf^-(w)\cdot \Tsf.
\end{align}
Indeed, specializing~\eqref{eq:collision_master} to $v=w$, we obtain~\eqref{eq:collision1_ref}. We obtain~\eqref{eq:collision1dual_ref} from~\eqref{eq:collision1_ref} by replacing $w$ with $z:=w^{-1}$ and then applying the involution $x\mapsto x^\iota$ of~\eqref{eq:iota}, while~\eqref{eq:collision1+} is obtained from~\eqref{eq:collision1dual_ref} by applying the involution $x\mapsto x^T$ of~\eqref{eq:transpose}.

To show~\eqref{eq:collision1}, observe that the inclusion $ \Bsf_- \cdot \dw^{-1} \cdot \Usf^-(w)\subset\Bsf_- \cdot \Usf(w^{-1})$ follows from~\eqref{eq:collision1_ref}. To show the reverse inclusion, we use~\eqref{eq:collision1+} to write
\[\Bsf_- \cdot \Usf(w^{-1})=\Bsf_-\cdot \dw^{-1}\cdot \dw\cdot \Usf(w^{-1})\subset \Bsf_-\cdot \dw^{-1}\cdot (\Uo\cap \dw U_- \dw^{-1})\cdot \Usf^-(w).\]
Since $\dw^{-1}\cdot (\Uo\cap \dw U_- \dw^{-1})\subset \Uo_-\dw^{-1}$, we obtain $ \Bsf_- \cdot \dw^{-1} \cdot \Usf^-(w)=\Bsf_- \cdot \Usf(w^{-1})$, which is the first part of~\eqref{eq:collision1}. The second part follows by applying the involution $x\mapsto x^\iota$ of~\eqref{eq:iota}.

It remains to show~\eqref{eq:collision2}. We argue by induction on $\ell(w)$, and the base case $\ell(w)=0$ is clear. Suppose that $v\leq w$, and let $w':=s_iw<w$ for some $i\in I$. If $v':=s_iv<v$ then by the same argument as in the proof of~\eqref{eq:tmp_collision}, we get 
\[\dv^{-1}\cdot  \Usf^-(w)\subset  \Bsf_-\cdot  \dv'^{-1} \cdot \Usf^-(w')\cdot \Usf(s_i).\]
Since $v'\leq w'$, we can apply the induction hypothesis to write $\dv'^{-1} \cdot \Usf^-(w')\subset \Bsf_-\cdot \Usf(v'^{-1})$. We thus obtain 
\[\dv^{-1}\cdot  \Usf^-(w)\subset  \Bsf_- \cdot \Usf(v'^{-1})\cdot \Usf(s_i)= \Bsf_-\cdot \Usf(v^{-1}),\]
finishing the induction step in the case $s_iv<v$. But if $s_iv>v$ then $\dv^{-1}y_i(t_1)\dv\in \Uo_-$, so in this case we have $\dv^{-1} \Usf^-(w)\subset \Uo_-\cdot \dv^{-1} \cdot \Usf^-(w')$, and the result follows by applying the induction hypothesis to the pair $v\leq w'$.
\end{proof}

\subsection{Alternative parametrizations for the top cell}
The following two lemmas are subtraction-free versions of~\cite[Lemmas~4.2 and~4.3]{Rie}.
\begin{lemma}
\label{lemma:switch_MR_w0} Let $v\in W$. Then we have
\[\Rsf_v^{w_0}=\Usf(vw_0)\cdot\dw_0\cdot\Bsf.\]
\end{lemma}
\begin{proof}
Recall from \cref{Rsf} that $\Rsf_v^w=\GMRsf\bv\bw \cdot\Bsf$. We have $w=w_0$, so choose a reduced expression $\bw_0$ for $w_0$ that ends with $v$. With this choice, $\GMRsf\bv{\bw_0}=\Usf^-(w_0v^{-1}) \cdot\dv$. Thus we can write
\[\Rsf_v^{w_0}=\GMRsf\bv{\bw_0}\cdot \Bsf=\Usf^-(w_0v^{-1})\cdot\dv \cdot\Bsf=\Usf^-(w_0v^{-1})\cdot\dv \dw_0^{-1}\cdot\dw_0 \cdot \Bsf.\]
Let $z:=w_0v^{-1}$. Using~\eqref{eq:collision1} and $\Bsf_-\cdot\dw_0=\dw_0\cdot\Bsf$, we have
\[\Usf^-(w_0v^{-1})\cdot\dv \dw_0^{-1}\cdot\dw_0\cdot\Bsf=\Usf^-(z)\cdot\dz^{-1}\cdot\dw_0\cdot\Bsf= \Usf(z^{-1}) \cdot\dw_0\cdot\Bsf.\]
Combining the above equations, we find $\Rsf_v^{w_0}= \Usf(z^{-1})\cdot\dw_0\cdot\Bsf$, and it remains to note that $z^{-1}=vw_0^{-1}=vw_0$.
\end{proof}

\begin{lemma}
\label{lemma:param_sf} Let $v\leq w\in W$. Then we have 
\begin{equation}\label{eq:param_sf}
\Usf(v^{-1})\cdot \Usf^-(w_0w^{-1})\cdot \Rsf_v^w=\Rsf_\id^{w_0}=\Usf^-(w_0) \cdot\Bsf.
\end{equation}
\end{lemma}

\begin{proof}
  It follows from the definition of $\GMRsf\bv\bw$ that if $w' w$ is length-additive then $\Usf^-(w')\Rsf_v^w=\Rsf_v^{w'w}$. Applying this to $w'=w_0w^{-1}$, we get $\Usf^-(w_0w^{-1})\cdot \Rsf_v^w=\Rsf_v^{w_0}$. By \cref{lemma:switch_MR_w0}, we have $\Rsf_v^{w_0}\cdot\Bsf=\Usf(vw_0)\cdot\dw_0\cdot\Bsf$. Thus $\Usf(v^{-1})\cdot \Usf(vw_0)\cdot\dw_0\cdot\Bsf=\Usf(w_0)\cdot\dw_0\cdot\Bsf$, so applying \cref{lemma:switch_MR_w0} again, we find $\Usf(w_0)\cdot\dw_0\cdot\Bsf=\Rsf_{\id}^{w_0}\cdot\Bsf$. The result follows since $\Rsf_{\id}^{w_0}=\Usf^-(w_0)\cdot \Bsf$.
\end{proof}

\subsection{Evaluation}
\label{eval} We explain the relationship between $\Rsf_v^w$ and $\Rtp_v^w$. 
Given $\t'\in\R_{>0}^{|\t|}$, we denote by $\eval_{\t'}:\Qsf\to \R_{>0}$ the evaluation homomorphism (of semifields) sending $f(\t)$ to $f(\t')$. It extends to a well-defined group homomorphism $\eval_{\t'}:\Go\to \GR$, and it follows from \cref{thm:GMR_param} that $\{\eval_{\t'}(g)B\mid g\in \Rsf_v^w\}=\Rtp_v^w$ as subsets of $\GBR$. It is clear that the following diagram is commutative.
  \begin{equation}\label{eq:eval_commute}
\begin{tikzcd}[row sep=15pt]
\Fcal \ar[d,"\eval_{\t'}"',dashed]&    \Go \ar[r, "\Deltapm_i"] \ar[l, "\Deltamp_i"'] \ar[d,"\eval_{\t'}"]& \Fcal \ar[d,"\eval_{\t'}",dashed]
    \\
\R  &  \GR \ar[r,"\Deltapm_i"'] \ar[l,"\Deltamp_i"]  & \R
\end{tikzcd}
  \end{equation}
Here  solid arrows denote regular maps, and dashed arrows denote maps defined on a subset $\Fcal'\subset \Fcal$ given by $\Fcal':=\{R(\t)/Q(\t)\mid R(\t),Q(\t)\in\R[\t],\ Q(\t')\neq0\}.$  Since the diagram~\eqref{eq:eval_commute} is commutative, it follows that the images $\Deltamp_i(\Go)$ and $\Deltapm_i(\Go)$ belong to $\Fcal'$.

Let $\t=(\t',\t'')$. Observe that any $f(\t',\t'')\in\Qsfs$ gives rise to a continuous function $\R_{>0}^{|\t'|}\times \R_{>0}^{|\t''|}\to \R_{>0}$. Moreover, if sending $\t''\to0$ in $f(\t',\t'')$ gives rise to a well-defined subtraction-free rational expression, then $f(\t',\t'')$ extends to a continuous function  $\R_{>0}^{|\t'|}\times \R_{\geq0}^{|\t''|}\to\R_{\geq0}$. Surprisingly, the converse is also true, as our next result shows.

\begin{lemma}
\label{lemma:SF_limit} Suppose that $f(\t',\t'')\in\Qsfs$ is such that the corresponding function $\R_{>0}^{|\t'|}\times \R_{>0}^{|\t''|}\to \R_{>0}$ extends to a continuous function  $\R_{>0}^{|\t'|}\times \R_{\geq0}^{|\t''|}\to\R_{\geq0}$. Then $\lim_{\t''\to 0}f(\t',\t'')$ can be represented (as a function $\R_{>0}^{|\t'|}\to\R_{\geq0}$) by a subtraction-free rational expression in $\t'$.
\end{lemma}

\begin{proof}
By induction, it is enough to prove this when $|\t''|=1$, where $\t''=t''$ is a single variable. In this case, $f(\t',t'')=R(\t',t'')/Q(\t',t'')$ where $R$ and $Q$ have positive coefficients. Let us consider $R$ and $Q$ as polynomials in $t''$ only. After dividing $R$ and $Q$ by $(t'')^k$ for some $k$, we may assume that one of them is not divisible by $t''$. Then $Q$ cannot be divisible by $t''$, since otherwise $f$ would not give rise to a continuous function $\R_{>0}^{|\t'|}\times \R_{\geq0}^{|\t''|}\to\R_{\geq0}$. We can write $Q(\t',t'')=Q_1(\t',t'')t''+Q_2(\t')$ and $R(\t',t'')=R_1(\t',t'')t''+R_2(\t')$, where $R_1,R_2,Q_1,Q_2$ are polynomials with nonnegative coefficients and $Q_2(\t')\neq0$. Thus $\lim_{t''\to 0}f(\t',t'')$ can be represented by $R_2(\t')/Q_2(\t')$, which is a subtraction-free rational expression in $\t'$.
\end{proof}

\begin{lemma}\label{lemma:Delta_nonzero}
(Assume $\K=\C$.) Suppose that $a\leq b\leq c\in W$. Then $\Deltamp(\db^{-1} x)\neq 0$ for some $x\in \GR$ such that $xB\in \Rtp_a^c$.
\end{lemma}
\begin{proof}
Suppose that $\Deltamp(\db^{-1} x)=0$ for all $x\in\GR$ such that $xB\in\Rtp_a^c$. Consider the map $\minormap_i:G/B\to \Proj^{W\omega_i}$ from \cref{gen_minors_flag}. We get that the $b\omega_i$-th coordinate of $\minormap_i$ is identically zero on $\Rtp_a^c$. Therefore it must be zero on the Zariski closure of $\Rtp_a^c$ inside $G/B$, which is $\Richcl_a^c$. By~\eqref{eq:GB_closure}, $\Richcl_a^c$ contains $\db B=\Rich_b^b$, and thus $\Deltamp_i(\db^{-1} \db)$ must be zero. We get a contradiction since by definition $\Deltamp_i(\db^{-1} \db)=1$.
\end{proof}

\subsection{Applications to the flag variety}
We use the machinery developed in the previous sections to obtain some natural statements about $(G/B)_{\geq0}$. 

\begin{lemma}
\label{lemma:01sf} (Assume $\K=\Fcalb$.) Suppose that $a\leq c\in W$ and $b\in W$. Then for any $x\in\Rsf_a^c$ and $i\in I$,
\begin{equation}\label{eq:lemma_01_sf}
\Deltamp_i(\db^{-1}x)\in\Qsf.
\end{equation} 
 Moreover, if $a\leq b\leq c$ then 
\begin{equation}\label{eq:lemma_01_sfs}
\Deltamp_i(\db^{-1}x)\in\Qsfs,\quad \text{and}\quad x\in \db B_- B.
\end{equation}
\end{lemma}

\begin{proof}
  Let $\t=(\t_1,\t_2,\t_3)$ with $|\t_1|=\ell(a)$, $|\t_2|=\ell(w_0)-\ell(c)$, $|\t_3|=\ell(c)-\ell(a)$. Choose reduced words $\bi$ for $a^{-1}$ and $\bj$ for $w_0c^{-1}$, and let $(\ba,\bc)\in \RedMR(a,c)$. Suppose that $x\in\gMR\ba\bc(\t_3)\Bsf$ and let 
\[g:=\bx_\bi(\t_1)\cdot \by_\bj(\t_2)\cdot \gMR\ba\bc(\t_3)\in \Usf(a^{-1})\cdot \Usf^-(w_0c^{-1})\cdot \Rsf_a^c.\]
By \cref{lemma:param_sf}, $g\in \Usf^-(w_0)\cdot\Bsf=\Usf^-(b)\cdot\Usf^-(b^{-1}w_0)\cdot\Bsf$. By~\eqref{eq:collision1}, we have $\db^{-1} \cdot\Usf^-(b)\subset\Bsf_-\cdot\Usf(b^{-1})$. Therefore
\[\db^{-1}g\in \Bsf_-\cdot\Usf(b^{-1})\cdot\Usf^-(b^{-1}w_0)\cdot\Bsf.\]
By~\eqref{eq:Usf_commute}, we get $\db^{-1}g\in\Bsf_-\cdot\Usf^-(b^{-1}w_0)\cdot\Usf(b^{-1})\cdot\Bsf=\Bsf_-\cdot\Bsf$, and by definition, $\Deltamp_i(y)\in\Qsfs$ for any $y\in \Bsf_-\cdot\Bsf$. Since $\Deltamp_i$ is a regular function on $G$ by \cref{gen_minors_regular}, the function $f(\t_1,\t_2,\t_3):=\Deltamp_i(\db^{-1}g)\in\Qsfs$ extends to a continuous function on $\R_{\geq0}^{|\t_1|}\times \R_{\geq0}^{|\t_2|}\times \R_{>0}^{|\t_3|}$. Therefore by \cref{lemma:SF_limit}, $\lim_{\t_1,\t_2\to0}f(\t_1,\t_2,\t_3)$ is a subtraction-free rational expression in $\t_3$. Since $\lim_{\t_1,\t_2\to0}g=\gMR\ba\bc(\t_3)$, we get that $\Deltamp_i(\db^{-1}\gMR\ba\bc(\t_3))\in\Qsf$. Since $x\in \gMR\ba\bc(\t_3)\Bsf$, \eqref{eq:lemma_01_sf} follows.

Suppose now that $a\leq b\leq c$. We would like to show~\eqref{eq:lemma_01_sfs}, so assume that for some  $i\in I$ and $x\in\Rsf_a^c$, we have $\Deltamp_i(\db^{-1}x)=0$. Let $\t'\in (\Qsfs)^{|\t|}$ and $(\ba,\bc)\in\RedMR(a,c)$ be such that $x\in \gMR\ba\bc(\t') \Bsf$, and let $y(\t):=\gMR\ba\bc(\t)$. Then we have $\Deltamp_i(\db^{-1} y(\t))\in \Qsf$ by~\eqref{eq:lemma_01_sf}. If $\Deltamp_i(\db^{-1} y(\t))$ were a nonzero rational function in $\t$ then clearly substituting $\t\mapsto \t'$ for $\t'\in(\Qsfs)^{|\t|}$ would also produce a nonzero rational function. Since substituting $\t\mapsto \t'$ yields $\Deltamp_i(\db^{-1}x)=0$, we must have $\Deltamp_i(\db^{-1} y(\t))=0$. Therefore $\Deltamp_i(\db^{-1} x')=0$ for all $x'\in \Rsf_a^c$. 

Now let $\t'\in \R_{>0}^{|\t|}$. Recall from \cref{eval} that the image of $\Rsf_a^c$ in $\GBR$ under the map $\eval_{\t'}$ equals $\Rtp_a^c$. Thus by~\eqref{eq:eval_commute}, $\Deltamp_i(\db^{-1}x')=0$ for all $x'\in \GR$ such that $x'B\in\Rtp_a^c$, which contradicts \cref{lemma:Delta_nonzero}. Hence $\Deltamp_i(\db^{-1}x)\in\Qsfs$, and therefore $x\in \db B_- B$ follows from \cref{lemma:G_0}, finishing the proof of~\eqref{eq:lemma_01_sfs}.
\end{proof}

\begin{corollary}
\label{lemma:01} (Assume $\K=\C$.) Suppose that $a\leq c\in W$ and $b\in W$. Then for any $(\ba,\bc)\in\RedMR(a,c)$ and $\t'\in \R_{>0}^{J_\ba^\circ}$, we have
\begin{equation}\label{eq:lemma_01_>=0}
\Deltamp_i(\db^{-1} \gMR\ba\bc(\t'))\geq 0.
\end{equation}
Moreover, if $a\leq b\leq c$ then 
\begin{equation}\label{eq:lemma_01_>0}
\Deltamp_i(\db^{-1}\gMR\ba\bc(\t'))> 0,\quad \text{and}\quad \Rtp_a^c\subset \db B_-B/B.
\end{equation}
\end{corollary}

\begin{proof}
By~\eqref{eq:lemma_01_sf}, we know that $\Deltamp_i(\db^{-1} \gMR\ba\bc(\t))\in\Qsf$ for all $i\in I$. Evaluating at $\t=\t'$ (cf. \cref{eval}), we find that $\Deltamp_i(\db^{-1} \gMR\ba\bc(\t'))\geq0$ for all $i\in I$, showing~\eqref{eq:lemma_01_>=0}. Similarly,~\eqref{eq:lemma_01_>0} follows from~\eqref{eq:lemma_01_sfs}.
\end{proof}

\begin{proposition} (Assume $\K=\Fcalb$.)
\label{gen_minors_sf} For all $v,w,v',w'\in W$ and $x\in \Usf(v')\cdot \Tsf\cdot \Usf^-(w')$, we have $\Deltapm_i(\dv x \dw^{-1})\in\Qsf$.
\end{proposition}
\begin{proof}
Let $\t=(\t_1,\t_2,\t_1',\t_2')$ with $|\t_1|=\ell(v')$, $|\t_2|=\ell(w')$, $|\t_1'|=\ell(w_0)-\ell(v')$, and $|\t_2'|=\ell(w_0)-\ell(w')$. Let $\t_v:=(\t_1',\t_1)$ and $\t_w:=(\t_2,\t_2')$. Choose reduced words $\bi,\bj$ for $w_0$ such that $\bi$ ends with a reduced word for $v'$ and $\bj$ starts with a reduced word for $w'$. Set $g=g(\t_1,\t_2,\t_v,\t_w):=\bx_\bi(\t_v)\cdot a\cdot \by_\bj(\t_w)$ for some arbitrary element $a\in\Tsf$. We get
\[\dv g\dw^{-1}\in \dv \cdot \Usf(w_0)\cdot \Tsf\cdot \Usf^-(w_0)\cdot \dw^{-1}\subset\dv \cdot \Usf(v^{-1})\cdot \Usf(vw_0)\cdot \Tsf\cdot \Usf^-(w_0w^{-1})\cdot \Usf^-(w)\cdot \dw^{-1}.\]
By~\eqref{eq:collision1+}, \eqref{eq:Usf_commute}, and~\eqref{eq:collision1}, we get $\dv g\dw^{-1}\in \Bsf\cdot \Usf^-(v)\cdot \Usf(w^{-1})\cdot \Bsf_-$. By~\eqref{eq:Usf_commute}, we can permute $\Usf^-(v)$ and $\Usf(w^{-1})$, showing $\dv g\dw^{-1}\in \Bsf\cdot \Bsf_-$. Thus $\Deltapm_i(\dv g\dw^{-1})\in\Qsfs$. It gives rise to a continuous function on $\R_{>0}^{|\t_1|}\times\R_{>0}^{|\t_2|}\times\R_{\geq0}^{|\t_1'|}\times\R_{\geq0}^{|\t_2'|}$, so sending $\t_1',\t_2'\to0$ via \cref{lemma:SF_limit} and varying $\t_1$, $\t_2$, and $a$, we get  $\Deltapm_i(\dv x \dw^{-1})\in\Qsf$ for all $x\in \Usf(v')\cdot \Tsf\cdot \Usf^-(w')$.
\end{proof}

\section{Bruhat projections and total positivity}\label{sec:KWY_TP}

In this section, we prove a technical result (\cref{thm:zeta}) which later will be used to finish the proof of \cref{thm:main_intro2}. Assume $\K$ is algebraically closed and fix $u\in W^J$.

\subsection{The map \texorpdfstring{$\zetamap$}{zeta}}\label{sec:zeta}
Retain the notation from \cref{dfn:gj_and_stuff}. Given $v\in W$ and $u\in W^J$, let us introduce a subset 
\begin{equation}\label{eq:Guvbig}
\Guvbig:=\{x\in \du\Goj\mid \hjmp_x x\in  \dv\Goj\}\subset G.
\end{equation}
Note that if $x\in \Guvbig$ then $xP\subset \Guvbig$; see \cref{hjmap_properties} below.

\begin{definition}
\label{lots_of_maps} 
Define a map $\Ftmap:\Guvbig\to L_J$ sending $x\in\Guvbig$ to $\Ftmap(x):=[\dv^{-1}\hjmp_x x]_J$. Also define a map $\pidup:\du\Goj\to \du P_-$ sending $x\in\du\Goj$ to the unique element $\pidup(x)\in \du P_-\cap x\Uj$. Explicitly (cf. \cref{lemma:bracketsJ}), we put 
\begin{equation}\label{eq:pidup_explicit}
\pidup(x):=\du [\du^{-1}x]_-\pj [\du^{-1}x]_J=x\cdot  ([\du^{-1}x]_+\pj)^{-1}.
\end{equation}
Finally, define $\zetamap:\Guvbig\to G$ by $\zetamap(x):=\pidup(x) \cdot \Ftmap(x)^{-1}$.
\end{definition}

\begin{lemma}\leavevmode
\begin{theoremlist}
\item\label{hjmap_properties_reg}\label{pidup_regular}The maps $\hjmap$ and $\pidup$ are regular on $\du\Goj$. 
\item\label{Ftmap_regular}\label{zetamap_regular} The maps $\Ftmap$ and $\zetamap$ are regular on $\Guvbig\subset \du\Goj$.
\item\label{hjmap_properties} If $x\in \du\Goj$ and $x'\in xP$ then $\hjmp_{x'}=\hjmp_x$.
\item \label{zeta_mod_P} If $x\in \Guvbig$ and $x'\in xP$ then $\zetamap(x)=\zetamap(x')$.
\end{theoremlist}
\end{lemma}
\begin{proof}
  Parts~\itemref{hjmap_properties_reg} and~\itemref{Ftmap_regular} are clear since each map is a composition of regular maps. Part~\itemref{hjmap_properties} follows from \cref{dfn:gj_and_stuff}, since by construction the map $\hjmap$ starts by applying the isomorphism in~\eqref{eq:Cuj_to_Uj}, which gives a regular map $\Cuj\to \du \Uj_-\du^{-1}$. To prove~\itemref{zeta_mod_P}, suppose that  $x\in \Guvbig$ and $x'\in xP$ is given by $x'=xp$ for $p\in P$. Then $\pidup(x')=\pidup(x)[p]_J$ by \cref{lemma:bracketsJ}. By~\itemref{hjmap_properties}, $\hjmp_{x'}=\hjmp_x$, and $\Ftmap(x')=[\dv^{-1}\hjmp_{x'} x']_J=[\dv^{-1}\hjmp_x x]_J[p]_J=\Ftmap(x)[p]_J$. Thus
\begin{align*}
\zetamap(x')&=\pidup(x')\cdot \Ftmap(x')^{-1}=\pidup(x)[p]_J\cdot [p]_J^{-1}\Ftmap(x)^{-1}=\zetamap(x).\qedhere
\end{align*}
\end{proof}

\begin{lemma}Let $x\in\du P_-$.
\begin{theoremlist}
\item\label{pidup_x=x} We have $\pidup(x)=x$. 
\item \label{zeta_inside_uP_-} If $x\in  \Guvbig$ then $\zetamap(x)=x \Ftmap(x)^{-1}$.
\end{theoremlist}
\end{lemma}
\begin{proof}
Both parts follow from \cref{lots_of_maps}.
\end{proof}

The ultimate goal of this section is to prove the following result.
\begin{theorem}\label{thm:zeta}
(Assume $\K=\C$.) Let $(u,u)\leqJ (v,w)\leqJ (v',w')\in Q_J$ and $x\in G$ be such that $xB\in \Rtp_{v'}^{w'}$. Then $x\in \Guvbig$ and $\zetamap(x)\in BB_-\dw$.
\end{theorem}

\subsection{Properties of \texorpdfstring{$\hjmap$}{kappa}}
We further investigate the element $\hjx x$. Denote $\ut:=u\woj\in W^J_\maxx$. 
\begin{lemma}
\label{lemma:U1_U2} 
The groups $\Uj$, $\Uj_1$, and $\Uj_2$ from \cref{dfn:gj_and_stuff} satisfy
\begin{align}\label{eq:U1_U2:Uj}
\du\Uj_- \du^{-1}&=\dut \Uj_- \dut^{-1},\\\label{eq:U1_U2:Uj1}
\Uj_1&=\du \Uj_- \du^{-1}\cap U=\du U_-\du^{-1}\cap U,\\ \label{eq:U1_U2:Uj2} 
\Uj_2&=\du\Uj_- \du^{-1}\cap U_-=\dut U_-\dut^{-1}\cap U_-.
\end{align}
\end{lemma}

\begin{proof}
By \cref{unip_radical}, we see that $\dwoj \Uj_-\dwoj^{-1}=\Uj_-$, which shows~\eqref{eq:U1_U2:Uj}. For~\eqref{eq:U1_U2:Uj1}, $\Uj_1=\du \Uj_- \du^{-1}\cap U$ by definition. By \cref{Inv_W^J}, we have $\du U_J^-\du^{-1}\subset U_-$, so~\eqref{eq:U1_U2:Uj1} follows from~\eqref{eq:U(Inv)}. For~\eqref{eq:U1_U2:Uj2}, observe that $\woj\Phi_J^+=\Phi_J^-$, so $\ut \Phi_J^+\subset \Phi^-$ by~\eqref{eq:Inv(ab)_sqcup}. We thus have $\dut U_-\dut^{-1}=(\dut U_J^-\dut^{-1})\cdot (\dut \Uj_-\dut^{-1})$ where $(\dut U_J^-\dut^{-1})\subset U$, and hence $\dut U_-\dut^{-1}\cap U_-=\dut \Uj_-\dut^{-1}\cap U_-=\Uj_2$ by the definition of $\Uj_2$.
\end{proof}

\begin{lemma}
\label{hj_2:alt:Uj1up} For $x\in\du\Goj$, there exists a unique element $h\in \Uj_2$  such that $hx\in \Uj_1 \du P$, and we have $h=\hjx$.
\end{lemma}

\begin{proof}
Let $\gj\in\Uj$ and $p\in P$ be such that $\gj \du= xp$. We first show that such an $h\in\Uj_2$ exists. By \cref{dfn:gj_and_stuff}, $\hjx$ is an element of $\Uj_2$ such that $\hjx\gj\in\Uj_1$. In particular, $\hjx x=\hjx\gj \du p^{-1}\in \Uj_1 \du P$, which shows existence. To show uniqueness, observe that the action of $\du \Uj_- \du^{-1}$ on $\du \Goj/P \subset G/P$ is free by~\eqref{eq:Cuj_to_Uj}, and in particular the action of $\Uj_2$ is also free.
\end{proof}

\begin{lemma}
\label{lemma:hj_2=1_for_ut} If $x\in \du \Goj\cap B \du \dr B$ for some $r\in W_J$, then $\hjx=1$.
\end{lemma}

\begin{proof}
By \cref{hj_2:alt:Uj1up}, it suffices to show that $B\du\dr B\subset \Uj_1 uP$. Write 
\[B\du\dr B\subset B\du P\subset (B\du B)\cdot P.\]
By~\eqref{eq:BuB_uU_cap_Uu}, $B\du B\subset (\du U_-\cap U\du)\cdot B$, and therefore we find 
\[B\du\dr B\subset (\du U_-\cap U\du)\cdot P=(\du U_-\du^{-1}\cap U)\du P=\Uj_1\du P,\]
where the last equality follows from~\eqref{eq:U1_U2:Uj1}.
\end{proof}

\def\rhocheck{\rho^\vee}
\def\rhow{\rho_w^\vee}
\begin{lemma}\label{lemma:torus} Let $a\in T$.
\begin{theoremlist}
\item \label{torus:conjug} The subgroups $\du \Uj \du^{-1}$, $\Uj_1$, and $\Uj_2$ are preserved under conjugation by $a$.
\item \label{torus:kappa} If $x\in \du\Goj$, then $ax\in \du\Goj$ and $\hjmp_{ax} ax=a\hjx x$.
\item \label{torus:closure} (Assume $\K=\C$.) For each $w\in W$, there exists $\rhow\in Y(T)$ such that for all $x\in \dw B_-B$, $\lim_{t\to0}\rhow(t)\cdot xB=\dw B$ in $G/B$. If $w\in W^J$, then for all $x\in \dw \Goj$, $\lim_{t\to0}\rhow(t)\cdot xP=\dw P$ in $G/P$.
\end{theoremlist}
\end{lemma}
\begin{proof}
Since $\du\in N_G(T)$, there exists $b\in T$ such that $a\du=\du b$. Thus $a\du \Uj \du^{-1}a^{-1}=\du b\Uj b^{-1}\du^{-1}=\du\Uj\du^{-1}$, which shows~\itemref{torus:conjug}, and~\itemref{torus:kappa} is a simple consequence of~\itemref{torus:conjug}. To show~\itemref{torus:closure}, assume $\K=\C$ and choose $\rhocheck\in Y(T)$ such that $\<\rhocheck,\alpha_i\><0$ for all $i\in I$. Then $\lim_{t\to 0}\rhocheck(t) y \rhocheck(t)^{-1}=1$ for all $y\in U_-$, and in particular for all $y\in \Uj_-$. Set $\rhow:=w^{-1}\rhocheck$, so that for $t\in \Cast$, $\rhow(t)=\dw \rhocheck(t)\dw^{-1}$ by~\eqref{eq:torus_conj}. Every $x\in \dw B_-B$ belongs to $\dw yB$ for some $y\in U_-$, so $\rhow(t)\cdot x\cdot B=\dw \rhocheck(t) y\rhocheck(t)^{-1}\cdot B\to\dw B$ as $t\to 0$. Similarly, if $w\in W^J$ then every $x\in \dw \Goj$ belongs to $\dw yP$ for some $y\in \Uj_-$ by~\eqref{eq:Cuj_to_Uj}, so $\rhow(t)\cdot xP\to \dw P$ as $t\to0$.
\end{proof}

\def\rhou{\rho_u^\vee}
\def\rhour{\rho_{ur}^\vee}
\begin{lemma}
\label{lemma:hj_2_x_cell_refined}
Suppose that $v''\leq ur\leq w''$ for some $v'',w''\in W$ and $r\in W_J$, and let $x\in G$. 
\begin{theoremlist}
\item \label{lemma:subset_uP_sf} (Assume $\K=\Fcalb$.) If $x\in \Rsf_{v''}^{w''}$, then $x\in \du\Goj$.
\item \label{lemma:subset_uP} \label{lemma:hj_2_x_cell_r_w} (Assume $\K=\C$.) If $xB\in \Rtp_{v''}^{w''}$, then $x\in \du\Goj$ and  $\hjx xB\in \Rtp_{v''}^{ur_w}$ for some $r_w\in W_J$ such that $r_w\geq r$.
\end{theoremlist}
\end{lemma}
\begin{proof}
When $\K=\Fcalb$, \eqref{eq:lemma_01_sfs} implies $\Rsf_{v''}^{w''}\subset \du\dr B_-B\subset \du P_-B$, and by \cref{G_0^J}, $P_-B=\Goj$, which shows~\itemref{lemma:subset_uP_sf}. Similarly  (for $\K=\C$),  by \cref{lemma:01},  we have $x\in \du\dr B_-B$ for any $x\in \Rtp_{v''}^{w''}$, so $\Rtp_{v''}^{w''}\subset \du\Goj$.

Assume now that $\K=\C$ and $xB\in \Rtp_{v''}^{w''}$. Let $p\in P$ and $\gj\in\du\Uj_-\du^{-1}$ be such that $xp=\gj \du$. Then $\hjx xp=\gj_1 \du$ for $\gj_1\in \Uj_1$. By~\eqref{eq:U1_U2:Uj1}, $\Uj_1\du\subset U\du\subset B\du B$. By \cref{G_0^J}, we have $p^{-1}\in B\dr_wB$ for some $r_w\in W_J$. We get $\hjx x=\gj_1\du \cdot p^{-1}\in B\du B\cdot B\dr_w B\subset B\du\dr_w B$ by~\eqref{eq:Bruhat_length_add}. On the other hand, $\hjx\in U_-$ and $x\in B_-v''B$, so $\hjx x\in B_-v''B$. Therefore $\hjx xB\in\Rich_{v''}^{ur_w}$.

We now show $r_w\geq r$. By~\eqref{eq:lemma_01_>0}, $x\in \du\dr B_-B$, so by \cref{torus:closure}, we have $\rhour(t)\cdot xB\to \du\dr B$ as $t\to0$ in $G/B$. Since $\du\dr \in \du\Goj$, $\hjmap$ is regular at $\du\dr B$, and by \cref{lemma:hj_2=1_for_ut}, we have $\hjmp_{\du\dr}=1$. Thus $\hjmp_{\rhour(t) x} \rhour(t) xB\to \du\dr B$ as $t\to0$. By \cref{torus:kappa}, $\hjmp_{\rhour(t) x} \rhour(t) xB=\rhour(t)\cdot \hjx xB$, which belongs to $\Rich_{v''}^{ur_w}$ for all $t\in\Cast$. We see that the closure of $\Rich_{v''}^{ur_w}$ contains $\du\dr B$, and so $v''\leq ur\leq ur_w$ by~\eqref{eq:GB_closure}. Thus $r\leq r_w$ by \cref{ur_leq_ur'}.

Finally, we show $\hjx x B\in \GBtnn$.  First, clearly the map $\hjmap$ is defined over $\R$, so $\hjx x B\in \GBR$. Consider the subset $\Rtpv:=\bigsqcup_{w''\geq\ut} \Rtp_{v''}^{w''}\subset \GBtnn$. It contains $\Rtp_{v''}^{w_0}$ as an open dense subset, and therefore $\Rtpv$ is connected. We have already shown that for any $x'\in \Rtpv$, $\hjmp_{x'}x' B\in \RRich_{v''}^{\ut}$ (because we have $r_w\geq r=\woj$). Thus the image of the set $\Rtpv$ under the map $x'\mapsto \hjmp_{x'}x'$ must lie inside a single connected component of $\RRich_{v''}^{\ut}$. However, if $x'\in \Rtp_{v''}^\ut\subset \Rtpv$ then $\hjmp_{x'}=1$ by \cref{lemma:hj_2=1_for_ut}, so in this case $\hjmp_{x'}x'\in \Rtp_{v''}^{\ut}$. We conclude that the image of $\Rtpv$ is contained inside $\Rtp_{v''}^{\ut}\subset \GBtnn$. It follows by continuity that for arbitrary $v''\leq ur\leq w''$ and $x\in \Rtp_{v''}^{w''}$, we have $\hjx x B\in \GBtnn$.
\end{proof}

\noindent We will use the following consequence of \cref{lemma:hj_2_x_cell_r_w} in \cref{sec:Gr_total-positivity}.
\begin{corollary}\label{cor:hj_2_x_cell_r_w_proj}
  (Assume $\K=\C$.) In the notation of \cref{lemma:hj_2_x_cell_r_w}, we have $\hjx xP\in \PRtp_{\bar v''}^{u}$ for $\bar v'':=v''\triangleleft r_w^{-1}$.
\end{corollary}
\begin{proof}
\cref{lemma:hj_2_x_cell_r_w} says that $\hjx xB\in \Rtp_{v''}^{ur_w}$, so applying \cref{cor:MR_projection}, we find that $\pi_J(\hjx xB)=\hjx xP\in \PRtp_{\bar v''}^{u}$.
\end{proof}

\subsection{Proof via subtraction-free parametrizations}
In this section, we fix some set $\t$ of variables and assume $\K=\Fcalb$. Also fix $u\in W^J$ and recall that $\ut=u \woj\in W^J_\maxx$.

By \cref{dfn:gj_and_stuff}, the map $\hjmap$ is defined on $\du\Goj$. By \cref{lemma:subset_uP_sf}, we have $\Rsf_{v''}^{w''}\subset \du \Goj$ whenever $v''\leq ur\leq w''$ for some $r\in W_J$. In particular, $\hjmap$ is defined on $\Usf^-(w'')\subset \Rsf_\id^{w''}$ for all $w''\geq\ut$.
\begin{proposition}\label{sf_generic_h2j_h_is_sf} Let $q\in W$ be such that $\ell(\ut q)=\ell(\ut)+\ell(q)$. Then for $h\in \Usf^-(\ut q)$, we have $\hjh h\in \Usf^-(\ut)$.
\end{proposition}
\begin{proof}
Write $h\in \Usf^-(\ut q)=\Usf^-(\ut)\cdot \Usf^-(q)$. Using~\eqref{eq:collision1}, we find
\[h\in \dut\cdot \dut^{-1}\cdot \Usf^-(\ut)\cdot \Usf^-(q)\subset \dut \cdot \Bsf_- \cdot\Usf(\ut^{-1})\cdot \Usf^-(q).\]
By~\eqref{eq:Usf_commute}, $\Bsf_- \cdot \Usf(\ut^{-1})\cdot \Usf^-(q)=\Bsf_- \cdot  \Usf^-(q)\cdot\Usf(\ut^{-1})\subset \Bsf_-\cdot\Usf(\ut^{-1})$. Writing $\Bsf_-\subset U_-\cdot\Tsf$, we get
\[h\in \dut \cdot U_- \cdot\Tsf \cdot \Usf(\ut^{-1})=\Tsf \cdot\dut U_-\dut^{-1}\cdot \dut \cdot \Usf(\ut^{-1}).\]
Applying~\eqref{eq:collision1+}, we find
\[h\in \Tsf \cdot\dut U_-\dut^{-1}\cdot \Tsf \cdot(\Uo\cap \dut U_- \dut^{-1})\cdot \Usf^-(\ut)\subset \dut U_-\dut^{-1}\cdot\Tsf \cdot \Usf^-(\ut).\]
Let $g\in \dut U_-\dut^{-1}$ be such that $h\in g\cdot \Tsf \cdot \Usf^-(\ut)$. Recall from~\eqref{eq:U1_U2:Uj2} that $\Uj_2=\dut U_-\dut^{-1}\cap U_-$. By \cref{U(R)_generated}, there exists $h'\in \Uj_2$ such that $h'g\in \dut U_-\dut^{-1}\cap U$. Thus
\[h' h\in (\dut U_-\dut^{-1}\cap U)\cdot \Tsf\cdot \Usf^-(\ut)\subset U\cdot \Tsf\cdot \Usf^-(\ut).\]
But observe that both $h$ and $h'$ belong to $U_-$. Since the factorization of $h' h$ as an element of $U\cdot T\cdot U_-$ is unique by \cref{G_0_uniquely}, it follows that $h' h\in \Usf^-(\ut)$. By~\eqref{eq:U_BwB_over_K}, $\Usf^-(\ut)\subset B\dut B$. By \cref{lemma:hj_2=1_for_ut}, $\hjmp_{h'h}=1$, so $\hjmp_h=h'$, and thus $\hjmp_h h\in \Usf^-(\ut)$.
\end{proof}

\begin{corollary}
\label{sf_generic_Guv}For $q\in W$ such that $\ell(\ut q)=\ell(\ut)+\ell(q)$ and $v\leq \ut$, we have $\Rsf_\id^{\ut q}\subset  \Guvbig$.
\end{corollary}
\begin{proof}
As we have already mentioned, \cref{lemma:subset_uP_sf} shows that $\Rsf_\id^{\ut q}\subset \du \Goj$. Let $x\in\Rsf_\id^{\ut q}=\Usf^-(\ut q)\cdot \Bsf$, and let $b\in \Bsf$ and $h\in \Usf^-(\ut q)$ be such that $x=hb$.  By \cref{hjmap_properties}, we have $\hjmp_x=\hjmp_h$. By \cref{sf_generic_h2j_h_is_sf}, $\hjmp_hh\in \Usf^-(\ut)$, and therefore $\hjmp_x x\in \Usf^-(\ut) \cdot\Bsf=\Rsf_\id^{\ut}$. By~\eqref{eq:lemma_01_sfs}, we get $\hjmp_x x\in \dv B_-B$.
\end{proof}

\noindent \cref{sf_generic_Guv} shows that the map $\zetamap$ is defined on the whole $\Rsf_\id^{\ut q}$.

\def\zu{d}
\begin{lemma}\label{lemma:case_h_3=1}
Suppose that $u_0\in W^J$ and $v_0\leq \ut_0:=u_0\woj$. Let $h\in \Usf^-(\ut_0)$, and let $b_u,b_v\in U$ be such that $\dut_0^{-1} h\in B_-\cdot b_u$ and $\dv_0^{-1} h\in B_-\cdot b_v$. Then $[b_ub_v^{-1}]_J\in \Usf(r)$ for some $r\in W_J$.
\end{lemma}
\begin{proof}
First, recall from \cref{G_0_uniquely} and~\eqref{eq:collision2} that $b_u$ and $b_v$ are uniquely defined and satisfy $b_u\in \Usf(\ut_0^{-1})$, $b_v\in \Usf(v_0^{-1})$. Let $h=h_1h_2$ for $h_1\in \Usf^-(u_0)$ and $h_2\in \Usf^-(\woj)$. Our first goal is to show that $[b_u]_J\in U_J$ satisfies (and is uniquely defined by) $\dwoj^{-1} h_2\in B_-\cdot  [b_u]_J$. Letting $b_u'\in U_J$ be uniquely defined by $\dwoj^{-1}h_2\in B_-\cdot  b_u'$, we thus need to show that $[b_u]_J=b_u'$.

By~\eqref{eq:collision_length_add}, there exists $\zu\in \Usf(u_0^{-1})$ such that 
\[\dwoj^{-1} \du_0^{-1} h_1\in \Bsf_-\cdot \dwoj^{-1} \cdot \zu.\]
Since $\zu\in U$, we can use \cref{lemma:bracketsJ} to factorize it as $\zu=[\zu]_J[\zu]_+\pj$. Since $h_2\in U_J^-\subset L_J$, \cref{unip_radical} shows that there exists $\zu'\in \Uj$ such that $[\zu]_+\pj h_2=h_2 \zu'$. Since $[\zu]_J\in U_J$ by \cref{bracketsJ_homo}, \eqref{eq:dwoj_U_J} shows that $\dwoj^{-1}[\zu]_J\in U_-\dwoj^{-1}$. Combining the pieces together, we get
\[\dut_0^{-1}h=\dwoj^{-1}\du_0^{-1} h_1h_2\in \Bsf_- \cdot \dwoj^{-1} \cdot [\zu]_J[\zu]_+\pj\cdot h_2\subset B_- \cdot \dwoj^{-1} h_2 \zu'=B_-\cdot b_u'\zu'.\]
On the other hand, $\dut_0^{-1}h\in B_-\cdot b_u$, so $b_u=b_u'\zu'$, where $b_u'\in U_J$ and $\zu'\in \Uj$. It follows that $[b_u]_J=b_u'$, and thus we have shown that $\dwoj^{-1} h_2\in B_-\cdot  [b_u]_J$.

We now prove the result by induction on $\ell(u_0)$. When $\ell(u_0)=0$, we have $\ut_0=\woj$ and $v_0\in W_J$. Thus there exists $v_1\in W_J$ such that $\woj=v_0\cdot v_1$ with $\ell(\woj)=\ell(v_0)+\ell(v_1)$. We have $b_u,b_v\in U_J$, so $[b_ub_v^{-1}]_J=b_ub_v^{-1}$ by \cref{bracketsJ_homo}. By~\eqref{eq:collision_multi}, there exist $b_0\in\Usf(v_0^{-1})$ and $b_1\in\Usf(v_1^{-1})$ such that 
\[\dv_0^{-1} h\in \Bsf_- \cdot b_0,\quad \dwoj^{-1} h\in \Bsf_- \cdot b_1b_0.\]
In particular, we have $b_v=b_0$ and $b_u=b_1b_0$. Thus $[b_ub_v^{-1}]_J=b_1\in\Usf(v_1^{-1})$, and we are done with the base case.

Assume $\ell(u_0)>0$, and let $i\in I$ be such that $u_1:=s_iu_0<u_0$. By \cref{s_iW^J_still_W^J}, $u_1\in W^J$, so define $\ut_1:=u_1\woj\in W^J_\maxx$. Let $h\in \Usf^-(\ut_0)$ be factorized as $h=h_ih_1'h_2$ for $h_i=y_i(t)\in \Usf^-(s_i)$, $h_1'\in \Usf^-(u_1)$, and $h_2\in \Usf^-(\woj)$. 

Suppose that $s_iv_0>v_0$, in which case we have $v_0\leq \ut_1$. Let $h':=h_1'h_2$ and $b_{u}'\in U$ be defined by $\dut_1^{-1} h'\in B_-\cdot b_{u}'$. Since $s_iv_0>v_0$, we see that $\dv_0^{-1} h_i\in B_-\cdot \dv_0^{-1}$, so $\dv_0^{-1} h'\in B_-\cdot \dv_0^{-1} h= B_-\cdot  b_v$. By the induction hypothesis applied to $v_0\leq \ut_1$ and $h'\in \Usf^-(\ut_1)$, we have $[b_u' b_v^{-1}]_J\in \Usf(r)$ for some $r\in W_J$. On the other hand, we have shown above that $[b_u]_J$ satisfies $\dwoj^{-1} h_2\in B_-\cdot [b_u]_J$. But since $h'=h_1'h_2$ for $h_2\in \Usf^-(\woj)$, we get that $[b_u']_J$ satisfies $\dwoj^{-1} h_2\in B_-\cdot [b_u']_J$, and thus $[b_u]_J=[b_u']_J$. Therefore using \cref{bracketsJ_homo}, we get 
\[ [b_ub_v^{-1}]_J=[b_u]_J[b_v^{-1}]_J=[b_u']_J[b_v^{-1}]_J=[b_u'b_v^{-1}]_J\in\Usf(r),\]
finishing the induction step in the case $s_iv_0>v_0$.

Suppose now that $v_1:=s_iv_0<v_0$. Let $h=h_ih_1'h_2\in \Usf^-(\ut_0)$ be as above.  By~\eqref{eq:collision1}, $\ds_i^{-1} h_i\in \Bsf_-\cdot  \Usf(s_i)$, so  let $d_i\in \Usf(s_i)$ be such that $\ds_i^{-1} h_i\in \Bsf_-\cdot d_i$. By~\eqref{eq:Usf_commute}, $\Usf(s_i)\cdot \Usf^-(\ut_1)=\Usf^-(\ut_1)\cdot \Usf(s_i)$, so let $b_i\in \Usf(s_i)$ and $h'\in \Usf^-(\ut_1)$ be such that $d_i h_1'h_2=h'b_i$. We check using~\eqref{eq:collision_length_add} that 
\begin{equation}\label{eq:b_u_b_v_b_i}
\dut_0^{-1} h\in \Bsf_-\cdot \dut_1^{-1} h'\cdot b_i,\quad \dv_0^{-1}h\in \Bsf_-\cdot \dv_1^{-1}h'\cdot b_i.
\end{equation}
Let $b_u',b_v'\in U$ be defined by $\dut_1^{-1}h'\in B_-\cdot b_u'$ and $\dv_1^{-1}h'\in B_-\cdot b_v'$. Then by the induction hypothesis applied to $v_1\leq \ut_1$ and $h'\in\Usf^-(\ut_1)$, we find $[b_u'b_v'^{-1}]_J\in \Usf(r)$ for some $r\in W_J$. But it is clear from~\eqref{eq:b_u_b_v_b_i} that $b_u=b_u'b_i$ and $b_v=b_v'b_i$. Therefore $[b_ub_v^{-1}]_J\in \Usf(r)$.
\end{proof}

\begin{theorem}\label{thm:sf_generic_zetamap}
For all $v\leq \ut$, $w\in W^J$, $i\in I$, and $x\in \Rsf_\id^{w_0}$, we have
  \begin{equation}
\Deltapm_i(\zetamap(x) \dw^{-1})\in \Qsf.\label{eq:deltapm_zetamap_Qsf}
\end{equation}
\end{theorem}
\begin{proof}
Let $q\in W$ be such that $w_0=\ut q$, so $\ell(\ut q)=\ell(\ut)+\ell(q)$. Let $x\in \Rsf_\id^{w_0}=\Usf^-(w_0)\cdot\Bsf$ be written as $x=h \cdot  b$, where $h=h_1 h_2 h_3\in\Usf^-(w_0)$ for $h_1\in \Usf^-(u)$, $h_2\in \Usf^-(\woj)$, $h_3\in\Usf^-(q)$, and $b\in \Bsf$.  By~\eqref{eq:collision_multi}, there exist $b_1\in \Usf(u^{-1})$, $b_2\in\Usf(\woj)$, and $b_3\in\Usf(q^{-1})$ such that 
\begin{equation}\label{eq:b1b2b3}
\du^{-1}h\in \Bsf_-\cdot b_1,\quad \dut^{-1}h\in \Bsf_-\cdot b_2b_1,\quad \dw_0^{-1}h\in \Bsf_-\cdot b_3b_2b_1.
\end{equation}
Let $x':=h b_1^{-1}$.  We have $x'=x b^{-1}b_1^{-1}\in xB\subset xP$, and therefore $x'\in\Guvbig$ and $\zetamap(x')=\zetamap(x)$ by \cref{zeta_mod_P}. On the other hand, by~\eqref{eq:b1b2b3}, $x'\in \du\Bsf_-\subset \du P_-$, so \cref{zeta_inside_uP_-} implies $\zetamap(x')=x'\Ft(x')^{-1}$. 

\def\zu{d_0}
Let us now compute $\Ft(x')=[\dv^{-1} \hjmp_{x'}x']_J$. By \cref{hjmap_properties}, $\hjx=\hjmp_{x'}=\hjh$, and by \cref{sf_generic_h2j_h_is_sf}, $\hjh h\in \Usf^-(\ut)$. Thus by~\eqref{eq:collision2}, $\dv^{-1}\hjh h\in\Bsf_-\cdot \Usf(v^{-1})$, so let $\zu\in\Bsf_-$ and $b_0\in \Usf(v^{-1})$ be such that $\dv^{-1}\hjh h=\zu b_0$. By definition, $\hjh\in \Uj_2$, so by~\eqref{eq:U1_U2:Uj2}, $\dut^{-1}\hjh\dut\in U_-$, and therefore using~\eqref{eq:b1b2b3} we find
\[\dut^{-1} \hjh h=\dut^{-1}\hjh \dut\cdot \dut^{-1}h\in U_-\cdot \dut^{-1}h\subset B_-\cdot b_2b_1.\]

We can now apply \cref{lemma:case_h_3=1}: we have $v\leq \ut$, $\hjh h\in \Usf^-(\ut)$, $\dut^{-1} \hjh h\in B_- \cdot b_2b_1$, and $\dv^{-1}\hjh h\in B_- \cdot b_0$. Let $b_u:=b_2b_1\in U$ and $b_v:=b_0\in U$. By \cref{lemma:case_h_3=1}, $[b_ub_v^{-1}]_J=[b_2b_1b_0^{-1}]_J\in \Usf(r)$ for some $r\in W_J$.

Recall that $\dv^{-1} \hjh h=\zu b_0$ for $\zu\in \Bsf_-$ and $b_0\in \Usf(v^{-1})$. Thus
\[ \Ft(x')=[\dv^{-1}\hjmp_{x'}x']_J=[\dv^{-1}\hjh x']_J=[\dv^{-1}\hjh h b_1^{-1}]_J=[\zu b_0b_1^{-1}]_J.\]
By \cref{lemma:bracketsJ}, we get $[\zu b_0b_1^{-1}]_J=[\zu]_J[b_0b_1^{-1}]_J$. Thus
\[\zetamap(x)=\zetamap(x')=x'\Ft(x')^{-1}=x'[b_0b_1^{-1}]_J^{-1}[\zu]_J^{-1}.\] 
By~\eqref{eq:b1b2b3}, we have $\dw_0^{-1} x'\in \Bsf_-\cdot  b_3b_2$, so $x'\in \Bsf\dw_0 b_3b_2$. Using \cref{bracketsJ_homo}, we thus get
\[\zetamap(x)=x'[b_0b_1^{-1}]_J^{-1}[\zu]_J^{-1}\in \Bsf \cdot \dw_0 b_3[b_2b_1b_0^{-1}]_J[\zu]_J^{-1}.\]
We are interested in the element $\zetamap(x)\dw^{-1}$. We know that $\zu\in \Bsf_-$, so $[\zu]_J\in\Tsf U_J^-$, and by \cref{Inv_W^J}, $\dw [\zu]_J\dw^{-1}\in \Tsf \cdot U_-$. Hence
\[\zetamap(x)\dw^{-1}\in \Bsf \cdot \dw_0 b_3[b_2b_1b_0^{-1}]_J[\zu]_J^{-1} \dw^{-1}\subset \Bsf \cdot \dw_0 b_3[b_2b_1b_0^{-1}]_J\dw^{-1} \cdot \Tsf\cdot U_-.\]
In particular, $\Deltapm_i(\zetamap(x)\dw^{-1})\in\Qsf$ if and only if $\Deltapm_i(\dw_0 b_3[b_2b_1b_0^{-1}]_J\dw^{-1})\in\Qsf$. Recall that $b_3\in \Usf(q^{-1})$ and $[b_2b_1b_0^{-1}]_J\in \Usf(r)$ for some $r\in W_J$. Thus $b_3[b_2b_1b_0^{-1}]_J\in \Usf(q^{-1}r)$, so we are done by \cref{gen_minors_sf}.
\end{proof}

\begin{proof}[Proof of \cref{thm:zeta}.]
Our strategy will be very similar to the one we used in the proof of \cref{lemma:01}. 

Fix $(u,u)\leqJ(v,w)\leqJ(v',w')\in Q_J$. Let $\t=(\t_1,\t_2,\t_3)$ with $|\t_1|=\ell(v')$, $|\t_2|=\ell(w_0)-\ell(w')$, and $|\t_3|:=\ell(w')-\ell(v')$, and assume $\K=\Fcalb$. Choose reduced words $\bi$ for $v'^{-1}$ and $\bj$ for $w_0w'^{-1}$, and let $(\bv',\bw')\in \RedMR(v',w')$. Suppose that $x\in\gMR{\bv'}{\bw'}(\t_3)\cdot \Bsf$. Then 
\[g(\t_1,\t_2,\t_3):=\bx_\bi(\t_1)\cdot \by_\bj(\t_2)\cdot \gMR{\bv'}{\bw'}(\t_3)\in \Usf(v'^{-1})\cdot \Usf^-(w_0w'^{-1})\cdot \Rsf_{v'}^{w'}.\]
By \cref{lemma:param_sf}, we have $g(\t_1,\t_2,\t_3)\in \Rsf_\id^{w_0}$. Thus by \cref{thm:sf_generic_zetamap}, for all $i\in I$ we have  $\Deltapm_i(\zetamap(g(\t_1,\t_2,\t_3))\dw^{-1})\in\Qsf$. Denote by $f(\t_1,\t_2,\t_3):=\Deltapm_i(\zetamap(g(\t_1,\t_2,\t_3))\dw^{-1})$ the corresponding subtraction-free rational expression, which yields a continuous function $\R_{>0}^{|\t_1|}\times \R_{>0}^{|\t_2|}\times \R_{>0}^{|\t_3|}\to \R_{\geq0}$. We claim that $f$ extends to a continuous function $\R_{\geq0}^{|\t_1|}\times \R_{\geq0}^{|\t_2|}\times \R_{>0}^{|\t_3|}\to \R_{\geq0}$. Indeed, fix some $(\t_1',\t_2',\t_3')\in \R_{\geq0}^{|\t_1|}\times \R_{\geq0}^{|\t_2|}\times \R_{>0}^{|\t_3|}$ and let $\K=\C$. The element $x':=g(\t_1',\t_2',\t_3')$ (obtained by evaluating at $(\t_1',\t_2',\t_3')$; see \cref{eval}) belongs to $\Gtnn\cdot \Rtp_{v'}^{w'}$, and by \cref{lemma:Demazure} there exist $v'',w''\in W$ such that $v''\leq v'\leq w'\leq w''$ and $x'\in \Rtp_{v''}^{w''}$. Recall from \cref{Q_J_charact} that we have 
\[v''\leq v'\leq vr'\leq ur\leq wr'\leq w'\leq w''\]
for some $r',r\in W_J$ such that $\ell(vr')=\ell(v)+\ell(r')$. In particular, by \cref{lemma:hj_2_x_cell_r_w}, $x'\in \du\Goj$ and $\hjmp_{x'}x'\in \Rtp_{v''}^{ur_w}$ for some $r_w\in W_J$ such that $r_w\geq r$. By \cref{lemma:01}, $\hjmp_{x'}x'\in \dv\dr'B_-B\subset \dv\Goj$, which shows that $x'\in \Guvbig$. The map $\zetamap$ is therefore regular at $x'$ by \cref{zetamap_regular}. The map $\Deltapm_i$ is regular on $G$ by \cref{gen_minors_regular}, so in particular it is regular at $\zetamap(x')\dw^{-1}$. We have shown that the map $x''\mapsto \Deltapm_i(\zetamap(x'')\dw^{-1})$ is regular at $x'=g(\t'_1,\t'_2,\t'_3)$ for all $(\t_1',\t_2',\t_3')\in \R_{\geq0}^{|\t_1|}\times \R_{\geq0}^{|\t_2|}\times \R_{>0}^{|\t_3|}$. Thus the map $f(\t_1,\t_2,\t_3)$ extends to a continuous function $\R_{\geq0}^{|\t_1|}\times \R_{\geq0}^{|\t_2|}\times \R_{>0}^{|\t_3|}\to \R_{\geq0}$. By \cref{lemma:SF_limit}, we find that $f(0,0,\t_3):=\lim_{\t_1,\t_2\to 0} f(\t_1,\t_2,\t_3)$ belongs to $\Qsf$, i.e., it can be represented by a subtraction-free rational expression in the variables $\t_3$. On the other hand, it is clear that $f(0,0,\t_3)=\Deltapm_i(\zetamap(\gMR{\bv'}{\bw'}(\t_3))\dw^{-1})$.

\def\ff{\bar f}
Our next goal is to show that $f(0,0,\t_3)\in\Qsfs$. Indeed, suppose otherwise that $f(0,0,\t_3)=0$ (as an element of $\Fcal$). By \cref{zeta_mod_P}, $\zetamap$ descends to a regular map $\Guvbig/P\to G$ (still assuming $\K=\C$). Therefore the map $\ff:\Guvbig/P\to \C$ sending $x'P$ to $\Deltapm_i(\zetamap(x')\dw^{-1})$ is also regular. If $f(0,0,\t_3)=0$ then $\ff$ vanishes on $\pi_J(\Rtp_{v'}^{w'})=\PRtp_{v'}^{w'}$, and therefore it vanishes on its Zariski closure, which is $\PRcl_{v'}^{w'}$. We have $\pi_J(\Rtp_{v}^{w})=\PRtp_{v}^{w}\subset \PRcl_{v'}^{w'}$, so $\ff(x)=0$ for any $x\in \Guvbig$ such that $xB\in \Rtp_v^w$. Let us show that this leads to a contradiction. 

Let $x\in G$ be such that $xB\in \Rtp_v^w$. By~\eqref{eq:KLS_pi_J_length_add}, there exists $x'\in xP$ such that $x'B\in\Rtp_{vr'}^{wr'}$. By \cref{lemma:subset_uP}, we have $x'\in \du\Goj$, and thus $x\in\du\Goj$. Having $xB\in \Rtp_v^w$ implies $x\in B_-\dv B\cap B\dw B$.  Since $\hjx\in \Uj_2\subset U_-$, we have $\hjx x\in B_-\dv B$. By~\eqref{eq:BuB_uU_cap_Uu}, $B_-\dv B= (\dv U_-\cap U_-\dv)B\subset \dv B_- B$, so $\hjx x\in \dv B_-B$, and therefore $x\in \Guvbig$. Moreover, $\dv^{-1} \hjx x\in B_-B$, and thus $\Ft(x)=[\dv^{-1} \hjx x]_J\in U_J^- T U_J$. On the other hand, $\pidup(x)\in x\Uj\subset xB\subset B\dw B$; see \cref{lots_of_maps}. Thus
\[\zetamap(x)=\pidup(x)\Ft(x)^{-1}\in B\dw B\cdot U_J T U_J^- = B\dw B\cdot U_J^-.\]
Recall that because $w\in W^J$, we have $U_J^-\dw^{-1}\subset \dw^{-1} U_-$ by \cref{Inv_W^J}. Hence
\[\zetamap(x) \dw^{-1}\in B\dw B \cdot U_J^-\cdot \dw^{-1}\subset B\dw B\dw^{-1} B_-.\]
By~\eqref{eq:BuB_uU_cap_Uu} (after taking inverses of both sides), $B\dw B=B\cdot (U_-\dw\cap \dw U)$, so 
\[\zetamap(x)\dw^{-1}\in B\cdot (U_-\cap \dw U\dw^{-1})\cdot B_-\subset B\cdot B_-.\] 
In particular, $\Deltapm_i(\zetamap(x)\dw^{-1})\neq 0$ for all $i\in I$. This gives a contradiction, showing $f(0,0,\t_3)\in\Qsfs$. But then evaluating $f$ at any $\t_3'\in\R_{>0}^{\ell(w')-\ell(v')}$ yields a positive real number. We have shown that $\Deltapm_i(\zetamap(x)\dw^{-1})\neq 0$ for all $x\in G$ such that $xB\in \Rtp_{v'}^{w'}$. We are done by \cref{lemma:G_0}.
\end{proof}


\def\evm{\ev_\infty}
\def\evp{\ev_0}
\def\ev{\bar{\operatorname{ev}}}
\def\Am{\Acal_-}
\def\Ap{\Acal_+}

\def\hG{{\hat \G}}
\def\hB{{\hat \B}}
\section{Affine Bruhat atlas for the projected Richardson stratification}\label{sec:kac-moody-groups}
In this section, we embed the stratification~\eqref{eq:GP_closure} of $G/P$ inside the affine Richardson stratification of the affine flag variety. Throughout, we work over $\K=\C$.

\subsection{Loop groups and affine flag varieties} \label{sec:loop-groups-affine}
Recall that $G$ is a simple and simply connected algebraic group. 
Let $\Acal:= \C[z,z^{-1}]$ and $\Ap,\Am\subset \Acal$ denote the subrings given by $\Ap:=\C[z]$, $\Am:=\C[z^{-1}]$. Then we have ring homomorphisms $\evp:\Ap\to \C$ (respectively, $\evm:\Am\to \C$), sending a polynomial in $z$ (respectively, in $z^{-1}$) to its constant term.  Let $\G := G(\Acal)$ denote the polynomial loop group of $G$. 

\begin{remark}
The group $\G$ is closely related to the \emph{(minimal) affine Kac--Moody group} $\Gmin$ associated to $G$, introduced by Kac and Peterson~\cite{KacP,PKac}. Below we state many standard results about $\G$ without proof. We refer the reader unfamiliar with Kac--Moody groups to \cref{sec:KM}, where we give some background and explain how to derive these statements from Kumar's book~\cite{Kum}.
\end{remark}

We introduce opposite Iwahori subgroups
$$
\B := \{g(z) \in G(\Ap) \mid \evp(g) \in B\}, \qquad \B_- := \{g(z^{-1}) \in G(\Am) \mid \evm(g) \in B_-\}
$$
of $\G$, and denote by 
$$
\U := \{g(z) \in G(\Ap) \mid \evp(g) \in U\}, \qquad \U_- := \{g(z^{-1}) \in G(\Am) \mid \evm(g) \in U_-\}
$$
their unipotent radicals. There exists a tautological embedding $G\hookrightarrow \G$, and we treat $G$ as a subset of $\G$. 

We let $\T:= \Cast \times T\subset \Cast\ltimes G$ be the affine torus, where $\Cast$ acts on $\G$ via loop rotation; see~\cref{sec:KM_torus-action}. The \emph{affine root system} $\Phiaff$ of $\G$ is the subset of $X(\T):=\Hom(\T,\Cast)\cong X(T) \oplus \Z\delta$ given by
\[\Phiaff=\Phire\sqcup \Phiim,\quad\text{where}\quad \Phire:=\{\beta+j\delta\mid \beta\in \Phi,\ j\in\Z\},\quad \Phiim:=\{j\delta\mid j\in \Z\setminus\{0\}\}\]
are the real and imaginary roots, and the set of \emph{positive roots} $\Phiaff^+\subset \Phiaff$ has the form
\begin{equation}\label{eq:Phiaff_+}
\Phiaff^+=\{j\delta\mid j>0\}\sqcup \{\beta+j\delta\mid \beta\in \Phi,\ j>0\}\sqcup\{\beta\mid \beta\in\Phi^+\}.
\end{equation}
We let $\Phire^+:=\Phiaff^+\cap \Phire$ and $\Phire^-:=\Phiaff^-\cap \Phire$. For each $\alpha \in \Phire^+$ (respectively, $\alpha \in \Phire^-$), we have a one-parameter subgroup $\U_\alpha \subset \U$ (respectively, $\U_\alpha \subset \U_-$). The group $\U$ (respectively, $\U_-$) is generated by $\{\U_\alpha\}_{\alpha\in \Phire^+}$ (respectively, $\{\U_\alpha\}_{\alpha\in \Phire^-}$), and for each $\alpha\in\Phire$, we fix a group isomorphism $x_\alpha:\C\xrasim \U_\alpha$.

Let $\CRL:=\bigoplus_{i\in I}\Z\alphacheck_i$ denote the \emph{coroot lattice} of $\Phi$. 
The \emph{affine Weyl group} $\Waff=W\ltimes \CRL$ is a semidirect product of $W$ and $\CRL$, i.e., as a set we have $\Waff=W\times \CRL$, and the product rule is given by $(w_1,\lch_1)\cdot (w_2,\lch_2):=(w_1w_2,\lch_1+w_1\lch_2).$
For $\lch\in \CRL$, we denote the element $(\id, \lch)\in \Waff$ by  $\tl$. The group $\Waff$ is isomorphic to $N_{\Cast\ltimes \G}(\T)/\T$, and for $f \in \Waff$, we choose a representative $\df \in \G$  of $f$ in $N_{\Cast\ltimes \G}(\T)$, with the assumption that for $w \in W$, the representative $\dw \in G  \subset \G$ is given by \eqref{eq:x_i(t)}. Thus $\Waff$ is a Coxeter group with generators $s_0\sqcup \{s_i\}_{i\in I}$, length function $\ell:\Waff\to \Z_{\geq0}$, and \emph{affine Bruhat order} $\leq$.
The group $\Waff$ acts on $\Delta$, and for $\alpha\in \Phi$, $\beta \in \Phire$, $\lch\in\CRL$, and $w\in W$, we have 
\begin{equation}\label{eq:tau_on_Phi}
w\tl w^{-1}=\tau_{w\lch},\quad  \tl\alpha=\alpha+\<\lch,\alpha\>\delta, \quad \tl \delta=\delta, \quad \dtl\U_\beta \dtl^{-1}=\U_{\tl\beta}.
\end{equation}

Let $\G/\B$ denote the \emph{affine flag variety} of $G$.  This is an ind-variety that is isomorphic to the flag variety of the corresponding affine Kac--Moody group $\Gmin$; see \cref{sec:AKM}. For each $h,f\in \Waff$ we have Schubert cells $\Xaff^f:=\B \df\B/\B$ and opposite Schubert cells $\Xaff_h:=\B_-\dh\B/\B$.
If $h\not\leq f\in\Waff$ then $\Xaff_h\cap \Xaff^f=\emptyset$. For $h\leq f$, we denote $\Richaff_h^f:=\Xaff_h\cap \Xaff^f$. For all $g\in \Waff$, we have
\begin{align}
\label{eq:affine_Xaff_sqcup_Raff} 
  \Xaff^g=\bigsqcup_{h\leq g} \Richaff_h^g,\quad \Xaff_g=\bigsqcup_{g\leq f} \Richaff_g^f,\quad
\Xaffcl^g:=\bigsqcup_{h\leq g} \Xaff^h,\quad  \Xaffcl_g:=\bigsqcup_{g\leq f} \Xaff_f.
\end{align}
For $g\in \Waff$, let 
\begin{equation}\label{eq:Caff_U_1(g)_U_2(g)}
  \Caff_g:=\dg \B_-\B/\B,\quad \U_1(g):=\dg \U_- \dg^{-1}\cap \U,\quad\text{and}\quad \U_2(g):=\dg\U_-\dg^{-1}\cap \U_-.
\end{equation}
As we explain in \cref{sec:KMapp_Gauss}, the map $x\mapsto x\dg \B$ gives biregular isomorphisms
\begin{equation}\label{eq:KM_affine_charts}
  \dg\U_-\dg^{-1}\xrasim \Caff_g,\quad \U_1(g)\xrasim \Xaff^g,\quad \U_2(g)\xrasim \Xaff_g.
\end{equation}

\def\Uvp{\U^{(I)}}
\def\Uvn{\U^{(I)}_-}

Let $\Uvp\subset \U$ be the subgroup generated by $\{\U_\alpha\}_{\alpha\in\Phire^+\setminus \Phi^+}$. Similarly, let $\Uvn\subset \U_-$ be the subgroup generated by $\{\U_\alpha\}_{\alpha\in\Phire^-\setminus \Phi^-}$. For $x\in G\subset \G$, we have
\begin{equation}\label{eq:Uvp_Uvn_normal}
x\cdot \Uvp\cdot x^{-1}=\Uvp,\quad x\cdot \Uvn\cdot x^{-1}=\Uvn.
\end{equation}

\subsection{Combinatorial Bruhat atlas for \texorpdfstring{$G/P$}{G/P}} \label{sec:BA_combinatorial-map}
We fix an element $\lch\in \CRL$ such that $\<\lch,\alpha_i\>=0$ for $i\in J$ and $\<\lch,\alpha_i\>\in\Z_{<0}$ for $i\in I\setminus J$.  Thus $\lch$ is anti-dominant and the stabilizer of $\lch$ in $W$ is equal to $W_J$.  Following~\cite{HL}, define a map 
\begin{equation}\label{eq:BA_comb}
\BAcomb:Q_J\to \Waff,\quad (v,w)\mapsto \vtw.
\end{equation}
By~\cite[Theorem~2.2]{HL}, the map $\BAcomb$ gives an order-reversing bijection between $Q_J$ and a subposet of $\Waff$. More precisely, let $\tmin:=\tl \wji$, and recall from~\eqref{eq:tau_on_Phi} that $u\tl u^{-1}=\tul$. By~\cite[\S2.3]{HL}, for all $(v,w)\in Q_J$ we have
\begin{equation}\label{eq:tau_length_add}
\vtw=v\cdot \tmin\cdot \wj w^{-1},\quad  \ell(\vtw)=\ell(v)+\ell(\tmin)+\ell(\wj w^{-1});
\end{equation}
see Figure~\ref{fig:Le} for an example. By~\cite[Theorem~2.2]{HL}, for all $u\in W^J$ we have
\begin{align}\label{eq:BA_comb_image}
\BAcomb(\QJfilter uu)&=\{g\in \Waff\mid \tmin \leq g\leq \tul\},\\   
\BAcomb(Q_J)&=\{g\in \Waff\mid \tmin \leq g\leq \twl\text{ for some $w\in W^J$}\}.
\end{align}

\begin{remark}\label{rmk:minuscule}
The construction of \cite{HL} can be applied in the more general setting where $\lambda$ is an anti-dominant coweight, and thus $\BAcomb$ sends $Q_J$ to the extended affine Weyl group.  This is especially natural when $\la$ is a minuscule coweight, and thus $G/P$ is a {\it cominuscule Grassmannian}.  In this case, the image of $\BAcomb$ is a lower order ideal in affine Bruhat order.  The map $\BAgeom$ below then sends $\Cuj$ to the Schubert cell $\Xaff^{\tul}$ as opposed to the more complicated intersection $\Xaffcl_{\tmin}\cap \Xaff^{\tul}$.
\end{remark}

\subsection{Bruhat atlas for the projected Richardson stratification of \texorpdfstring{$G/P$}{G/P}}\label{sec:bruh-atlas-proj}
Let $u\in W^J$.   Recall that $\lch \in \CRL$ has been fixed.  We further assume that the representatives $\dtl$ and $\dtul$ satisfy the identity $\du \dtl \du^{-1} = \dtul$.

Our goal is to construct a geometric lifting of the map $\BAcomb$. Recall the maps $x\mapsto \gj_1$ and $x\mapsto \gj_2$ from \cref{dfn:gj_and_stuff}. We define maps
\begin{align}\label{eq:BA_geo}
\BAgeo:\Cuj&\to \G, &&xP\mapsto \gj_1\du\cdot \dtl\cdot  (\gj_2\du)^{-1}=\gj_1\cdot \dtul\cdot (\gj_2)^{-1},\quad\text{and}\\
\label{eq:BA_geom}
\BAgeom:\Cuj&\to \G/\B, &&xP\mapsto \BAgeo(xP)\cdot \B.
\end{align}

The main result of this section is the following theorem.
\def\myarabic#1{\textup{(\arabic{#1})}}
\begin{theorem}\label{thm:BA}\leavevmode
\begin{theoremlist}
\item
\label{BAgeom1}
The map $\BAgeom$ is a biregular isomorphism
\[\BAgeom: \Cuj\xrasim \Xaffcl_{\tmin}\cap \Xaff^{\tul}=\bigsqcup_{(v,w)\in \QJfilter uu} \Richaff_{\vtw}^{\tul},\]
and for all $(v,w)\geqJ (u,u)\in Q_J$, $\BAgeom$ restricts to a biregular isomorphism 
$$\BAgeom: \Cuj\cap \PR_v^w\xrasim \Richaff_{\vtw}^{\tul}.$$
\item 
\label{BAgeom3}
Suppose that $(u,u)\leqJ (v,w)\leqJ(v',w')\in Q_J$. Then 
\begin{equation*}
\BAgeom\left(\PRtp_{v'}^{w'}\right)\subset \Caff_{\vtw}.
\end{equation*}
\end{theoremlist}
\end{theorem}
\noindent The remainder of this section will be devoted to the proof of \cref{thm:BA}.  

\def\myarabic#1{\normalfont(\roman{#1})}

\subsection{An alternative definition of \texorpdfstring{$\BAgeom$}{phi}}
Recall the notation from \cref{dfn:gj_and_stuff}, and that we have fixed $u \in W^J$ and $\lch\in \CRL$ satisfying $\<\lch,\alpha_i\>=0$ for $i\in J$ and $\<\lch,\alpha_i\>\in\Z_{<0}$ for $i\in I\setminus J$.  We list the rules for conjugating elements of $G\subset \G$ by $\dtl$.
\begin{lemma}
We have
\begin{align}
\label{eq:tau_conj_L_J}
&\dtl \cdot p=p\cdot \dtl\quad \text{for all $p\in L_J$,}\\
\label{eq:tau_conj_Uj}
  & \dtl \cdot \Uj\cdot \dtl^{-1}\subset \Uvn,\quad \dtl \cdot \Uj_-\cdot \dtl^{-1}\subset \Uvp,\\
\label{eq:tau_inv_conj_Uj}
 &\dtl^{-1} \cdot \Uj\cdot \dtl\subset \Uvp,\quad \dtl^{-1} \cdot \Uj_-\cdot \dtl\subset \Uvn,\\
\label{eq:tul_conj_Uj}
 &\dtul \cdot \Uj_2\cdot \dtul^{-1}\subset \Uvp,\quad \dtul^{-1} \cdot \Uj_1\cdot \dtul\subset \Uvn.
\end{align}
\end{lemma}
\begin{proof}
Recall that $L_J$ is generated by $T$, $U_J$, and $U_J^-$, and since $\tl\alpha=\alpha$ for all $\alpha\in\Phi_J$, we see that~\eqref{eq:tau_conj_L_J} follows from~\eqref{eq:tau_on_Phi}. By~\eqref{eq:tau_on_Phi}, we find $\tl \alpha\in \Phire^+\setminus \Phi^+$ for $\alpha\in \Phij_-$ and $\tl \alpha\in \Phire^-\setminus \Phi^-$ for $\alpha\in \Phij_+$, which shows~\eqref{eq:tau_conj_Uj}. Similarly, $\tl^{-1} \alpha\in \Phire^+\setminus \Phi^+$ for $\alpha\in \Phij_+$ and $\tl^{-1} \alpha\in \Phire^-\setminus \Phi^-$ for $\alpha\in \Phij_-$, which shows~\eqref{eq:tau_inv_conj_Uj}.

To show~\eqref{eq:tul_conj_Uj}, we use~\eqref{eq:Uvp_Uvn_normal}, \eqref{eq:tau_conj_Uj},~\eqref{eq:tau_inv_conj_Uj}, and $\Uj_1,\Uj_2\subset \du\Uj_-\du^{-1}$ to get
\begin{align*}
&\dtul \cdot \Uj_2\cdot \dtul^{-1}= \du \dtl \du^{-1}\cdot \Uj_2\cdot \du\dtl^{-1}\du^{-1}\subset \du \dtl\cdot \Uj_-\cdot \dtl^{-1}\du^{-1}\subset \du \Uvp \du^{-1}= \Uvp,\\
&\dtul^{-1} \cdot \Uj_1\cdot \dtul=\du \dtl^{-1} \du^{-1}\cdot \Uj_1\cdot \du\dtl\du^{-1}\subset \du \dtl^{-1}\cdot \Uj_-\cdot \dtl\du^{-1}\subset \du \Uvn \du^{-1}= \Uvn.\qedhere
\end{align*}
\end{proof}

The map $\BAgeom$ can alternatively be characterized as follows. Recall from \cref{dfn:gj_and_stuff} that we have a regular map $\hjmap:\du\Goj\to \Uj_2$ that descends to a regular map $\hjmap:\Cuj\to \Uj_2$ by \cref{hjmap_properties}. Recall also from \cref{G_0^J} that $\du\Goj=\du P_-\cdot B$.
\begin{lemma}
Let $x\in \du P_-$. Then
\begin{equation}\label{eq:BA_geom2}
\BAgeom(xP)=\hjx x\cdot  \dtl\cdot x^{-1}\cdot \B.
\end{equation}
\end{lemma}
\begin{proof}
We continue using the notation of \cref{dfn:gj_and_stuff}. Let $p\in L_J$ and $\gj\in \du\Uj_-\du^{-1}$ be such that $xp=\gj\du$. Note that $\gj_2\du=\hj_1\gj\du=\hj_1xp$, and since $\hj_1\in \Uj_1\subset U\subset \B$, we see that $(\gj_2\du)^{-1}\cdot \B=(xp)^{-1}\cdot \B$. On the other hand, $\hjx xp=\hj_2\gj\du=\gj_1\du$. Since $p$ commutes with $\dtl$ by~\eqref{eq:tau_conj_L_J}, we find
\[\BAgeom(xP)=\gj_1\du\cdot \dtl\cdot  (\gj_2\du)^{-1}\cdot \B=\hjx xp\cdot  \dtl\cdot(xp)^{-1}\cdot \B= \hjx x\cdot  \dtl\cdot x^{-1}\cdot \B.\qedhere\]
\end{proof}

\subsection{The affine Richardson cell of \texorpdfstring{$\BAgeom$}{phi}} 
\begin{lemma}
We have
\begin{equation}\label{eq:Cuj_Rich_disjoint}
  \Cuj=\bigsqcup_{(v,w)\in \QJfilter uu} (\Cuj\cap \PR_v^w).
\end{equation}
\end{lemma}
\begin{proof}
The torus $T$ acts on $G/P$ by left multiplication and preserves the sets $\Cuj$ and $\PR_v^w$ for all $(v,w)\in Q_J$. By~\eqref{eq:GP_closure}, $\PRcl_v^w$ contains $\du P$ if and only if $(u,u)\leqJ (v,w)$. Suppose that $xP\in \Cuj\cap \PR_v^w$ for some $(v,w)\in Q_J$. Then $TxP/P\subset  \Cuj$, and by \cref{torus:closure}, the closure of this set contains $\du P$. On the other hand, the closure of this set is contained inside $\PRcl_v^w$, and thus  $(u,u)\leqJ (v,w)$.
\end{proof}

\begin{lemma}\label{BAgeom_BAcomb_Richaff}
Let $(v,w)\in \QJfilter uu$. Then 
\begin{equation}\label{eq:BAgeom_BAcomb_Richaff}
\BAgeom(\Cuj\cap \PR_v^w)\subset \Richaff_{\vtw}^{\tul}.
\end{equation}
\end{lemma}
\begin{proof}
Let $x\in \du\Goj$ be such that $xP\in \PR_v^w$. Let us first show that $\BAgeom(xP)\in \Xaff^{\tul}$. By~\eqref{eq:BA_geom}, we have
\begin{equation}\label{eq:BA_geom3}
\BAgeom(xP)=\gj_1\cdot \dtul\cdot (\gj_2)^{-1}\cdot\dtul^{-1}\cdot \dtul\cdot \B.
\end{equation}
Observe that $\gj_1\in \Uj_1\subset U$, and by~\eqref{eq:tul_conj_Uj}, $\dtul\cdot (\gj_2)^{-1}\cdot\dtul^{-1}\in \Uvp$. We get
\begin{equation}\label{eq:BA_geom_BtulB}
  \BAgeo(xP)\cdot\dtul^{-1}\in \U,\quad\text{so}\quad   \BAgeo(xP)\in \B \cdot \dtul\cdot \B.
\end{equation}
This proves that $\BAgeom(xP)\in \Xaff^{\tul}$.

We now show $\BAgeom(xP)\in \Xaff_{\vtw}$. Recall that $\PR_v^w=\pi_J(\Rich_v^w)$, so assume that $x\in B_-\dv B\cap B\dw B$.  Since $\du\Goj=\du P_-B$ by \cref{G_0^J}, we may assume that $x\in \du P_-$, in which case $\BAgeom(xP)$ is given by~\eqref{eq:BA_geom2}.  We have $\hjx x\in B_-\dv B$ and $x^{-1}\in B\dw^{-1}B$, so it suffices to show
\begin{equation}
B_-\dv B\cdot \dtl \cdot B\dw^{-1} B\subset \B_- \cdot \dv\dtl \dw^{-1}\cdot  \B.\label{eq:B_-vtwB}
\end{equation}

Clearly we have 
\[B_-\dv B\cdot \dtl \cdot B\dw^{-1} B\subset \B_- \cdot \dv\cdot \Uj \cdot U_J\cdot \dtl \cdot \Uj\cdot U_J\cdot   \dw^{-1}\cdot  \B.\] 
By~\eqref{eq:tau_conj_L_J} and \cref{unip_radical}, $U_J$ can be moved to the right past $\dtl$ and $\Uj$. We can then move $\Uj$ to the left past $\dtl$ using~\eqref{eq:tau_conj_Uj}, which gives
\[B_-\dv B\cdot \dtl \cdot B\dw^{-1} B\subset \B_- \cdot \dv\cdot \Uj\cdot \Uvn\cdot  \dtl \cdot U_J\cdot   \dw^{-1}\cdot  \B.\]
By~\eqref{eq:Uvp_Uvn_normal}, $\Uvn$ can be moved to the left past $\dv\cdot \Uj$, and then $\Uj$ can be moved to the right past $\dtl$ using~\eqref{eq:tau_inv_conj_Uj}, yielding
\[B_-\dv B\cdot \dtl \cdot B\dw^{-1} B\subset \B_- \cdot \dv\cdot  \dtl \cdot \Uvp \cdot U_J\cdot   \dw^{-1}\cdot  \B.\]
By~\eqref{eq:Uvp_Uvn_normal}, $\Uvp$ can be moved to the right past $U_J\cdot \dw^{-1}$. Since $w\in W^J$, \cref{Inv_W^J} implies that $U_J\cdot \dw^{-1}\subset \dw^{-1} U$, so~\eqref{eq:B_-vtwB} follows.
\end{proof}

\def\BAgeomx{\BAgeom^\dagger}
\def\BAgeomxx{\BAgeom^\ddagger}

\subsection{Proof of \texorpdfstring{\cref{BAgeom1}}{Theorem \ref{BAgeom1}}}
Observe that $\Xaffcl_{\tmin}\cap \Xaff^{\tul}=\bigsqcup_{(v,w)\in \QJfilter uu} \Richaff_{\vtw}^{\tul}$ by~\eqref{eq:affine_Xaff_sqcup_Raff} and~\eqref{eq:BA_comb_image}. By~\eqref{eq:BAgeom_BAcomb_Richaff},  $\BAgeom(\Cuj)\subset \Xaffcl_{\tmin}\cap \Xaff^{\tul}$. Let us identify $\Xaff^{\tul}$ with the affine variety $\U_1(\tul)$ via~\eqref{eq:KM_affine_charts}, and denote by $\BAgeomx:\Cuj\to \U_1(\tul)$ the composition of~\eqref{eq:KM_affine_charts} and $\BAgeom$.

 We claim that $\BAgeomx$ gives a biregular isomorphism between $\Cuj$ and a closed subvariety of $\U_1(\tul)$. Let $x\in \du\Goj$ and let $\gj$, $\gj_1$, $\gj_2$ be as in \cref{dfn:gj_and_stuff}. Let $y:=\BAgeo(xP)\cdot \dtul^{-1}$, so $\BAgeom(xP)=y\cdot \dtul\cdot \B$. Thus $\BAgeomx(xP)=y$ if and only if $y\in \U_1(\tul)$. By~\eqref{eq:BA_geom_BtulB}, we have $y\in \U$. Hence in order to prove  $y\in \U_1(\tul)$, we need to show $y\in \dtul \U_-\dtul^{-1}$. Conjugating both sides by $\dtul$, we get
\[\dtul^{-1}\cdot y\cdot \dtul= \dtul^{-1}\gj_1\dtul \cdot (\gj_2)^{-1},\]
which belongs to $\U_-$ since $(\gj_2)^{-1}\in \Uj_2\subset U_-$ by definition and $\dtul^{-1}\gj_1\dtul\in \Uvn$ by~\eqref{eq:tul_conj_Uj}. Thus  $y\in \U_1(\tul)$ and $\BAgeomx(xP)=y$. By \cref{lemma:KWY}, we may identify $\Cuj$ with $\Uj_1\times \Uj_2$, so let $\BAgeomxx:\Uj_1\times \Uj_2\to \U_1(\tul)$ be the map sending $(\gj_1,\gj_2)$ to $y:=\gj_1\cdot \dtul (\gj_2)^{-1}\dtul^{-1}$. 

Let $\Theta_1:=u\Phij_- \cap\Phi^+$ and $\Theta_2:=u\Phij_-\cap \Phi^-$, so $\Uj_1=U(\Theta_1)$, $\Uj_2=U_-(\Theta_2)$, and $\Theta_1\sqcup \Theta_2=u\Phij_-$. By the proof of~\eqref{eq:tul_conj_Uj}, $\tul \Theta_2\subset \Phire^+\setminus \Phi^+$ and $\tul^{-1}\Theta_1\subset \Phire^-$, and thus $\Theta_1\sqcup \tul\Theta_2\subset \Inv(\tul^{-1})$. Let $\Theta_3\subset \Phire^+$ be defined by $\Theta_3:=\Inv(\tul^{-1})\setminus(\Theta_1\sqcup \tul\Theta_2)$. By \cref{KM_U(R)_generated}, the multiplication map gives a biregular isomorphism 
\begin{equation}\label{eq:KM_T1_T2_T3}
\U(\Theta_1)\times \U(\tul\Theta_2)\times \prod_{\alpha \in \Theta_3} \U_\alpha\xrasim \U(\Inv(\tul^{-1}))=\U_1(\tul),
\end{equation}
where $\U(\Theta)$ denotes the subgroup generated by $\{\U_\alpha\}_{\alpha\in\Theta}$. In particular, $\U(\Theta_1)\cdot \U(\tul\Theta_2)$ is a closed subvariety of $\U_1(\tul)$ isomorphic to $\C^{|\Theta_1|+|\Theta_2|}=\C^{\ell(\wj)}$. Observe that $\U(\tul\Theta_2)=\dtul\Uj_2\dtul^{-1}$, and hence $\BAgeomxx$  essentially coincides with the restriction of the map~\eqref{eq:KM_T1_T2_T3} to $\U(\Theta_1)\times \U(\tul\Theta_2)\times \{1\}$. We have thus shown that $\BAgeomxx$ gives a biregular isomorphism between $\Uj_1\times \Uj_2$ and a closed $\ell(\wj)$-dimensional subvariety of $\U_1(\tul)$. Therefore $\BAgeom$ gives a biregular isomorphism between $\Cuj$ and a closed $\ell(\wj)$-dimensional subvariety $\BAgeom(\Cuj)$ of $\Xaff^{\tul}$. By \cref{Xaff_cl_irreducible}, $\Xaffcl_{\tmin}\cap \Xaff^{\tul}$ is a closed irreducible subvariety of $\Xaff^{\tul}$, and by~\eqref{eq:tau_length_add} and \cref{Xaff_cl_irreducible}, it has dimension $\ell(\wj)$. Since $\BAgeom(\Cuj)\subset \Xaffcl_{\tmin}\cap \Xaff^{\tul}$, it follows that $\BAgeom(\Cuj)= \Xaffcl_{\tmin}\cap \Xaff^{\tul}$. We are done with the proof of \cref{BAgeom1}.

\begin{remark}
Alternatively, the proof of \cref{BAgeom1} could be deduced from \emph{Deodhar-type parametrizations}~\cite{Had2,Had1,Billig} of $\Richaff_{\vtw}^{\tul}$, by observing that any reduced word for $\tul$ that is compatible with the length-additive factorization $\tul=u\cdot \tmin\cdot \wj u^{-1}$ in~\eqref{eq:tau_length_add} contains a unique reduced subword for $\tmin$.
\end{remark}

\subsection{Proof of \texorpdfstring{\cref{BAgeom3}}{Theorem \ref{BAgeom3}}}\label{sec:BAgeom3_proof}
We use the notation and results from \cref{sec:KWY_TP}.  Let $x\in G$ be such that $xP\in \PRtp_{v'}^{w'}$. Since $\PRtp_{v'}^{w'}=\pi_J(\Rtp_{v'}^{w'})$, we may assume that $xB\in \Rtp_{v'}^{w'}$. Then $x\in\du\Goj$ by \cref{lemma:subset_uP}, so $\BAgeom(xP)$ is defined. In addition, by \cref{G_0^J} we may assume that $x\in \du P_-$. By definition, $\BAgeom(xP)\in \Caff_{\vtw}$ if and only if $\dw\dtl^{-1}\dv^{-1} \BAgeom(xP)\in \B_-\B/\B$. By~\eqref{eq:BA_geom2}, this is equivalent to 
\begin{equation}\label{eq:KM_vtw_B_-_Bmin}
\dw\dtl^{-1}\dv^{-1}\cdot \hjx x\cdot  \dtl\cdot x^{-1} \in \B_-\B.
\end{equation}
By \cref{thm:zeta}, $x\in \Guvbig$, so $\dv^{-1}\hjx x\in \Goj$. Let us factorize $y:=\dv^{-1}\hjx x$ as $y=[y]_-\pj [y]_J [y]_+\pj$ using \cref{lemma:bracketsJ}. By~\eqref{eq:tau_conj_L_J} and~\eqref{eq:tau_inv_conj_Uj}, we get
\[\dw\dtl^{-1}\dv^{-1}\cdot \hjx x\cdot  \dtl\cdot x^{-1}= \dw\cdot \dtl^{-1}[y]_-\pj \dtl \cdot \dtl^{-1}[y]_J \dtl\cdot \dtl^{-1}[y]_+\pj \dtl\cdot x^{-1}\in  \dw\cdot \Uvn \cdot [y]_J \cdot \Uvp\cdot x^{-1}.\]
Using~\eqref{eq:Uvp_Uvn_normal}, we can move $\Uvn$ to the left and $\Uvp$ to the right, so we see that~\eqref{eq:KM_vtw_B_-_Bmin} is equivalent to $\dw[y]_Jx^{-1}\in \B_-\B$. By \cref{lots_of_maps}, we have $[y]_J=\Ft(x)$, and by \cref{zeta_inside_uP_-}, we have $\zetamap(x)=x \Ft(x)^{-1}=x[y]_J^{-1}$. By \cref{thm:zeta}, $\zetamap(x)\in BB_-\dw$, and after taking inverses, we obtain $\dw[y]_Jx^{-1}\in B_-B\subset \B_-\B$, finishing the proof. \qed

\tikzset{
    labl/.style={anchor=south, rotate=90, inner sep=.5mm}
}

\def\isogaff{\tilde{\nu}_g}
\def\isogaffa{\tilde{\nu}_{g,1}}
\def\isogaffb{\tilde{\nu}_{g,2}}
\section{From Bruhat atlas to Fomin--Shapiro atlas}\label{sec:BA_implies_FS}
We use \cref{thm:BA} to prove \cref{thm:main_intro2}.
 
\subsection{Affine Bruhat projections}
We first define the affine flag variety version of the map $\isog$ from~\eqref{eq:FSiso_intro}. We will need some results on the Gaussian decomposition inside $\G$; see \cref{sec:KMapp_Gauss} for a proof.

\begin{lemma}\leavevmode\label{lemma:KM_gaussian} Let $\G_0:=\B_-\cdot \B$.
\begin{theoremlist}
\item \label{KM_gaussian} The multiplication map gives a biregular isomorphism of ind-varieties
\begin{equation}\label{eq:KM_gaussian}
  \U_-\times \T\times \U \xrasim \G_0.
\end{equation}
For $x\in \G_0$, we denote by $[x]_-\in\U_-$, $[x]_0\in\T$, and $[x]_+\in \U$ the unique elements such that $x=[x]_-[x]_0[x]_+$.
\item \label{KM_U1_U2} For $g\in \Waff$, the multiplication map gives biregular isomorphisms of ind-varieties
\begin{equation}\label{eq:KM_U1_U2}
  \mu_{12}:\U_1(g)\times \U_2(g) \xrasim \dg \U_-\dg^{-1},\quad   \mu_{21}:\U_2(g)\times \U_1(g) \xrasim \dg \U_-\dg^{-1}.
\end{equation}
\end{theoremlist}
\end{lemma}
\noindent The group $\dg \U_-\dg^{-1}$, as well as its subgroups $\U_1(g)$ and $\U_2(g)$, act on $\Caff_g$.  The following result, which we state for the polynomial loop group $\G$, holds in Kac--Moody generality.

\begin{proposition}\label{prop:FS_iso}
Let $g\in \Waff$.
\begin{theoremlist}
\item\label{FS_iso_g1_g2} For $x\in \G$ such that $x\B\in\Caff_g$, there exist unique elements $y_1\in \U_1(g)$ and $y_2\in \U_2(g)$ such that $y_1x\B\in\Xaff_g$ and $y_2x\B\in\Xaff^g$.
\item\label{FS_iso_biregular} The map $\isogaff:\Caff_g\xrasim  \Xaff_g\times \Xaff^g$ sending $x\B$ to $(y_1x\B,y_2x\B)$ is a  biregular isomorphism of ind-varieties.
\item\label{FS_iso_Rich} For all $h,f\in \Waff$ satisfying $h\leq g\leq f$, the map $\isogaff$ restricts to a biregular isomorphism $\Caff_g\cap \Richaff_h^f\xrasim  \Richaff_g^f\times \Richaff_h^g$ of finite-dimensional varieties.
\end{theoremlist}
\end{proposition}
\begin{proof}
Let us first prove an affine analog of \cref{lemma:KWY}. Let $\nu_1:\dg\U_-\dg^{-1}\to \U_2(g)$, $\nu_2:\dg\U_-\dg^{-1}\to \U_1(g)$ denote the second component of $\mu_{12}^{-1}$ and $\mu_{21}^{-1}$ (cf.~\eqref{eq:KM_U1_U2}), respectively, and let $\nu:=(\nu_1,\nu_2):\dg\U_-\dg^{-1}\to \U_2(g)\times \U_1(g)$. We claim that $\nu$ is a biregular isomorphism. By \cref{KM_U1_U2}, $\nu$ is a regular morphism. Let us now compute the inverse of $\nu$. Given $x_1\in \U_1(g)$ and $x_2\in\U_2(g)$, we claim that there exist unique $y_1\in \U_1(g)$ and $y_2\in \U_2(g)$ such that $y_1x_2=y_2x_1$. Indeed, this equation is equivalent to $y_2^{-1}y_1=x_1x_2^{-1}$, so we must have $y_2=[x_1x_2^{-1}]_-^{-1}$ and $y_1=[x_1x_2^{-1}]_+$. Clearly, $\nu^{-1}(x_2,x_1)=y_1x_2=y_2x_1$, and by \cref{KM_gaussian}, the map $\nu^{-1}$ is regular. Applying~\eqref{eq:KM_affine_charts} finishes the proof of~\itemref{FS_iso_g1_g2} and~\itemref{FS_iso_biregular}.

We now prove~\itemref{FS_iso_Rich}. Observe that if $x\B\in \Caff_g\cap \Richaff_h^f$ for some $h\leq f\in\Waff$ then $x\in \B_- \dh \B\cap \B \df\B$. Let $y_1,y_2$ be as in~\itemref{FS_iso_biregular}. Then $y_1\in \U_1(g)\subset \U$, so $y_1x\in \B\df\B$. Similarly, $y_2\in \U_2(g)\subset \U_-$, so $y_2x\in \B_-\dh \B$. It follows that if $x\B\in \Caff_g\cap \Richaff_h^f$ then $\isogaff(x\B)\in \Richaff_h^g\times \Richaff_g^f$. In particular, we must have $h\leq g\leq f$, and we are done by ~\eqref{eq:affine_Xaff_sqcup_Raff}.
\end{proof}

\subsection{Torus action}\label{sec:KM_torus-action}
Recall that $\T= \Cast \times T$ is the affine torus. The group $\Cast$ acts on $\G$ via loop rotation as follows. For $\tht \in \Cast$, we have $\tht \cdot g(z) = g(tz)$.  We form the semidirect product $\Cast \ltimes \G$ with multiplication given by $(\tht_1,x_1(z))\cdot (\tht_2,x_2(z)):=(\tht_1\tht_2,x_1(z)x_2(\tht_1 z))$ for $(\tht_1,x_1(z)),(\tht_2,x_2(z))\in\Cast\times \G$. Let $Y(\T):=\Hom(\Cast,\T)\cong \Z d \oplus Y(T)$. For $\lch\in Y(\T)$, $t\in \Cast$, $t'\in \C$, and $\alpha\in \Phire$, we have 
\begin{equation}\label{eq:KM_torus_rules}
\lch(t) x_\alpha(t') \lch(t)^{-1}=x_\alpha(t^{\<\lch,\alpha\>}t'),
\end{equation}
where $x_\alpha:\C\xrasim \U_\alpha$ is as in \cref{sec:loop-groups-affine}, and $\<\cdot,\cdot\>:Y(\T)\times X(\T)\to \Z$ extends the pairing from \cref{sec:pinnings} in such a way that $\<d,\delta\> = 1$ and $\<d,\alpha_i\>=\<\alphacheck_i,\delta\>= 0$ for $i \in I$.

\def\NNg{N}
Let $g\in\Waff$ and define $\NNg:=\ell(g)$. If $\Inv(g)=\{\alpha^\parr1,\dots,\alpha^\parr{\NNg}\}$, then by \cref{KM_U(R)_generated}, the map $\bx_g:\C^{\NNg}\to \U_1(g)$ given by
\begin{equation}\label{eq:KM_bx_g}
\bx_g(t_1,\dots,t_{\NNg}):=x_{\alpha^\parr1}(t_1)\cdots x_{\alpha^\parr{\NNg}}(t_{\NNg})
\end{equation} 
is a biregular isomorphism. For $\t=(t_1,\dots,t_\NNg)\in \C^\NNg$, define $\normg{\t}:=\left(|t_1|^2+\dots+|t_\NNg|^2\right)^{\frac12}\in \R_{\geq0}$, and 
let $\normg\cdot:\U_1(g)\to \R_{\geq0}$ be defined by $\normg y:=\normg{\bx_g^{-1}(y)}$. Identifying $\U_1(g)$ with $\Xaff^g$ via~\eqref{eq:KM_affine_charts}, we get a function $\normg\cdot:\Xaff^g\to \R_{\geq0}$.

\def\rch{\tilde\rho}
\def\dilgaff{\FSdil_g}
We say that $\rch\in Y(\T)$ is a \emph{regular dominant integral coweight} if $\<\rch,\delta\>\in\Z_{>0}$ and $\<\rch,\alpha_i\>\in\Z_{>0}$ for all $i\in I$. In this case, we have $\<\rch,\alpha\>\in\Z_{>0}$ for any $\alpha\in\Phire^+$. Let us choose such a coweight $\rch$, and define $\dilgaff:\R_{>0}\times \G/\B\to \G/\B$ by $\dilgaff(t,x\B):=\rch(t)x\B$. 

It follows from~\eqref{eq:KM_torus_rules} that if $g\in \Waff$ and $y\in \U_1(g)$ is such that $\bx_g^{-1}(y)=(t_1,\dots,t_{\NNg})$ then there exist  $k_1,\dots, k_{\NNg}\in\Z_{>0}$ satisfying 
\begin{equation}\label{eq:FSstr_FSdil}
  \normg{\dilgaff(t,y\dg \B)}=\left(t^{k_1}|t_1|^2+\dots+t^{k_{\NNg}}|t_{\NNg}|^2\right)^\frac12 \quad\text{for all $t\in \R_{>0}$.}
\end{equation}

\def\Ogc{\Og^{\C}}
\def\scgc{\sc_g^{\C}}
\def\scoghc{\accentset{\circ}{\sc}_{g,h}^{\C}}

\subsection{Proof of \texorpdfstring{\cref{thm:main_intro2}}{Theorem \ref{thm:main_intro2}}}\label{sec:FS_proof}
By \cref{cor:GP_TNN_space}, $(\GPR,\GPtnn,Q_J)$ is a shellable TNN space in the sense of \cref{dfn:TNNspace}. Thus it suffices to construct a Fomin--Shapiro atlas.

Let $(u,u)\leqJ (v,w)\in Q_J$, and define $f:=(u,u)$, $g:=(v,w)$. Thus we have $\BAcomb(f)=\tul$ and $\BAcomb(g)=\vtw$. Moreover, for the maximal element $\hat1=(\id,\wj)\in Q_J$, we have $\BAcomb(\hat1)=\tmin$. By \cref{BAgeom1}, the map $\BAgeom$ gives an isomorphism $\Cuj\xrasim \Xaffcl_{\BAcomb(\hat1)}\cap \Xaff^{\BAcomb(f)}$. Let $\Ogc\subset \Cuj$ be the preimage of $\Caff_{\BAcomb(g)}\cap \Xaffcl_{\BAcomb(\hat1)}\cap \Xaff^{\BAcomb(f)}$ under $\BAgeom$, and denote by $\Og:=\Ogc\cap \GPR$. Since $\Caff_{\BAcomb(g)}$ is open in $\G/\B$, we see that $\Ogc$ is open in $\Cuj$ which is open in $G/P$, so $\Og$ is an open subset of $\GPR$. By \cref{BAgeom3}, $\Og$ contains $\Startnn_g$, which shows~\axref{FS:Startnn}. Moreover, we claim that $\Og\subset\Star_g$. Indeed, if $h\geqJ f$ but $h\not\geqJ g$ then $\BAcomb(h)\not\leq \BAcomb(g)$. The map $\BAgeom$ sends $\Povar_h\cap \Cuj$ to $\Richaff_{\BAcomb(h)}^{\BAcomb(f)}$, which does not intersect $\Caff_{\BAcomb(g)}$ by~\eqref{eq:KM_Rich_stratif}.

We now define the smooth cone $(\scg,\dilg)$. Throughout, we identify $\Xaff^{\BAcomb(g)}$ with $\C^\Ng$ for $\Ng:=\ell(\BAcomb(g))$ via~\eqref{eq:KM_bx_g}. We set $\scgc:=\Xaffcl_{\BAcomb(\hat1)}\cap \Xaff^{\BAcomb(g)}$ and $\scoghc:=\Richaff_{\BAcomb(h)}^{\BAcomb(g)}$ for $g\leqJ h\in Q_J$. We let $\scg:=\scgc\cap \R^\Ng$ and $\scogh:=\scoghc\cap \R^\Ng$ denote the corresponding sets of real points. Thus~\axref{FS:mnf_strat} follows. The action $\dilgaff$ restricts to $\R^\Ng$, and by~\eqref{eq:FSstr_FSdil}, it satisfies~\axref{SC:str_derivative}. As we discussed in \cref{sec:KM_torus-action}, the action of $\dilgaff$ also preserves both $\scg$ (showing~\axref{SC:R_+_action}) and $\scogh$ (showing~\axref{FS:mnf_dil}).

Finally, we define a map $\isog:\Ogc\to (\Povar_g\cap \Ogc)\times \C^\Ng$ as follows. Let $\isogaff=(\isogaffa,\isogaffb):\Caff_g\xrasim  \Xaff_g\times \Xaff^g$ be the map from \cref{prop:FS_iso}. We let $\isogb:=\isogaffb\circ\BAgeom$, so it sends $\Ogc\to \Caff_{\BAcomb(g)}\to \Xaff^{\BAcomb(g)}\cong \C^\Ng$. By \cref{FS_iso_Rich}, the image of $\isogb$ is precisely $\scgc$. We also let $\isoga:=\BAgeom^{-1}\circ \isogaffa\circ \BAgeom$, so it sends 
\[\Ogc\xrasim \Caff_{\BAcomb(g)}\cap \Xaffcl_{\BAcomb(\hat1)}\cap \Xaff^{\BAcomb(f)}\to \Richaff_{\BAcomb(g)}^{\BAcomb(f)}\xrasim \Povar_g\cap \Ogc.\]
It follows from \cref{BAgeom1} and \cref{prop:FS_iso} that $\isog:=(\isoga,\isogb)$ gives a biregular isomorphism $\Ogc\xrasim (\Povar_g\cap \Ogc)\times \scgc$. All maps in the definition of $\scgc$ are defined over $\R$, so $\isog$ gives a smooth embedding $\Og\to (\RPovar_g\cap \Og)\times \R^\Ng$ with image $(\RPovar_g\cap \Og)\times \scg$. By \cref{lemma:smooth_manifolds_embedded}, we find that $\scg$ is an embedded submanifold of $\R^\Ng$, so we get a diffeomorphism
\[\isog:\Og\xrasim (\RPovar_g\cap \Og)\times \scg.\]
By \cref{BAgeom1} and \cref{FS_iso_Rich}, we find that for $h\geqJ g$, $\isog$ sends $\Povar_h\cap \Og$ to $(\Povar_g\cap \Og)\times \scogh$, showing~\axref{FS:image_of_Yh}. When $xP\in \Povar_g\cap \Og$, we have $\BAgeom(xP)\in \Richaff_{\BAcomb(g)}^{\BAcomb(f)}$, so $\isogaffa(\BAgeom(xP))=\BAgeom(xP)$ and $\isogaffb(\BAgeom(xP))\in\Richaff_{\BAcomb(g)}^{\BAcomb(g)}$. Thus $\isoga(xP)=x$ and $\isogb(xP)=\baseg$, showing~\axref{FS:Yg_id}. We have checked all the requirements of Definitions~\ref{dfn:TNNspace}, \ref{dfn:sc}, and~\ref{dfn:FS_atlas}.\qed 

\def\center{C}
\def\centeraff{\tilde{C}}
\def\pone{^1}

\def\BAgeos{\BAgeo'}
\def\BAgeoms{\BAgeom'}

\def\trnc#1#2{{\operatorname{\tr}_{#1}^{#2}}}
\def\trunc#1#2{#1^{\trnc u{#2}}}
\def\trunci#1#2#3{#1^{\trnc u{#2}}_{#3}}
\def\truncij#1#2#3#4{#1^{\trnc u{#2}}_{#3,#4}}
\def\trmin#1#2#3{\Delta^{\trnc u{#2}}_{#3}(#1)}
\def\pivot#1#2{\theta^u_{#1,#2}}

\def\parz{}
\def\BAP{B(\Ap)}
\def\BAM{B_-(\Am)}
\def\BAPP{B\pone(\Ap)}
\def\BAMM{B\pone_-(\Am)}
\def\UAP{U(\Ap)}
\def\UAM{U_-(\Am)}

\def\inveps#1#2{\epsilon_{#1,#2}}
\def\dft{\tilde{f}}
\def\xt{\tilde{x}}
\def\yt{\tilde{y}}
\def\Glnap{\GL_n(\Ap)}
\def\Glnam{\GL_n(\Am)}
\def\Gminp{\Gmin_+}
\def\Gminm{\Gmin_-}

\def\Gminx#1{\Gmin_{(#1)}}
\def\dett#1{\det\,#1}
\def\vall#1{\val\,#1}
\def\Ctr{C}
\def\Slna{\SL_n(\Acal)}
\def\Glona{\GL_n^\circ(\Acal)}
\def\Glnax#1{\GL_n^\parr{#1}(\Acal)}
\def\Glna{\GL_n(\Acal)}
\def\val{\operatorname{val}}
\def\rkk#1#2#3{\rank(#1;#2,#3)}

\def\av{\operatorname{av}}
\def\Wext{\widehat{W}}
\def\Saffn{\tilde{S}_n}
\def\Saffxn#1{\tilde{S}_{#1,n}}
\def\Saffkn{\Saffxn k}
\def\uS#1#2{#1\cdot #2}
\def\uS#1#2{#1 #2}
\def\uk#1{\uS #1{[k]}}

\def\uphi#1#2{#1 \cdot #2}

\def\nchk{{[n]\choose k}}
\def\Cast{\C^\ast}

\def\GR#1{\left[#1\right|}

\newcommand\scalemath[2]{\scalebox{#1}{\mbox{\ensuremath{\displaystyle #2}}}}

\def\sbl#1#2{\scalemath{0.6}{\left[
\begin{array}{c|c}
#1 \\
\hline
#2
\end{array}
\right]}}

\def\sbld#1#2{\left[\scalebox{0.8}{$
\begin{array}{c|c}
#1 \\
\hline
#2
\end{array}$}
\right]}

\def\svec#1#2{\scalemath{0.6}{\left[
\begin{array}{c}
#1 \\
\hline
#2
\end{array}
\right]}}

\section{The case \texorpdfstring{$G=\SL_n$}{G = SL\textunderscore n}}\label{sec:type_A}
 In this section, we illustrate our construction in type $A$. We mostly focus on the case when $G/P$ is the \emph{Grassmannian} $\Gr(k,n)$ so that $\GPtnn$ is the \emph{totally nonnegative Grassmannian} $\Grtnn(k,n)$ of Postnikov~\cite{Pos}. Throughout, we assume $\K=\C$.

\subsection{Preliminaries}\label{sec:SL_n_preliminaries}
Fix an integer $n\geq 1$ and denote $[n]:=\{1,2,\dots,n\}$. For $0\leq k\leq n$, let $\nchk$ denote the set of all $k$-element subsets of $[n]$.  

Let $G=\SL_n$ be the group of $n\times n$ matrices over $\C$ of determinant $1$. We have subgroups $B, B_-, T, U, U_-\subset G$ consisting of upper triangular, lower triangular, diagonal, upper unitriangular, and lower unitriangular matrices of determinant $1$, respectively. The Weyl group $W$ is the group $S_n$ of permutations of $[n]$, and for $i\in I=[n-1]$, $s_i\in W$ is the simple transposition of elements $i$ and $i+1$. If $w\in W$ is written as a product $w=s_{i_1}\dots s_{i_l}$, then the action of $w$ on $[n]$ is given by $w(j)=s_{i_1}(\cdots(s_{i_l}(j))\cdots )$ for $j\in[n]$. For $S\subset [n]$, we set $\uS wS:=\{w(j)\mid j\in S\}$. For example, if $n=3$ and $w=s_2s_1$ then $w(1)=3$, $w(2)=1$,  $w(3)=2$, and $w\{1,3\}=\{2,3\}$.

For $i\in [n-1]$, the homomorphism $\phi_i:\SL_2\to G$ just sends a matrix $A\in \SL_2$ to the $n\times n$ matrix $\phi_i(A)\in\SL_n$ which has a $2\times 2$ block equal to $A$ in rows and columns $i,i+1$. Thus if $n=3$ then $\ds_1=\smat{0 & -1 & 0\\ 1 & 0 &0\\ 0& 0 & 1}$, $\ds_2=\smat{1 & 0 & 0\\ 0 & 0 &-1\\ 0& 1 & 0}$, and if $w=s_2s_1$ then $\dw=\smat{0 & -1 & 0\\ 0 & 0 &-1\\ 1& 0 & 0}$. In general, given $w\in S_n$, $\dw$ contains a $\pm1$ in row $w(j)$ and column $j$  for each $j\in [n]$, and the sign of this entry is $-1$ if and only if the number of $\pm1$'s strictly below and to the left of it is odd.  In other words, the $(w(j),j)$-th entry of $\dw$ equals $(-1)^{\#\{i<j\mid w(i)>w(j)\}}$.

For $x\in \SL_n$, $x^T$ is just the matrix transpose of $x$, and $x^\iota$ defined in~\eqref{eq:iota} is the ``positive inverse'' given by $(x^\iota)_{i,j} = (-1)^{i+j}(x^{-1})_{i,j}$ for all $i,j$.

For $i\in [n-1]$, the function $\Deltamp_i:\SL_n\to \C$ is the top-left $i\times i$ principal minor, while $\Deltapm_i:\SL_n\to \C$ is the bottom-right $i\times i$ principal minor. The subset $\Gomp=B_-B$ consists precisely of matrices $x\in \SL_n$ all of whose top-left principal minors are nonzero, in agreement with \cref{lemma:G_0}. We define $\Deltamp_n(x)=\Deltapm_n(x):=\dett{x}=1$.

\def\fm_#1{\Delta^\flag_{#1}}
\def\lex{\leq_{\operatorname{lex}}}
\subsection{Flag variety}
The group $B$ acts on $G=\SL_n$ by right multiplication, and $G/B$ is the \emph{complete flag variety in $\C^n$}. It consists of \emph{flags} $\{0\}=V_0\subset V_1\subset \dots \subset V_n=\C^n$ in $\C^n$ such that $\dim V_i=i$ for $i\in [n]$. For a matrix $x\in \SL_n$, the element $xB\in G/B$ gives rise to a flag $V_0\subset V_1\subset \dots \subset V_n$ such that $V_i$ is the span of columns $1,\dots, i$ of $x$. For $k\in[n]$, $S\in \nchk$, and $x\in \SL_n$, we denote by $\fm_S$ the determinant of the $k\times k$  submatrix of $x$ with row set $S$ and column set $[k]$. Thus for each $k\in [n]$, we have a map $\minormap_k:G/B\to \CP^{{n\choose k}-1}$ sending $xB$ to $(\fm_S(x))_{S\in\nchk}$. Here $\nchk$ is identified with the set $W\omega_k$ from \cref{gen_minors_flag}. 

\subsection{Partial flag variety}
For $J\subset [n]$, we have a parabolic subgroup $P\subset G$, and the partial flag variety $G/P$ consists of \emph{partial flags} $\{0\}=V_0\subset V_{j_1}\subset\dots\subset V_{j_l}\subset V_n=\C^n$, where $\{j_1<\dots<j_l\}:=[n-1]\setminus J$ and $\dim V_{j_i}=j_i$ for $i\in [l]$. The projection $\pi_J:G/B\to G/P$ sends a flag $(V_0, V_1,\dots,V_n)$ to $(V_0,V_{j_1},\dots,V_{j_l},V_n)$. When $J=\emptyset$, we have $P=B$ and $G/P=G/B$. We will focus on the ``complementary'' special case:
\[\text{Unless otherwise stated, we assume that $J=[n-1]\setminus \{k\}$ for some fixed $k\in [n-1]$.} \]
 In this case, $G/P$ is the \emph{(complex) Grassmannian $\Gr(k,n)$}, which we will identify with the space of $n\times k$ full rank matrices modulo column operations. Let us write matrices in $\SL_n$ in block form $\sbl{A&B}{C&D}$, where $A$ is of size $k\times k$ and $D$ is of size $(n-k)\times (n-k)$. For a matrix $x=\sbl{A&B}{C&D}\in\SL_n$, we denote by $\GR x:=\svec{A}{C}$ the $n\times k$ submatrix consisting of the first $k$ columns of $x$. Thus every $x\in\SL_n$ gives rise to an element $xP$ of $G/P$ which is a $k$-dimensional subspace $V_k\subset \C^n$ equal to the column span of $\GR x$. The map $\minormap_k$ in this case is the classical \emph{Pl\"ucker embedding} $\minormap_k:\Gr(k,n)\hookrightarrow \CP^{{n\choose k}-1}$.

The set $W^J$ from \cref{sec:parab-subgr-w_j} consists of \emph{Grassmannian permutations}: we have $w\in W^J$ if and only if $w=\id$ or every reduced word for $w$ ends with $s_k$. Equivalently, $w\in W^J$ if and only if $w(1)<\dots <w(k)$ and $w(k+1)<\dots <w(n)$, so the map $w\mapsto \uk w$ gives a bijection $W^J\to \nchk$. The parabolic subgroup $W_J$ (generated by $\{s_j\}_{j\in J}$) consists of permutations $w\in S_n$ such that $\uk w=[k]$, and the longest element $\woj\in W_J$ is given by $(\woj(1),\dots,\woj(n))=(k,\dots,1,n,\dots,k+1)$. The maximal element $\wj$ of $W^J$ is given by $(\wj(1),\dots,\wj(n))=(n-k+1,\dots,n, 1,\dots,n-k)$. We have 
\[U_J=\left\{\sbld{U_k&0}{0&U_{n-k}}\right\},\quad \Uj_-=\left\{\sbld{I_k&0}{C&I_{n-k}}\right\},\quad L_J=\left\{\sbld{A&0}{0&D}\right\},\quad P=\left\{\sbld{A&B}{0&D}\right\},\]
where $U_r$ is an $r\times r$ upper unitriangular matrix, $I_r$ is the $r\times r$ identity matrix, $A\in\SL_k$, $D\in \SL_{n-k}$, and $B$, $C$ are arbitrary $k\times (n-k)$ and $(n-k)\times k$ matrices, respectively.

\subsection{Affine charts}\label{sec:Gr_affine-charts}
We have $\Goj:=\{x\in G\mid \fm_{[k]}(x)\neq0\}$, and for $x=\sbl{A&B}{C&D}\in\Goj$ (such that $\dett{A}=\fm_{[k]}(x)\neq 0$), the factorization $x=[x]_-\pj [x]_0\pj [x]_+\pj$ from \cref{lemma:bracketsJ} is given by 
\begin{equation}\label{eq:Gr_bracketsJ}
\sbld{A&B}{C&D}=\sbld{I_k&0}{CA^{-1}&I_{n-k}}\cdot \sbld{A&0}{0&D-CA^{-1}B}\cdot \sbld{I_k& A^{-1}B}{0&I_{n-k}}.
\end{equation}
The matrix $D-CA^{-1}B$ is called the \emph{Schur complement of $A$ in $x$}.

For $u\in W^J$, the set $\Cuj\subset G/P$ from \cref{sec:isomorphisms} consists of elements $xP$ such that $\fm_{\uk u}(x)\neq 0$.  The (inverse of the) isomorphism~\eqref{eq:Cuj_to_Uj} essentially amounts to computing the reduced column echelon form of an $n\times k$ matrix: if $x\in G$ is such that $xP\in \Cuj$ is sent to $\gj\in \du\Uj_-\du^{-1}$ via~\eqref{eq:Cuj_to_Uj}, then the $n\times k$ matrices $\GR x$ and $\GR{\gj \du}$ have the same column span, and the submatrix of $\GR{\gj \du}$ with row set $\uk u$ is the $k\times k$ identity matrix. Let us say that an $n\times k$ matrix $M$ is in \emph{$\uk u$-echelon form} if its submatrix with row set $\uk u$ is the $k\times k$ identity matrix. 

The matrices $\gj_1\du$ and $\gj_2\du$ from \cref{dfn:gj_and_stuff} are obtained from $\gj \du$ simply by replacing some entries with $0$. Explicitly, let $(M_{i,j}):=\GR{\gj\du}$, $(M'_{i,j}):=\GR{\gj_1\du}$, and $(M''_{i,j}):=\GR{\gj_2\du}$  be the corresponding $n\times k$ matrices. Thus $M_{i,j}=\delta_{i,u(j)}$ for all $i\in \uk u$ and $j\in[k]$, and we have
\[M'_{i,j}=
  \begin{cases}
    M_{i,j}, &\text{if $i\leq u(j)$,}\\
    0, &\text{otherwise,}\\
  \end{cases}\qquad M''_{i,j}=
  \begin{cases}
    M_{i,j}, &\text{if $i\geq u(j)$,}\\
    0, &\text{otherwise,}\\
  \end{cases}\quad \text{for all $i\in[n]$ and $j\in[k]$.}\]
The operation $M\mapsto M'$, which we call \emph{$u$-truncation}, will play an important role.

\begin{example}\label{ex:ssa_1}
Let $G/P=\Gr(\ssak,\ssan)$ and  $u=\ssau\in W^J$, so $\uk u=\ssauk$. We have
\[x=\gj \du=\ssax,\quad \GR{x}=\ssaGRx,\quad \GR{\gj_1\du}=\ssaGRgjAu,\quad \GR{\gj_2\du}=\ssaGRgjBu,\]
where blank entries correspond to zeros.
\end{example}

\subsection{Positroid varieties}\label{sec:positroid-varieties}
We review the background on \emph{positroid varieties} inside $\Gr(k,n)$, which were introduced in~\cite{KLS2}, building on the work of Postnikov~\cite{Pos}. Let $\Saffn$ be the group of \emph{affine permutations}, i.e., bijections $f:\Z\to\Z$ such that $f(i+n)=f(i)+n$ for all $i\in\Z$. We have a function $\av:\Saffn\to \Z$ sending $f$ to $\av(f):=\frac1n\sum_{i=1}^n (f(i)-i)$, which is an integer for all $f\in \Saffn$. For $j\in\Z$, denote $\Saffxn j:=\{f\in\Saffn\mid \av(f)=j\}$. Every $f\in\Saffn$ is determined by the sequence $f(1),\dots,f(n)$, and we write $f$ in \emph{window notation} as $f=[f(1),\dots,f(n)]$.  For $\lch\in\Z^n$, define $\tl\in \Saffn$ by $\tl:=[d_1,\dots,d_n]$, where $d_i=i+n\lch_i$ for all $i\in [n]$. Let $\Boundkn\subset \Saffkn$ be the set of \emph{bounded affine permutations}, which consists of all $f\in\Saffn$ satisfying $\av(f)=k$ and $i\leq f(i)\leq i+n$ for all $i\in \Z$. The subset $\Saffxn0$ is a Coxeter group with generators $s_1,\dots,s_{n-1}, s_n=s_0$, where for $i\in[n]$, $s_i:\Z\to\Z$ sends $i$ to $i+1$, $i+1$ to $i$, and $j$ to $j$ for all $j\not\equiv i,i+1\pmod n$. We let $\leq$ denote the Bruhat order on $\Saffxn0$, and $\ell:\Saffxn0\to\Z_{\geq0}$ denote the length function. We have a bijection $\Saffxn0\to\Saffkn$ sending $(i\mapsto f(i))$ to $(i\mapsto f(i)+k)$, which induces a poset structure and a length function on $\Saffkn$. When $f\leq g$, we write $g\leqop f$, and we will be interested in the poset $(\Boundkn,\leqop)$, which has a unique maximal element $\tauk:=[1+k,2+k,\dots,n+k]$. It is known that $\Boundkn$ is a lower order ideal of $(\Saffkn,\leqop)$. We fix $\lch=1^k0^{n-k}:=(1,\dots,1,0,\dots,0)\in\Z^n$ (with $k$ $1$'s). Then $\tl=[1+n,\dots,k+n,k+1,\dots,n]$ is one of the $\nchk$ minimal elements of $(\Boundkn,\leqop)$.  The group $S_n$ is naturally a subset of $\Saffxn0$, and we have $\tauk=\tl \wji=\tmin$, where $\tmin$ was introduced in \cref{sec:BA_combinatorial-map}. Note that $\la$ is cominuscule; see \cref{rmk:minuscule}.

Given an $n\times k$ matrix $M$ and $i\in[n]$, we let $M_i$ denote the $i$th row of $M$. We extend this to all $i\in\Z$ in such a way that $M_{i+n}=(-1)^{k-1}M_i$ for all $i\in\Z$.  Thus we view $M$ as a periodic $\Z\times k$ matrix.  (The sign $(-1)^{k-1}$ is chosen so that if $M \in \Grtnn(k,n)$, then the matrix with rows $M_i,\dots,M_{i+n-1}$ belongs to $\Grtnn(k,n)$ for all $i\in\Z$; see \cref{sec:Gr_total-positivity}.) 
Every $n\times k$ matrix $M$ of rank $k$ gives rise to a map $f_M:\Z\to\Z$ sending $i\in \Z$ to the minimal $j\geq i$ such that $M_i$ belongs to the linear span of $M_{i+1},\dots,M_j$. It is easy to see that $f_M\in\Boundkn$ and $f_M$ depends only on the column span of $M$.  For $h\in\Boundkn$, the \emph{(open) positroid variety} $\Povar_h\subset \Gr(k,n)$ is the subset $\Povar_h:=\{M\in\Gr(k,n)\mid f_M=h\}$. Its Zariski closure inside $\Gr(k,n)$ is $\Povarcl_h=\bigsqcup_{g\leqop h} \Povar_g$; see~\cite[Theorem~5.10]{KLS2}. 

For $h\in\Boundkn$, define the \emph{Grassmann necklace} $\Ical_h=(I_a)_{a\in\Z}$ of $h$ by
\begin{equation}\label{eq:Gr_Grneck}
  I_a:=\{h(i)\mid i<a,\ h(i)\geq a\}\quad\text{for $a\in\Z$.}
\end{equation}
Then $I_a$ is a $k$-element subset of $[a,a+n)$, where for $a\leq b\in\Z$ we set $[a,b):=\{a,a+1,\dots,b-1\}$. For $a\leq b\in \Z$ and $M\in\Gr(k,n)$, define $\rkk Mab$ to be the rank of the submatrix of $M$ with row set $[a,b)$. For $a, b\in \Z$ and $h\in\Saffn$, define $r_{a,b}(h):=\#\{i<a\mid h(i)\geq b\}$. We describe two well-known characterizations of open positroid varieties; see~\cite[\S5.2]{KLS2}.

\begin{proposition} Let $h\in\Boundkn$ and let  $\Ical_h=(I_a)_{a\in\Z}$ be its Grassmann necklace.
\begin{theoremlist}
\item\label{Gr_Grneck_charact} The set $\Povar_h$ consists of all $M\in\Gr(k,n)$ such that for each $a\in\Z$, $I_a$ is the lexicographically minimal $k$-element subset $S$ of $[a,a+n)$ such that the rows $(M_i)_{i\in S}$ are linearly independent.
\item\label{Gr_rank_conditions} For $M\in\Gr(k,n)$, we have $M\in \Povar_h$ if and only if
\begin{equation}\label{eq:Gr_rank_conditions}
k-\rkk Mab=r_{a,b}(h)\quad \text{for all $a\leq b\in\Z$}.
\end{equation}
\end{theoremlist}
\end{proposition}
\noindent We use \emph{window notation} for Grassmann necklaces as well, i.e., we write $\Ical_h=[I_1,\dots,I_n]$.

Recall that we have fixed $\lch=1^k0^{n-k}\in\Z^n$. For $(v,w)\in Q_J$, define $f_{v,w}\in\Saffn$ by 
\begin{equation}\label{eq:KLS_f_v_w}
f_{v,w}:=v \tl w^{-1}.
\end{equation}

\begin{theorem}[{\cite[Propositions 3.15 and 5.4]{KLS2}}]\label{thm:Povar_PR}
  The map $(v,w)\mapsto f_{v,w}$ gives a poset isomorphism 
\[(Q_J,\leqJ)\xrasim (\Boundkn,\leqop).\] 
For $(v,w)\in Q_J$, we have $\PR_v^w=\Povar_{f_{v,w}}$ and $\PRcl_v^w=\Povarcl_{f_{v,w}}$ as subsets of $G/P=\Gr(k,n)$.
\end{theorem}

\begin{example}\label{ex:Gr_KLS_f_v_w}
There are $n$ positroid varieties of codimension $1$, each given by the condition $\fm_{\{i-k+1,\dots,i\}}=0$ for some $i\in[n]$. Indeed, the top element $(\id,\wj)\in Q_J$ covers $n$ elements, namely $(s_i,\wj)$ for $i\in[n-1]$, together with $(\id, s_{n-k}\wj)$. In the former case we have $f_{s_i,\wj}=s_i\tmin$, which corresponds to the variety $\fm_{\{i-k+1,\dots,i\}}=0$. In the latter case we have $f_{\id,s_{n-k}\wj}=\tmin s_{n-k}$, which corresponds to the variety $\fm_{\{n-k+1,\dots,n\}}=0$.
\end{example}

\begin{example}\label{ex:Gr_neck_v[k]}
One can check directly from~\eqref{eq:KLS_f_v_w} and~\eqref{eq:Gr_Grneck} that the first element of the Grassmann necklace of $f_{v,w}$ is $I_1=\uk v$. Similarly, $\uk w=\{i\in[n]\mid f_{v,w}(i)>n\}$.
\end{example}

\begin{example}\label{ex:Le}
Elements of $\Boundkn$ and $Q_J$ are in bijection with \emph{\Le -diagrams} of~\cite{Pos}. The bijection between $Q_J$ and the set of \Le -diagrams is described in~\cite[\S19]{Pos}: the pair $(v,w)\in Q_J$ gives rise to a \Le -diagram whose \emph{shape} is a Young diagram inside a $k\times (n-k)$ rectangle, corresponding to the set $w[k]$. The squares of the \Le -diagram correspond to the terms in a reduced expression for $w$, as shown in Figure~\ref{fig:Le} (top left): the box with coordinates $(i,j)$ in matrix notation is labeled by $s_{k+j-i}$, and we form the expression by reading boxes from right to left, bottom to top. The terms in the positive subexpression for $v$ inside $w$ correspond to the squares of the \Le -diagram that are \emph{not} filled with dots; see Figure~\ref{fig:Le} (bottom left). Thus the bijection of \cref{thm:Povar_PR} can be pictorially represented as in Figure~\ref{fig:Le} (right). We refer to~\cite[\S19]{Pos} or~\cite[Appendix~A]{Wil} for the precise description. For the example in Figure~\ref{fig:Le}, we have $v=\ssbvp$, $w=\ssbwp$, and $f_{v,w}=\ssbh$ in window notation, which is obtained by following the strands in Figure~\ref{fig:Le} (right) from top to bottom.
\end{example}

\begin{figure}
\figureLe
\caption{\label{fig:Le} A \Le -diagram (bottom left), the labeling of its squares by simple transpositions (top left), and the result of applying the bijection of \cref{thm:Povar_PR} (right). See \cref{ex:Le} for details.}
\end{figure}

\subsection{Polynomial loop group}
We explain how the construction in \cref{sec:kac-moody-groups} applies to the case $G/P=\Gr(k,n)$. Recall that $\Acal:= \C[z,z^{-1}]$. Let $\Glna$ denote the \emph{polynomial loop group of $\GL_n$}, consisting of $n\times n$ matrices with entries in $\Acal$ whose determinant is a nonzero Laurent monomial in $z$, i.e., an invertible element of $\Acal$. (We use $\Glna$ instead of $\Slna$ as the constructions are combinatorially more elegant.)
We have a group homomorphism $\val:\Glna\to \Z$ sending $x\parz \in\Glna$ to $j\in \Z$ such that $\dett{x\parz }=cz^{-j}$ for some $c\in \Cast$, and we let $\Glnax j:=\{x\parz \in\Glna\mid \vall{x\parz }=j\}$.  The subgroups $\Glnap$ and $\Glnam$ are contained inside the group $\Glnax0$ of matrices whose determinant belongs to $\Cast$.  We have subgroups $\UAP:=\evp^{-1}(U)$, $\UAM:=\evm^{-1}(U_-)$, $\BAP:=\evp^{-1}(B)$ and $\BAM:=\evm^{-1}(B_-)$ of $\Glnax0$. Thus in the notation of \cref{sec:kac-moody-groups} for $G=\SL_n$, we have $\G=\Slna\subsetneq \Glnax0$, $\B=\Slna\cap \BAP\subsetneq \BAP$, $\U=\UAP$, and $\U_-=\UAM$.

To each matrix $x\parz \in \Glna$, we associate a $\Z\times \Z$ matrix $\xt=(\xt_{i,j})_{i,j\in\Z}$ that is uniquely defined by the conditions
\begin{enumerate}
\item $\xt_{i,j}=\xt_{i+n,j+n}$ for all $i,j\in\Z$, and
\item the entry $x_{i,j}(z)$ equals the finite sum $\sum_{d\in \Z}\xt_{i,j+dn}z^d$ for all $i,j\in[n]$.
\end{enumerate}
One can check that if $x = x_1 x_2$, then $\xt = \xt_1 \xt_2$. With this identification, the subgroups $\U$, $\U_-$, $\BAP$, and $\BAM$ have a very natural meaning. For example, $x\parz \in\Glna$ belongs to $\U$ if and only if $\xt_{i,j}=0$ for $i>j$ and $\xt_{i,i}=1$ for all $i\in\Z$. Similarly, $\BAP$ consists of all elements $x\parz \in\Glna$ such that $\xt_{i,j}=0$ for $i>j$ and $\xt_{i,i}\neq0$ for all $i\in\Z$.

To each affine permutation $f\in \Saffkn$, we associate an element $\df\parz \in\Glna$ so that the corresponding $\Z\times \Z$ matrix $\dft$ satisfies $\dft_{i,j}=1$ if $i=f(j)$ and $\dft_{i,j}=0$ otherwise, for all $i,j\in\Z$. 
 In other words, if for $i,j\in[n]$ there exists $d\in\Z$ such that $f(j)=i+dn$ then $\df_{i,j}(z):= z^{-d}$, and otherwise $\df_{i,j}(z):=0$. Observe that $\vall{\df\parz }=k$ for all $f\in\Saffkn$, and thus $\df\in \Glnax k$. Recall that we have fixed $\l=1^k0^{n-k}\in\Z^n$. We obtain $\dtl\parz =\diag\left(\frac1z,\dots,\frac1z,1,\dots,1\right)$ with $k$ entries equal to $\frac1z$, and for $u\in W^J$, we therefore get $\dtul\parz =\diag(c_1,\dots,c_n)$, where $c_i=\frac1z$ for $i\in \uk u$ and $c_i=1$ for $i\notin \uk u$.

\def\Flaffk{\Glnax k/\BAP}
\def\Flaffx#1{\Glnax{#1}/\BAP}
\subsection{Affine flag variety}
The quotient $\Flaffk$ is isomorphic to the affine flag variety $\G/\B$ of \cref{sec:kac-moody-groups} for the case $G = \SL_n$.  Indeed, $\Glnax0$ acts simply transitively on $\Glnax k$ and we clearly have $\Glnax0/\BAP\cong \G/\B$. For $f\leqop h\in \Saffkn$ and $g\in \Saffkn$, we have subsets $\Xaff^f,\Xaff_h,\Richaff_h^f,\Caff_g\subset \Flaffk$ defined by 
\begin{align*}
&\Xaff^f:=\BAP\cdot \df\parz \cdot \BAP/\BAP, &\Xaff_h:=\BAM\cdot \dh\parz \cdot \BAP/\BAP,\\
&\Richaff_h^f:=\Xaff_h\cap \Xaff^f, &\Caff_g:=\dg\parz \cdot \BAM\cdot \BAP/\BAP.
\end{align*}

Let us now calculate the map $\BAgeo$ from~\eqref{eq:BA_geo}. Recall that it sends $xP\in \Cuj$ to $\gj_1\cdot \dtul\parz \cdot (\gj_2)^{-1}$. Assuming as before that $x=\gj \du\in \du\Uj_-$, consider the corresponding $n\times k$ matrix $(M_{i,j}):=\GR{x}$ in $\uk u$-echelon form.

\begin{proposition}
The matrix $y\parz :=\BAgeo(xP)\in\Glnax k$ is given for all $i,j\in[n]$ by
\begin{equation}\label{eq:Snider}
y_{i,j}(z)=
\begin{cases}
 \delta_{i,j}, &\text{if $j\notin \uk u$,}\\
 -M_{i,s}, &\text{if $i>j$ and $j=u(s)$ for some $s\in[k]$,}\\
\frac{M_{i,s}}{z} , &\text{if $i\leq j$ and $j=u(s)$ for some $s\in[k]$.}\\
\end{cases}
\end{equation}
\end{proposition}
\begin{proof}
This follows by directly computing the product $\gj_1\cdot \dtul\parz \cdot (\gj_2)^{-1}$. 
\end{proof}

\begin{example}\label{ex:ssa_2}
In the notation of \cref{ex:ssa_1}, we have
\begin{equation}\label{eq:ssa_2}
y\parz =\gj_1\cdot \dtul\parz \cdot (\gj_2)^{-1}=\ssagjA\cdot \ssadtul\cdot \ssagjBi=\ssabageom.
\end{equation}
\end{example}

\begin{remark}\label{rmk:Snider}
The map $\BAgeom: xP\mapsto \gj_1\cdot \dtul\parz \cdot (\gj_2)^{-1}\cdot \BAP$ is a slight variation of a similar embedding of~\cite{Sni} which we denote $\BAgeoms$.  We have $\BAgeoms(xP)=\gj_1\cdot \dtul\parz \cdot \gj_2\cdot \BAP$, and the corresponding matrix $y'=\BAgeos(xP) :=\gj_1\cdot \dtul\parz \cdot \gj_2$ is given by~\eqref{eq:Snider} except that $-M_{i,s}$ should be replaced by $M_{i,s}$.   Thus $y'$ is obtained from $y$ by substituting $z\mapsto -z$ and then changing the signs of all columns in $\uk u$. In particular, $y'$ and $y$ are related by an element of the affine torus from \cref{sec:KM_torus-action}. 

\Cref{snider} below is due to Snider~\cite{Sni}. \Cref{BAgeom1} generalizes Snider's result to arbitrary $G/P$. The advantage of introducing the sign change in our map $\BAgeom$ is that it is better suited for applications to total positivity: for instance, the analog of \cref{BAgeom3} does not hold for the map $\BAgeoms$.
\end{remark}

We give a standard convenient characterization of $\Xaff_h$ using \emph{lattices}. For each $x\parz \in \Glna$ and column $a\in \Z$, we introduce a Laurent polynomial $x_a(t)\in \C[t,t^{-1}]$ defined by $x_a(t):=\sum_{i\in\Z}\xt_{i,a} t^{i}$, and an infinite-dimensional linear subspace $L_a(x)\subset\C[t,t^{-1}]$ given by $L_a(x):=\Span\{x_{j}(t)\mid j< a\}$, where $\Span$ denotes the space of all finite linear combinations.  For $b\in\Z$, define another linear subspace $E_b\subset \C[t,t^{-1}]$ by $E_b:=\Span\{t^{i}\mid i\geq b\}$. Finally, for $a,b\in \Z$, define $r_{a,b}(x)\in\Z$ to be the dimension of $L_a(x)\cap E_b$. In other words, $r_{a,b}(x)$ is the dimension of the space of $\Z\times 1$ vectors that have zeros in rows $b-1,b-2,\dots$ and can be obtained as finite linear combinations of columns $a-1,a-2,\dots$ of $\xt$. Recall from \cref{sec:positroid-varieties} that for $a,b\in\Z$ and $h\in\Saffn$, we define $r_{a,b}(h):=\#\{i<a\mid h(i)\geq b\}$.

 \begin{lemma}
Let $x\in \Glnax d$ and $h\in\Saffxn d$ for some $d\in\Z$. Then
\begin{equation}\label{eq:Gr_Xaff_rank_cond}
x\cdot \BAP\in \Xaff_h\quad \text{if and only if}\quad  r_{a,b}(x)=r_{a,b}(h)\quad\text{for all $a,b\in\Z$.}
\end{equation}
 \end{lemma}
\begin{proof}
It is clear that $r_{a,b}(x)=r_{a,b}(h)$ when $x=\dh$. One can check that $r_{a,b}(y_-xy_+)=r_{a,b}(x)$ for all $x\in\Glnax d$, $y_-\in\BAM$, $y_+\in\BAP$, and $a,b\in\Z$. This proves~\eqref{eq:Gr_Xaff_rank_cond} since $\Flaffx d=\bigsqcup_{h\in\Saffxn d}\Xaff_h$ by~\eqref{eq:KM_Bruhat_Birkhoff}.
\end{proof}

\begin{remark}
  A \emph{lattice} $\Lcal$ is usually defined  (see e.g.~\cite[\S13.2.13]{Kum}) to be a free $\C[[z]]$-submodule of $\C((t)) \cong \C((z))^n$ (where $z = t^n$) satisfying $\Lcal \otimes_{\C[[z]]} \C((z)) \cong \C((z))^n$.  The $\C[[z]]$-submodule generated by our $L_a(x)$ gives a lattice $\Lcal_a(x)$ in the usual sense.
\end{remark}

\begin{definition}\label{dfn:trunc_matrix}
Suppose we are given an $n\times k$ matrix $M$ in $\uk u$-echelon form. Recall that we have defined the row $M_a$ for all $a\in \Z$ in such a way that $M_{a+n}=(-1)^{k-1}M_a$. For $a\in\Z$ and $j\in[k]$, denote by $\pivot aj\in[a,a+n)$ the unique integer that is equal to $u(j)$ modulo $n$. Define the \emph{$u$-truncation} $\trunc Ma$ of $M$ to be the $[a,a+n)\times k$ matrix $\trunc Ma=(\truncij Maij)$ such that for $i\in [a,a+n)$ and $j\in[k]$, the entry $\truncij Maij$ is equal to $M_{i,j}$ if $i\leq \pivot aj$ and to $0$ otherwise; see \cref{ex:ssa_4}. Thus $\trunc Ma$ is obtained from the matrix with rows $M_a,\dots,M_{a+n-1}$ by setting an entry to $0$ if it is below the corresponding $\pm1$ in the same column, and we label its rows by $a, \dots, a+n-1$ rather than by $1, \dots, n$. For example, if $x=\gj \du$ and $M=\GR x$ then $\trunc M1=\GR{\gj_1\du}$; cf.~\cref{ex:ssa_1}. 
\end{definition}

\begin{lemma}
Let $x=\gj\du\in\du\Uj_-$, $M:=\GR{x}$, and $y:=\BAgeo(xP)$. Then for all $a\in\Z$, the space $L_a(y)$ has a basis
\begin{equation}\label{eq:Gr_L_a_basis}
\{t^i\mid i<a\}\sqcup \{P_1(t),\dots,P_k(t)\},\quad \text{where}\quad P_s(t):=\sum_{i=a}^{a+n-1} \truncij Mais t^i\quad\text{for $s\in[k]$}.
\end{equation}
\end{lemma}
\begin{proof}
  For a subset $S\subset \Z$, define $S+n\Z:=\{j+in\mid j\in S,\  i\in\Z\}$. The space $L_a(y)$ is the span of $y_j(t)$ for all $j< a$. If $j\notin \uk u+n\Z$ then $y_j(t)=t^j$ by definition. If $j\in \uk u+n\Z$ then $y_{j-n}(t)=t^j+\sum_{j-n<i<j} c_it^i$, where $c_i$ is zero for $i\in \uk u+n\Z$. It follows that $L_a(y)$ contains $t^i$ for all $i<a$. Moreover, the only indices $j< a$ such that $y_{j}(t)\notin \Span\{t^i\mid i<a\}$ are those that belong to $[a-n,a)\cap (\uk u+n\Z)$. Let $j\in[a-n,a)\cap (\uk u+n\Z)$ be such an index, and let  $s\in [k]$ be the unique index such that $u(s)\in j+n\Z$. Then clearly $y_{j}(t)\pm P_s(t)\in \Span\{t^i\mid i<a\}$, where the sign depends on the parity of $\frac{j-u(s)}{n}\in\Z$. Thus $P_s(t)\in L_a(y)$ for all $s\in[k]$, and $L_a(y)$ is the span of $\{t^i\mid i<a\}\sqcup \{P_1(t),\dots,P_k(t)\}$. Since the Laurent polynomials $P_s(t)$ have different degrees, they must be linearly independent.
\end{proof}

We give an alternative proof of \cref{BAgeom1} for the case $G/P=\Gr(k,n)$. 
\begin{proposition}\label{snider}
For $h\in\Boundkn$ such that $\tul\leqop h$, the map $\BAgeom$ gives isomorphisms 
\[\BAgeom:\Cuj\xrasim \Xaff^{\tul},\quad \BAgeom: \Cuj\cap\Povar_h\xrasim \Richaff_h^{\tul}.\]
\end{proposition}
\begin{proof}
  It is clear from~\eqref{eq:Snider} that we have a biregular isomorphism $\Uj_1\times \Uj_2\xrasim \U_1(\tul)$ sending $(\gj_1,\gj_2)$ to $\gj_1\cdot \dtul (\gj_2)^{-1}\dtul^{-1}$. Thus the map $(\gj_1,\gj_2)\mapsto \gj_1\cdot \dtul\cdot (\gj_2)^{-1}\cdot \BAP$ gives a parametrization of $\Xaff^{\tul}$; see~\eqref{eq:KM_affine_charts}. Since $\Cuj=\bigsqcup_{h\in\Boundkn}( \Cuj\cap\Povar_h)$, let us fix $h\in\Boundkn$ and $x=\gj\du\in \du\Uj_-$. Define $M:=\GR{x}$ and $y:=\BAgeo(xP)$. By~\eqref{eq:Gr_rank_conditions}, we have $M\in\Povar_h$ if and only if $k-\rkk Mab=r_{a,b}(h)$ for all $a\leq b\in\Z$. By~\eqref{eq:Gr_Xaff_rank_cond}, we have $y\cdot \BAP\in \Xaff_h$ if and only if $r_{a,b}(y)=r_{a,b}(h)$ for all $a, b\in\Z$. If $a>b$ then $r_{a,b}(y)=r_{a,b+1}(y)+1$ by~\eqref{eq:Gr_L_a_basis} and $r_{a,b}(h)=r_{a,b+1}(h)+1$ since $h\in\Boundkn$ satisfies $h^{-1}(b)\leq b$, so $h^{-1}(b)<a$. We have shown that $y\cdot \BAP\in \Xaff_h$ if and only if $r_{a,b}(y)=r_{a,b}(h)$ for all $a\leq b\in\Z$. Thus it suffices to show 
\begin{equation}\label{eq:rab_y_M_need}
r_{a,b}(y)+\rkk Mab=k\quad \text{for all $a\leq b\in\Z$.}
\end{equation}
By~\eqref{eq:Gr_L_a_basis}, $r_{a,b}(y)$ is the dimension of $\Span\{P_1(t),\dots,P_k(t)\}\cap E_b$. By the rank-nullity theorem, $k-r_{a,b}(y)$ is the rank of the submatrix of $\trunc Ma$ with row set $[a,b)$, which is obtained by downward row operations from the submatrix of $M$ with row set $[a,b)$. This shows~\eqref{eq:rab_y_M_need}.
\end{proof}

\begin{remark}
By \cref{BAgeom1}, the image of $\BAgeom$ is $\Xaffcl_{\tmin}\cap \Xaff^{\tul}$, where $\tmin=\tl \wji$. But recall from \cref{sec:positroid-varieties} that $\tl\wji=\tauk$, and since $\Xaff_{\tauk}$ is dense in $\Flaffk$, we find that $\Xaffcl_{\tmin}\cap \Xaff^{\tul}=\Xaff^{\tul}$.
\end{remark}

\begin{example}\label{ex:ssa_3}
Suppose that $x=\gj\du$ is given as in \cref{ex:ssa_1}, so that $y\parz=\BAgeo(xP)$ is the matrix from \cref{ex:ssa_2}. It is clear that $y\parz \in \BAP\cdot \dtul$ regardless of the values of $x_1,x_2,x_3,x_4$, and therefore $y\cdot \BAP$ belongs to $\Xaff^{\tul}$. We can try to factorize $y$ as an element of $\BAM\cdot \dtauk\cdot \BAP$:
\[y\parz =\ssaBm\cdot \ssatauk\cdot \ssaBp.\]
This factorization makes sense only when all denominators on the right-hand side are nonzero, which shows that $y\parz\cdot  \BAP\in \Richaff_{\tauk}^{\tul}$ whenever the minors $\fm_{12}(x)=x_2$, $\fm_{23}=x_1x_4-x_2x_3$, and $\fm_{34}=x_3$ are nonzero.  Observe also that $\fm_{14}(x)=1$.  Thus $y\parz\cdot  \BAP\in \Richaff_{\tauk}^{\tul}$ precisely when $xP\in \Povar_{\tauk}$, where $\tauk=[3,4,5,6]$ in window notation. If  $x_2=0$ then $xP\in \Povar_{h}$ for $h=[2,4,5,7]$. In this case, we have
\[\dh\parz =\ssahx,\quad  y\parz |_{x_2=0}=\ssaBmx\cdot \ssahx\cdot \ssaBpx.\]
Therefore $y\parz|_{x_2=0}$ belongs to $\Richaff_{h}^{\tul}$ whenever $x_1,x_3,x_4\neq0$. Observe that the Grassmann necklace of $h$ is given by $\Ical_h=[\{1,3\},\{2,3\},\{3,4\},\{4,5\}]$ in window notation, and the corresponding flag minors of $x|_{x_2=0}$ are given by $\fm_{13}=x_4$, $\fm_{23}=x_1x_4$, $\fm_{34}=x_3$, and $\fm_{14}=1$, in agreement with \cref{snider}.
\end{example}

\subsection{Preimage of \texorpdfstring{$\Caff_g$}{C\textunderscore g}}\label{sec:Gr_preimage-caff_g}
For this section, we fix $\tul\leqop g\in\Boundkn$. We would like to understand the preimage of $\left(\Xaff^{\tul}\cap \Caff_g \right)\subset\Flaffk$ under the map $\BAgeom$. For a set $S\subset [a,a+n)$ of size $k$, define $\trmin MaS$ to be the determinant of the $k\times k$ submatrix of $\trunc Ma$ with row set $S$. Let $\Ical_g=(I_a)_{a\in\Z}$ be the Grassmann necklace of $g$. 

\begin{proposition}\label{thm:Snider_preimage_Cg}
Suppose that $xP\in\Cuj$ and let $M:=\GR{\gj\du}$. Then $\BAgeom(xP)\in\Caff_g$ if and only if $\trmin Ma{I_a}\neq0$ for all $a\in[n]$.
\end{proposition}
\begin{proof}
 Let $h\in\Saffn$ be the unique element such that $\dg^{-1}\BAgeom(xP)$ belongs to $\Xaff_h$, so that $\BAgeom(xP)\in\Caff_g$ if and only if $h=\id$. Since $\vall{\BAgeo(xP)}=k$ and $\vall{\dg^{-1}}=-k$, we get $h\in\Saffxn0$. Hence $h=\id$ if and only if $r_{a,a}(h)=0$ for all $a\in\Z$. Let $y:=\BAgeo(xP)$ and $y':=\dg^{-1}y$. Then for $a\in\Z$, we get $L_a(y')=g^{-1}L_a(y)$, where $g^{-1}$ acts on $\C[t,t^{-1}]$ as a linear map sending $t^j$ to $t^{g^{-1}(j)}$. In particular, $L_a(y')\cap E_a=(g^{-1}L_a(y))\cap E_a$ has the same dimension as $L_a(y)\cap gE_a$. Let us define $H_a:=\{t^i\mid i\geq a\}$, so $E_a=\Span (H_a)$ and $gE_a=\Span (gH_a)$. Since $g(i)\geq i$ for all $i\in\Z$, it follows from~\eqref{eq:Gr_Grneck} that $gH_a=H_a\setminus \{t^j\}_{j\in I_a}$. Therefore by~\eqref{eq:Gr_L_a_basis}, $L_a(y)\cap gE_a=\{0\}$ if and only if $\Span\{P_j(t)\}_{j\in[k]}\cap \Span \left(H_a\setminus \{t^j\}_{j\in I_a}\right)=\{0\}$, which happens precisely when the submatrix of $\trunc Ma$ with row set $I_a$ is nonsingular, i.e., $\trmin Ma{I_a}\neq0$.
\end{proof}

\begin{example}\label{ex:ssa_4}
Suppose that $x$ is the matrix from \cref{ex:ssa_1}, so that  $y\parz :=\BAgeo(xP)$ is given in \cref{ex:ssa_2}. We have
\[M=\ssaGRx,\quad \trunc M1=\ssatrMA,\quad \trunc M2=\ssatrMB,\quad \trunc M3=\ssatrMC,\quad \trunc M4=\ssatrMD.\]
Suppose that $g=[2,4,5,7]$ as in  \cref{ex:ssa_3}, so that its Grassmann necklace is $\Ical_g=[\{1,3\},\{2,3\},\{3,4\},\{4,5\}]$ in window notation. This gives
\begin{equation}\label{eq:trmin_ex}
\trmin M1{13}=\ssatrminA,\quad \trmin M2{23}=\ssatrminB,\quad \trmin M3{34}=\ssatrminC,\quad \trmin M4{45}=\ssatrminD.
\end{equation}
On the other hand, recall from \cref{ex:ssa_3} that $\dg\parz =\ssahx$. Since $y\parz \in\Caff_g$ if and only if $\dg^{-1}\parz y\parz \in \BAM\cdot \BAP$, we can factorize it as
\begin{equation}\label{eq:ssa_4_gi_yz}
\dg^{-1}\parz  y\parz =\ssagy=\ssaBmg\cdot \ssaBpg.
\end{equation}
Again, this is valid only when the denominators in the right-hand side are nonzero. Thus we see that $\dg^{-1}\parz  y\parz $ belongs to $\BAM\cdot \BAP$ precisely when all minors in~\eqref{eq:trmin_ex} are nonzero, in agreement with \cref{thm:Snider_preimage_Cg}.
\end{example}

\subsection{Fomin--Shapiro atlas}\label{sec:Gr_fomin-shap-proj}
The computation in~\eqref{eq:ssa_4_gi_yz} can now be used to find the maps $\isog$ and $\dilg$. As in \cref{sec:FS_proof}, denote by $\Og\subset \Cuj$ the preimage of $\Caff_g\cap \Xaff^{\tul}$ under $\BAgeom$. Thus for our running example, $\Og$ is the subset of $\Cuj$ where all minors in~\eqref{eq:trmin_ex} are nonzero. We are interested in the map $\isog=(\isoga,\isogb):\Og\to (\Povar_g\cap\Og)\times \scg$ from~\eqref{eq:FSiso_intro}, defined in \cref{sec:FS_proof}. The first component is $\isoga=\BAgeom^{-1}\circ \isogaffa\circ \BAgeom$, where $\isogaff:\Caff_g\cap \Xaff^{\tul}\xrasim  \Richaff_g^{\tul}\times \Xaff^g$ is the map from \cref{FS_iso_biregular}. In order to compute it, we consider the factorization $\dg^{-1}\parz y\parz =y_-\cdot y_+\in \U_-\cdot \BAP$ from~\eqref{eq:ssa_4_gi_yz}. The group $\U_1(g)$ is $1$-dimensional since $\ell(g)=1$, and the corresponding element $y_1\parz \in\U_1(g)$ from \cref{FS_iso_biregular} can be computed by factorizing $\dg\parz  y_-\dg\parz ^{-1}$ as an element of $\U_1(g)\cdot \U_2(g)$:
\[\dg\parz  y_-\dg\parz ^{-1}=\ssagUigi=\ssayAi\cdot \ssaxB,\quad  y_1=\ssayA.\]
Therefore the map $\isogaffa$ sends $y\parz\cdot  \BAP$ from~\eqref{eq:ssa_2} to
\begin{equation*}
y_1y\parz \cdot \BAP=\ssayAbageom\cdot \BAP=\ssayAbageomB\cdot \BAP.
\end{equation*}
Applying $\BAgeom^{-1}$ to the right-hand side, we see that the map $\isoga$ is given by
\[\isoga:\Og\to \Povar_g\cap \Og,\quad \ssaGRx\mapsto \ssapigx.\]
Similarly, factorizing $\dg\parz  y_-\dg\parz ^{-1}$ as an element of $\U_2(g)\cdot \U_1(g)$, we find that 
\[\isogaffb(y\cdot \BAP)=y_2y\cdot \BAP=\ssayBbageomA\cdot \dg\cdot  \BAP.\]
We have $\Ng=\ell(g)=1$, and the map $\isogb:\Og\to \scg=\R$ sends  $\ssaGRx$ to $\frac{x_2}{x_4}$.

\subsubsection{Torus action}\label{sec:Gr_torus-action}
We compute the maps from \cref{sec:KM_torus-action}. Let $\rch\in Y(\T)$ denote the group homomorphism $\rch:\Cast\to \Cast\times T$  sending $\tht$ to $\rch(\tht):=(\tht^n,\diag(\tht^{n-1},\dots,\tht,1))$. If $x\in \Glna$ is represented by a $\Z\times \Z$ matrix $(\xt_{i,j})$ then the element $y:=\rch(\tht) x\rch(\tht)^{-1}\in\Glna$ satisfies $\yt_{i,j}=\tht^{j-i}\xt_{i,j}$ for all $i,j\in\Z$. 

\begin{example}
  Continuing the example above, we find that 
  \[\rch(t) \cdot y_2y\cdot \rch(t)^{-1}\cdot \BAP=\ssayBbageomAt\cdot \dg\cdot  \BAP,\quad\text{and}\quad \normg{y_2y\cdot \BAP}=\frac{|x_2|}{|x_4|}.\]
Thus the action of $\dilg$ on $\scg$ is given by $\dilg\left(t,\frac{x_2}{x_4}\right)=\frac{tx_2}{x_4}$. The pullback of this action to $\Og\subset \Cuj$  via $\isog^{-1}$ preserves $x_3$, $x_4$, and $x_1x_4-x_2x_3$ (since it preserves $\isoga(x)$), but multiplies $\frac{x_2}{x_4}$ by $t$. Therefore it is given by
\[\isog^{-1}\circ \left(\id\times \dilg(t,\cdot)\right)\circ \isog: \Og\to\Og,\quad \ssaGRx\mapsto  \ssaGRxt.\] 
\end{example}

\subsection{\texorpdfstring{The maps \texorpdfstring{$\hjmap$}{kappa} and \texorpdfstring{$\zetamap$}{zeta}}{The maps kappa and zeta}}
The subset $\du\Goj$ consists of matrices $x\in G$ such that $\fm_{\uk u}(x)\neq0$. Suppose that $x=\gj\du\in \du\Uj_-$. Then the elements $\gj_1\du$ and $\gj_2\du$ are obtained from $x$ by setting some entries to zero; see \cref{sec:Gr_affine-charts}. The map $x\mapsto \hjx x$ from \cref{dfn:gj_and_stuff} sends $x=\gj\du$ to $\gj_1\du$, e.g., if $\GR{x}=\ssaGRx$ then $\GR{\hjx x}=\ssaGRgjAu$ as in \cref{ex:ssa_1}. Comparing this to \cref{sec:Gr_preimage-caff_g}, we see that if $M=\GR{x}$ is in $\uk u$-echelon form then $\GR{\hjx x}$ is the $u$-truncation $\trunc M1$. 

Now let $(v,w)\in \QJfilter uu$, so $\tul\leqop g:=f_{v,w}$, and define $\Ical_g:=(I_a)_{a\in\Z}$. The set $\Guvbig$ from~\eqref{eq:Guvbig} consists of $x\in G$ such that $\fm_{\uk u}(x)\neq0$ and $\fm_{\uk v}(\hjx x)\neq0$. But recall from \cref{ex:Gr_neck_v[k]} that $\uk v=I_1$. Thus 
\begin{equation}\label{eq:Gr_Guv}
\Guvbig=\left\{x\in G\mid \fm_{\uk u}(x)\neq0\text{ and } \trmin M1{I_1}\neq0\right\},\quad \text{where $M:=\GR{\gj\du}$.}
\end{equation}

\begin{example}\label{ex:ssa_5}
We compute the maps $\hjmap$ and $\zetamap$ for our running example. Suppose that $x=\gj\du$ is given as in \cref{ex:ssa_1}, and let $g=[2,4,5,7]$ as in \cref{ex:ssa_4}. Then $g=s_2\tauk$, so under the correspondence~\eqref{eq:KLS_f_v_w}, we have $g=f_{v,w}$ for $v=s_2$ and $w=\wj=s_2s_1s_3s_2$; cf.\ \cref{ex:Gr_KLS_f_v_w}. Since $\uk v=I_1=\{1,3\}$, we see that $x\in \Guvbig$ whenever $x_4\neq0$. We have just computed that $\GR{\hjx x}=\ssaGRgjAu$, so $\dv^{-1}\hjx x=\ssavikappax$.  Factorizing the latter as an element of $\Uj_-\cdot L_J\cdot \Uj$ via~\eqref{eq:Gr_bracketsJ}, we get
\[\dv^{-1}\hjx x=\ssavikappax=\ssaetaL\cdot \ssaeta\cdot \ssaetaR,\;\; [\dv^{-1}\hjx x]_J=\ssaeta.\]
Thus we have computed $\Ft(x)=[\dv^{-1}\hjx x]_J$ from \cref{lots_of_maps}. Since $x\in \du\Uj_-$, we use \cref{zeta_inside_uP_-} to find 
\[\zetamap(x)=x\Ft(x)^{-1}=\ssazeta,\quad \text{so}\quad\zetamap(x)\dw^{-1}=\ssazetaw.\] 
Therefore the bottom-right principal minors of $\zetamap(x)\dw^{-1}$ are
\begin{equation}\label{eq:ssa_5_zeta_principal}
\Deltapm_1(\zetamap(x)\dw^{-1})=\ssazetawA,\quad \Deltapm_2(\zetamap(x)\dw^{-1})=\ssazetawB,\quad \Deltapm_3(\zetamap(x)\dw^{-1})=\ssazetawC.
\end{equation}
\end{example}
\noindent By \cref{thm:Snider_preimage_Cg}, the preimage of $\Caff_g$ under $\BAgeom$ is described by  $\trmin Ma{I_a}\neq0$ for all $a\in[n]$. Alternatively, as we showed in \cref{sec:BAgeom3_proof}, the preimage of $\Caff_g$ under $\BAgeom$ is described by  $\Deltapm_i(\zetamap(x)\dw^{-1})\neq0$ for all $i\in[n-1]$. The following result has been computationally checked for all $n\leq 5$, $k\in[n]$, and $(u,u)\leqJ (v,w)\in Q_J$:
\begin{conjecture}\label{conj::Gr_conj_zeta_trunc}
Let $(u,u)\leqJ(v,w)\in Q_J$. Define $g:=f_{v,w}$, and let $\Ical_g=(I_a)_{a\in\Z}$ be its Grassmann necklace. Suppose that $x=\gj\du\in \Guvbig$ and let $M:=\GR{x}$. Then
\begin{equation}\label{eq:Gr_conj_zeta_trunc}
\Deltapm_{n+1-i}(\zetamap(x)\dw^{-1})=\frac{\trmin M{i}{I_{i}}}{\trmin M1{I_1}} \quad \text{for all $i\in[n]$.}
\end{equation}
\end{conjecture}
\noindent For example, compare~\eqref{eq:ssa_5_zeta_principal} with~\eqref{eq:trmin_ex}. Also recall that when $i=1$, $\Deltapm_n(\zetamap(x)\dw^{-1}):=1$, so in this case~\eqref{eq:Gr_conj_zeta_trunc} holds trivially.

\def\shi{\chi}
\def\dshi{{\ddot{\chi}}}
\subsection{Total positivity}\label{sec:Gr_total-positivity}
We recall some background on the totally nonnegative Grassmannian $\Grtnn(k,n)$ of~\cite{Pos}. By a result of Whitney~\cite{Whitney}, $\Gtnn$ is the set of matrices in $\SL_n(\R)$ all of whose minors (of arbitrary sizes) are nonnegative. We have the following characterizations:
\begin{align}
\label{eq:Gr_GBtnn_minors}
\GBtnn&=\left\{xB\in \GBR\mid \fm_S(x)\geq0\text{ for all }S\subset[n]\right\},\\
\label{eq:Gr_Grtnn_minors}
\Grtnn(k,n)=\GPtnn&=\left\{xP\in \GPR\mid \fm_S(x)\geq0\text{ for all }S\in\textstyle\nchk\right\}.
\end{align}
The equality~\eqref{eq:Gr_Grtnn_minors} is due to Rietsch; see~\cite[Remark~3.8]{Lam16} for a proof. The equality~\eqref{eq:Gr_GBtnn_minors} can be proved using arguments from~\cite{Whitney} (cf.\ the proof of \cref{lemma:Demazure}). We caution the reader that the analogous statement can fail to hold for other choices of $J$. For instance, when $G=\SL_4$ and $J=\{2\}$, $\GPtnn$ does \emph{not} contain all $xP\in \GPR$ such that $\fm_S(x)\geq0$ for all $S\in {[n]\choose1}\cup{[n]\choose3}$; see~\cite[\S 10.1]{Che}.

For $f\in \Boundkn$, we let $\Povtp_f:=\Povar_f\cap \Grtnn(k,n)$ and $\Povtnn_f:=\Povarcl_f\cap \Grtnn(k,n)$. Thus for $(v,w)\in Q_J$, we have $\Povtp_{f_{v,w}}=\PRtp_v^w$ and $\Povtnn_{f_{v,w}}=\PRtnn_v^w$ by \cref{thm:Povar_PR}.

\begin{proposition}\label{prop:truncation}
  Let $\tul\leqop g\leqop h\in\Boundkn$, and let $\Ical_g=(I_a)_{a\in\Z}$ be the Grassmann necklace of $g$. Suppose that a matrix $M$ in $\uk u$-echelon form belongs to $\Povtp_h$. Then
\begin{equation}\label{eq:trmin_>0}
  \trunc Ma\in\Grtnn(k,n)\quad\text{and}\quad\trmin Ma{I_a}>0 \quad\text{for all $a\in\Z$}.
\end{equation}
\end{proposition}
\begin{proof}
Applying \cref{thm:Povar_PR}, we have $(u,u)\leqJ(v,w)\leqJ(v',w')\in Q_J$, where $g=f_{v,w}$ and $h=f_{v',w'}$. By~\eqref{Q_J_charact_eqn}, we get $v'\leq vr'\leq ur\leq wr'\leq w'$ for some $r,r'\in W_J$. 

First suppose that $a=1$. Let $x\in G$ be such that $M=\GR{\gj\du}$ and $xP\in \Povtp_h$, and define $M':=\trunc M1$. We may assume that $xB\in \Rtp_{v'}^{w'}$. By \cref{cor:hj_2_x_cell_r_w_proj}, we find that $\hjx xP\in \PRtp_{\bar v'}^u$, where $\bar v':=v'\triangleleft r_w^{-1}$ for some $r_w\in W_J$ satisfying $r_w\geq r$; see \cref{lemma:hj_2_x_cell_r_w}. This shows that $M'\in \Grtnn(k,n)$. Since $ur\leq ur_w$, we find that $ur\triangleleft r_w^{-1}\leq u$ by \cref{x*y:compare}, and therefore $ur\triangleleft r_w^{-1}=u$. Applying ${}\triangleleft r_w^{-1}$ to $v'\leq vr'\leq ur$ via \cref{x*y:compare}, we see that $\bar v'\leq (vr'\triangleleft r_w^{-1})\leq u$. Let $v=v_1v_2$ for $v_1\in W^J$ and $v_2\in W_J$ be the parabolic factorization of $v$. Then $vr'\triangleleft r_w^{-1}\in v_1W_J$, and thus $(v_1,v_1)\leqJ (\bar v',u)\in Q_J$, which is equivalent to $\fm_{\uk{v_1}}(\hjx x)>0$.  From \cref{ex:Gr_neck_v[k]} we have $\uk v=I_1$, and  $\uk{v_1}=\uk{v}$ since $v\in v_1W_J$, so $\trmin M1{I_1} = \fm_{I_1}(\hjx x)>0$. We have shown~\eqref{eq:trmin_>0} for $a=1$.  Applying the cyclic shift $\shi: \Grtnn(k,n)\to \Grtnn(k,n)$ (which takes $M$ to the matrix with rows $(M_{a+1})_{a\in[n]}$), we obtain~\eqref{eq:trmin_>0} for all $a\in\Z$.
\end{proof}
\noindent Note that our proof of \cref{prop:truncation} involves a lifting from $G/P$ to $G/B$, so it does not stay completely inside $\Gr(k,n)$.
\begin{problem}
Give a self-contained proof of \cref{prop:truncation}.
\end{problem}
\begin{example}\label{ex:ssb}
We now consider an example for the case $G/P=\Gr(\ssbk,\ssbn)$. Let $u:=\ssbu\in W^J$, so $\uk u=\ssbuk$. Consider $(v',w')\in Q_J$ given by $v':=\ssbvp$ and $w':=\ssbwp$ as in Figure~\ref{fig:Le}, so that $h:=f_{v',w'}=\ssbh$. We use Marsh--Rietsch parametrizations\footnote{For the Grassmannian case, Marsh--Rietsch parametrizations are closely related to \emph{BCFW bridge parametrizations}; see~\cite{BCFW,ABCGPT,KarDeo}.} from \cref{sec:mr-param-gbtnn} to compute $x\in G$ such that $xB\in \Rtp_{v'}^{w'}$ and $xP\in \Povtp_h$:
\[x:=\ssbdeovpwp=\ssbx,\quad M:=\GR{\gj\du}=\ssbGRgju,\]
where $\t=(t_1,t_3,t_4,t_5)\in \R_{>0}^4$. Observe that $xB\in \GBtnn$ since all flag minors of $x$ are nonnegative. (For instance, the first column of $x$ consists of nonnegative entries.) In fact, flag minors of $x$ are subtraction-free rational expressions in $\t$; cf.~\eqref{eq:lemma_01_sf}. The $n\times k$ matrix $\GR{x}$ is \emph{not} in $\uk u$-echelon form, but the matrix $M:=\GR{\gj\du}$ is. Up to a common scalar, the $2\times 2$ flag minors of $M$ are the same as the corresponding flag minors of $x$; however, other (i.e., $1\times 1$) flag minors of $M$ are not necessarily nonnegative. The  Grassmann necklace of $h$ is $\Ical_h=\ssbhneck$. Using \cref{Gr_Grneck_charact}, we check that indeed  $xP\in\Povtp_h$.

Let us choose $(v,w)\in Q_J$ with $v:=\ssbv$, $w:=\ssbw$, so that $g:=f_{v,w}=\ssbg$. The corresponding \Le -diagram is obtained from the one in Figure~\ref{fig:Le} (bottom left) by removing the dot in the bottom row. We have $(u,u)\leqJ (v,w)\leqJ (v',w')$ and $\tul\leqop g\leqop h$.  We compute the elements $\hjx=\hj_2\in\Uj_2$, $\pidup(x)$, $\Ft(x)$, and $\zetamap(x)=\pidup(x)\cdot \Ft(x)^{-1}$ from \cref{lots_of_maps}:
\begin{align*}
\gj\du&=\ssbgju, &\hjx&=\ssbhjx, &\hjx x&=\ssbhjxx,\\
\pidup(x)&=\ssbpidup, &\Ft(x)&=\ssbeta, &\zetamap(x)&=\ssbzeta.
\end{align*}
We see that all flag minors of $\hjx x$ are nonnegative; cf. \cref{lemma:hj_2_x_cell_r_w}. Observe that $\hjmp_{\gj\du}=\hjx$ by \cref{hjmap_properties}, so by \cref{zeta_inside_uP_-}, we could alternatively compute $\zetamap(x)$ as the product $\gj\du\cdot \Ft(\gj\du)^{-1}$:
\[\Ft(\gj\du)=\ssbetaB,\quad \zetamap(x)=\gj\du\cdot \Ft(\gj\du)^{-1}=\ssbgju\cdot \ssbetaBi.\]
Finally, we compute the bottom-right $i\times i$ principal minors of $\zetamap(x)\dw^{-1}$ and observe that they are all nonzero subtraction-free expressions in $\t$, agreeing with \cref{thm:zeta,thm:sf_generic_zetamap}:
\def\parzetadwi{(\zetamap(x)\dw^{-1})}
\[\zetamap(x)\dw^{-1}=\ssbzetaw,\quad \text{\begin{tabular}{lll}
$\Deltapm_1\parzetadwi=\ssbzetawA$,&& $\Deltapm_2\parzetadwi=\ssbzetawB$,\\ 
$\Deltapm_3\parzetadwi=\ssbzetawC$,&& $\Deltapm_4\parzetadwi=\ssbzetawD$.
                                      \end{tabular}}\]
Let us check that this agrees with \cref{conj::Gr_conj_zeta_trunc}. The Grassmann necklace of $g$ is $\Ical_g=\ssbgneck$ in window notation. We see that the corresponding $u$-truncated minors of $M=\GR{\gj\du}$ are indeed given by
\[\trmin M1{\ssbgAneck}=\ssbtrminA,\quad \trmin M2{\ssbgBneck}=\ssbtrminB,\quad \trmin M3{\ssbgCneck}=\ssbtrminC,\quad \trmin M4{\ssbgDneck}=\ssbtrminD,\quad \trmin M5{\ssbgEneck}=\ssbtrminE.\]

\end{example}


\def\Ummin{\Umin_-}
\def\Bmmin{\Bmin_-}
\def\Tmin{\T^\minn}
\def\Pmin{{\Pcal^\minn}}
\def\Umtnn{\U^-_{\geq0}}

\section{Further directions}\label{sec:further}
In addition to \cref{thm:main_intro} and Hersh's result \cite{Her} (cf.~\cref{cor:Her}), we expect the regularity theorem to hold for many other spaces occurring in total positivity. The most natural immediate direction is total positivity for Kac--Moody flag varieties.

Let $\Gmin$ be a minimal Kac--Moody group, $\Umin,\Ummin,\Bmin,\Bmmin$ be unipotent and Borel subgroups, and $\Waff$ be the Weyl group as in \cref{sec:KM}.  Furthermore, let $\Pmin \supset \Bmin$ denote a standard parabolic subgroup of $\Gmin$ (a group of the form $\Gmin \cap \Pcal_Y$ in the notation of \cite{Kum}).  

\begin{definition}\label{defn:TNN}
Define the \emph{totally nonnegative part} $\Umtnn$ of $\Ummin$ to be the subsemigroup generated by $\{x_{\alpha_i}(t)\mid t\in\R_{>0},\ 1\leq i\leq r\}$. 
Define the \emph{totally nonnegative part} of the flag variety $\Gmin/\Pmin$ to be the closure $(\Gmin/\Pmin)_{\geq 0} :=\overline{\Umtnn \Pmin/\Pmin}$.
\end{definition}
\noindent We remark that our notion of $\Umtnn$ coincides with the one studied recently  by Lusztig~\cite{Lus2018,Lus2019}  in the simply laced case. 

When $\Gmin$ is an affine Kac--Moody group of type $A$, Definition \ref{defn:TNN} agrees with the definition of Lam and Pylyavskyy (cf.~\cite[Theorem 2.6]{LP}) for the polynomial loop group.

\begin{conjecture}[Regularity conjecture for Kac--Moody groups and flag varieties]\label{conj:KM} \leavevmode
\begin{enumerate}
\item\label{conj:KM_G}
The intersection of $\Umtnn$ with the Bruhat stratification $\{\Bmin \dot w \Bmin \mid w \in \Waff\}$ of $\Gmin$ endows $\Umtnn$ with an (infinite) cell decomposition with closure partial order equal to the Bruhat order of $\Waff$.   Furthermore, the link of the identity in any (closed) cell is a regular CW complex homeomorphic to a closed ball.
\item\label{conj:KM_B}
The intersection of $(\Gmin/\Bmin)_{\geq 0}$ with the open Richardson stratification $\Richaff_u^v$ of $\Gmin/\Bmin$ endows $(\Gmin/\Bmin)_{\geq 0}$ with the structure of a regular CW complex.  The closure partial order is the interval order of the Bruhat order of $\Waff$, and after adding a minimum, every interval of the closure partial order is thin and shellable.
\item\label{conj:KM_P}
The intersection of $(\Gmin/\Pmin)_{\geq 0}$ with the open projected Richardson stratification $\pcvar {v,w}$ of $\Gmin/\Pmin$ endows $(\Gmin/\Pmin)_{\geq 0}$ with the structure of a regular CW complex.  The closure partial order is the natural partial order on $\Pcal$-Bruhat intervals of $\Waff$, and after adding a minimum, every interval of the closure partial order is thin and shellable.
\end{enumerate}
\end{conjecture}
\noindent Note that every interval in the Bruhat order of $\Waff$ is known to be thin and shellable \cite{BW}.  The stratification $\pcvar{v,w}$ and the $\Pcal$-Bruhat order can be defined analogously to \cite{KLS}.

We include a list of some other spaces occurring in total positivity which we expect to have a natural regular CW complex structure. 
\begin{enumerate}
\item The {\it totally nonnegative part of double Bruhat cells} \cite{FZ}.  It has been expected that a link of a double Bruhat cell inside another double Bruhat cell is a regular CW complex homeomorphic to a closed ball.  Our \cref{thm:Lk_ball} confirms this in type $A$, since double Bruhat cells for $\GL_n$ embed in the Grassmannian $\Gr(n,2n)$; see \cite[Remark 3.11]{Pos}.
\item The {\it compactified space of planar electrical networks} \cite{Lam} and the {\it space of boundary correlations of planar Ising models} \cite[Conjecture~9.1]{Ising}.  These spaces are known to be homeomorphic to closed balls \cite{GKL, Ising}, and have cell decompositions \cite{Lam,Ising} whose face poset is graded, thin, and shellable \cite{HK}.
\item {\it Amplituhedra} \cite{AT} and, more generally, {\it Grassmann polytopes} \cite{Lam16}.  Grassmann polytopes generalize convex polytopes into the Grassmannian $\Gr(k,n)$. The former are well known to be regular CW complexes homeomorphic to closed balls.  Some amplituhedra and Grassmann polytopes have been shown to be homeomorphic to closed balls in \cite{KW19,GKL,BGPZ}, though we caution that not all Grassmann polytopes are balls. 
\item The {\it totally nonnegative part of the wonderful compactification} of a semisimple algebraic group \cite{He2}.  A cell decomposition of this space was constructed in \cite{He2}.
\end{enumerate}
\noindent We expect that most spaces in this list are (complexes of) shellable TNN spaces that admit a Fomin--Shapiro atlas.

Finally, let us mention the analogy between totally nonnegative spaces
and Teichm\"uller space \cite{FoGo,Guichard, GuWi,Labourie}.  Thurston's compactification of the
Teichm\"uller space of a compact surface of genus $g\ge 2$ is
homeomorphic to a closed ball of dimension $6g-6$  \cite{Thu}, a
result that could be compared to \cref{thm:main_intro}.

\appendix

\section{Kac--Moody flag varieties}\label{sec:KM}
We recall some background on Kac--Moody groups, and refer to~\cite{Kum} for all missing definitions. We start by introducing the minimal Kac--Moody group $\Gmin$ and its flag variety $\Gmin/\Bmin$, and then explain how they relate to the polynomial loop group $\G$ and its flag variety $\G/\B$ from \cref{sec:kac-moody-groups}.  
\subsection{Kac--Moody Lie algebras}\label{sec:KM_LA}
Suppose that $\GCMaff$ is a \emph{generalized Cartan matrix}~\cite[Definition~1.1.1]{Kum}. Thus $\GCMaff$ is an $r\times r$ integer matrix for some $r\geq1$. We assume $\GCMaff$ is \emph{symmetrizable}, that is, there exists a diagonal matrix $D\in\GL_{r}(\Q)$ such that $D\GCMaff$ is a symmetric matrix.  As in~\cite[\S1.1]{Kum}, denote by $\gfr$ the \emph{Kac--Moody Lie algebra} associated to $\GCMaff$, and let $\hfr\subset \gfr$ be its \emph{Cartan subalgebra}, whose dual is denoted by $\hfrast$. Thus $\hfr$ and $\hfrast$ are vector spaces over $\C$ of dimension $\rt:=2r-\rank(\GCMaff)$, and we let $\<\cdot,\cdot\>:\hfr\times \hfrast\to \C$ denote the natural pairing.

We let $\Phiaff\subset \hfrast$ denote the \emph{root system} of $\gfr$, as defined in~\cite[\S1.2]{Kum}. Let $\{\alpha_i\}_{i=1}^r\subset \hfrast$ be the \emph{simple roots} and $\{\alphacheck_i\}_{i=1}^r\subset\hfr$ be the \emph{simple coroots}.

Let $\Phire\subset \Phiaff$ denote the set of \emph{real roots} and $\Phiim\subset \Phiaff$ denote the set of \emph{imaginary roots}, so $\Phiaff=\Phire\sqcup \Phiim$. Also let $\Phiaff=\Phiaff^+\sqcup \Phiaff^-$ denote the decomposition of $\Phiaff$ into \emph{positive} and \emph{negative} roots, and denote $\Phire^+:=\Phiaff^+\cap \Phire$ and $\Phire^-:=\Phire\cap \Phiaff^-$. Denote by $\Waff$ the \emph{Weyl group} associated to $\GCMaff$ as in~\cite[\S1.3]{Kum}. Thus $\Waff$ acts on $\Phiaff$, and preserves the subset $\Phire$. Moreover, $\Waff$ is generated by \emph{simple reflections} $s_1,\dots,s_r\in \Waff$, and $(\Waff,\{s_i\}_{i=1}^r)$ is a Coxeter group by~\cite[Proposition~1.3.21]{Kum}. We let $(\Waff, \leq)$ denote the \emph{Bruhat order} on $\Waff$ and $\ell:\Waff\to \Z_{\geq0}$ denote the length function.

 \subsection{Kac--Moody groups}\label{sec:KMgroups}
 Let $\Gmin$ be the \emph{minimal Kac--Moody group} associated to $\GCMaff$ by Kac and Peterson~\cite{KacP,PKac}; see~\cite[\S7.4]{Kum}. For each real root $\alpha\in\Phire$, there is a one-parameter subgroup $\U_\alpha\subset \Gmin$ by~\cite[Definition~6.2.7]{Kum}.\footnote{The results in~\cite{Kum} are usually stated for the \emph{maximal Kac--Moody group} which he denotes by $\G$. However, these results apply to $\Gmin$ as well; see \cref{rmk:Gmin_vs_hG}.} For each $\alpha\in\Phire$, we fix an isomorphism $x_\alpha:\C\xrasim \U_\alpha$ of algebraic groups. Similarly to the subgroups $U,U_-,T,B,B_-$ of $G$, we have subgroups $\Umin,\Ummin,\Tmin,\Bmin,\Bmmin$ of $\Gmin$. The subgroup $\Umin$ is generated by $\{\U_\alpha\}_{\alpha\in \Phire^+}$, and $\Ummin$ is generated by $\{\U_\alpha\}_{\alpha\in \Phire^-}$. Next, $\Tmin$ is an $\rt$-dimensional algebraic torus defined in~\cite[\S6.1.6]{Kum}, $\Bmin=\Tmin\ltimes \Umin$ is the \emph{standard positive Borel subgroup} and $\Bmmin=\Tmin\ltimes \Ummin$ is the  \emph{standard negative Borel subgroup}. 

 We define a \emph{bracket closed subset} $\Theta\subset \Phire$ in the same way as in \cref{sec:subgroups-u}, and for a bracket closed subset $\Theta\subset \Phire^+$ (respectively, $\Theta\subset \Phire^-$), we have a subgroup $\U(\Theta)\subset \Umin$ (respectively, $\U_-(\Theta)\subset \Ummin$), generated by $\U_\alpha$ for $\alpha\in\Theta$; see~\cite[6.1.1(6) and \S6.2.7]{Kum}. For $w\in \Waff$,  $\Inv(w):=\Phiaff^+\cap w^{-1}\Phiaff^-\subset \Phire^+$ is a bracket closed subset of size $\ell(w)$; cf.~\cite[Example~6.1.5(b)]{Kum}. We state the Kac--Moody analog of \cref{U(R)_generated}.
 \begin{lemma}[{\cite[Lemma~6.1.4]{Kum}}]\label{KM_U(R)_generated}
Suppose that  $\Theta=\bigsqcup_{i=1}^n \Theta_i$ and $\Theta,\Theta_1,\dots,\Theta_n\subset \Phire^+$ are finite bracket closed subsets. Then $\U(\Theta),\U(\Theta_1), \dots, \U(\Theta_n)$ are  finite-dimensional unipotent algebraic groups, and the multiplication map gives a biregular isomorphism
\begin{equation}\label{eq:KM_U(R)_generated}
  \U(\Theta_1)\times\dots\times \U(\Theta_n)\xrasim\U(\Theta).
\end{equation}
\end{lemma}

\subsection{Kac--Moody flag varieties}
The Weyl group $\Waff$ equals $N_{\Gmin}(\Tmin)/\Tmin$ , where $N_{\Gmin}(\Tmin)$ is the normalizer of $\Tmin$ in $\Gmin$; cf.~\cite[Lemma~7.4.2]{Kum}. For $f\in\Waff$, we denote by $\df\in\Gmin$ an arbitrary representative of $f$ in $N_{\Gmin}(\Tmin)$.

By~\cite[Lemma~7.4.2, Exercise~7.4.E(9), and Theorem~5.2.3(g)]{Kum}, we have Bruhat and Birkhoff decompositions of $\Gmin$:
\begin{equation}\label{eq:KM_Bruhat_Birkhoff}
  \Gmin=\bigsqcup_{f\in \Waff} \Bmin \df\Bmin,\quad   \Gmin=\bigsqcup_{h\in \Waff} \Bmmin \dh\Bmin.
\end{equation}
We let $\Gmin/\Bmin$ denote the \emph{Kac--Moody flag variety} of $\Gmin$. 
For each $h,f\in \Waff$, we have Schubert cells $\Xaff^f:=\Bmin \df\Bmin/\Bmin$ and opposite Schubert cells $\Xaff_h:=\Bmmin\dh\Bmin/\Bmin$.
If $h\not\leq f\in\Waff$ then by~\cite[Lemma~7.1.22(b)]{Kum}, $\Xaff_h\cap \Xaff^f=\emptyset$. For $h\leq f$, we define $\Richaff_h^f:=\Xaff_h\cap \Xaff^f$. Therefore~\eqref{eq:affine_Xaff_sqcup_Raff}  follows from~\eqref{eq:KM_Bruhat_Birkhoff}. The flag variety $\Gmin/\Bmin$ is a \emph{projective ind-variety} by~\cite[\S7.1]{Kum}. The Schubert cell $\Xaff^f$ and Schubert variety $\Xaffcl^f$ are finite-dimensional subvarieties, while the opposite Schubert cell $\Xaff_h$ and opposite Schubert variety $\Xaffcl_h$ are ind-subvarieties.

\begin{proposition}\label{Xaff_cl_irreducible}
Let $h\leq f\in\Waff$. Then $\Xaffcl_h\cap \Xaffcl^f$ is a closed irreducible $(\ell(f)-\ell(h))$-dimensional subvariety of $\Xaffcl^f$, and $\Richaff_h^f$ is an open dense subset of $\Xaffcl_h\cap \Xaffcl^f$.
\end{proposition}
\begin{proof}
  By~\eqref{eq:KM_affine_charts}, $\Xaff^f$ is $\ell(f)$-dimensional, and by~\cite[Lemma~7.3.10]{Kum}, $\Xaff_h\cap \Xaffcl^f$ has codimension $\ell(h)$ in $\Xaffcl^f$. The rest follows by~\cite[Proposition~6.6]{Kum2}.
\end{proof}

For $g\in \Waff$, let $\Caff_g:=\dg \Bmmin\Bmin/\Bmin$.  We have
\begin{equation}\label{eq:KM_Rich_stratif}
  \Gmin/\Bmin=\bigsqcup_{h\leq f} \Richaff_h^f\quad\text{and}\quad \Caff_g=\bigsqcup_{h\leq g\leq f} (\Caff_g\cap \Richaff_h^f),
\end{equation}
where the unions are taken over $h,f\in \Waff$. The first part of~\eqref{eq:KM_Rich_stratif} follows from~\eqref{eq:KM_Bruhat_Birkhoff}, and for the second part, see the proof of \cref{FS_iso_Rich}.

\begin{remark}\label{rmk:Gmin_vs_hG}
Let $\hG\supset \Gmin$ be the ``maximal'' Kac--Moody group (denoted $\G$ in \cite{Kum}) associated to $\GCMaff$, and let $\hB\supset \Bmin$ be its standard positive Borel subgroup. Then the standard negative Borel subgroup of $\hG$ is still $\Bmmin$. By~\cite[7.4.5(2)]{Kum}, we may identify $\Gmin/\Bmin$ with $\hG/\hB$. By~\cite[7.4.2(3)]{Kum}, $\Xaff^f$ coincides with the variety $\hB f \hB/\hB$ in~\cite[Definition~7.1.13]{Kum} for $f\in\Waff$. Similarly, for $h\in \Waff$, $\Xaff_h=\Bmmin\cdot \dh\Bmin/\Bmin$ coincides with the variety $\B_\emptyset^h:=\Bmmin h\hB/\hB$ defined in the last paragraph of~\cite[\S7.1.20]{Kum}.  
\end{remark}

\def\center{C}
\def\centeraff{\tilde{C}}
\def\pone{^1}
\subsection{Affine Kac--Moody groups and polynomial loop groups}\label{sec:AKM}
Suppose that $\GCMaff$ is the affine Cartan matrix associated to a simple and simply-connected algebraic group $G$. Thus we have $r=|I|+1$, $\rt=|I|+2$, and $\GCMaff$ is defined by~\cite[13.1.1(7)]{Kum}. Let $\G$ denote the polynomial loop group from \cref{sec:kac-moody-groups}. Our goal is to explain that the flag varieties $\G/\B$ and $\Gmin/\Bmin$ are isomorphic.

Let $\center\subset T\subset G$ be the center of $G$, and let $\centeraff\subset\Tmin\subset\Gmin$ be the center of $\Gmin$; see~\cite[Lemma~6.2.9(c)]{Kum}.  By~\cite[Corollary~13.2.9]{Kum}, there exists a surjective group homomorphism $\psi:\Gmin\to (\Cast\ltimes \G)/\center$ with kernel $\centeraff$, where $\Cast$ acts on $\G$ as in \cref{sec:KM_torus-action}; see also~\cite[Definition~13.2.1]{Kum}. The groups $\U,\U_-\subset \G$ are identified with the groups $\Umin,\Ummin\subset \Gmin$, and we have $\T/\center\cong \Tmin/\centeraff$. Thus $\psi$ induces an isomorphism $\Gmin/\Bmin\xrasim \G/\B$ between the affine Kac--Moody flag variety and the affine flag variety. The Weyl groups $\Waff$ of $\G$ and $\Gmin$ are isomorphic by~\cite[Proposition~13.1.7]{Kum}, and the root systems $\Delta$ coincide by~\cite[Corollary~13.1.4]{Kum}. Therefore the subsets $\Xaff^f$, $\Xaff_h$, $\Richaff_h^f$, and $\Caff_g$ of $\G/\B$ get sent by $\psi$ to the corresponding subsets of $\Gmin/\Bmin$. As explained in the last paragraph of~\cite[\S13.2.8]{Kum}, $G$ can be viewed as a subset of $\Gmin$ as well, and the restriction of $\psi$ to $G$ is the quotient map $G\to G/C$.

We justify some of the other statements that we used in \cref{sec:loop-groups-affine,sec:KM_torus-action}. 
For~\eqref{eq:tau_on_Phi}, see~\cite[\S13.1]{Kum}. For~\eqref{eq:Uvp_Uvn_normal}, see~\cite[\S6.1.13]{Kum}. For a description of $Y(\T)$ from~\cref{sec:KM_torus-action}, see~\cite[\S13.2.2]{Kum}.  For a description of the pairing $\<\cdot,\cdot\>:Y(\T)\times X(\T)\to \Z$ in the same section, see~\cite[\S13.1.1]{Kum}.

\subsection{Gaussian decomposition and affine charts} \label{sec:KMapp_Gauss}
By~\cite[Theorem~7.4.14]{Kum}, $\Gmin$ is an \emph{affine ind-group}. Similarly, $\Umin$, $\Ummin$, $\T$, $\Bmin$, and $\Bmmin$ are affine ind-groups; see e.g.~\cite[\S7.4]{Kum} and~\cite[Corollary~7.3.8]{Kum}.

Let $\Gmin_0:=\Bmmin\Bmin$ and $g\in\Waff$. Recall the subgroups $\U_1(g)$ and $\U_2(g)$ from~\eqref{eq:Caff_U_1(g)_U_2(g)}. Then $\U_1(g)$ is a closed $\ell(g)$-dimensional subgroup of $\Umin\cong \U$, and $\U_2(g)$ is a closed ind-subgroup of $\Ummin\cong \U_-$. 

\begin{proof}[Proof of~\cref{lemma:KM_gaussian}.]
  For~\itemref{KM_gaussian}, see~\cite[Proposition~7.4.11]{Kum}. For~\itemref{KM_U1_U2}, we use an argument given in~\cite[Proposition~2.5]{HW}. Both maps are bijective morphisms by~\cite[Lemma~6.1.3]{Kum}. In particular, it follows that $\dg \Ummin\dg^{-1}\subset \Gmin_0$ and for $x\in \dg \Ummin\dg^{-1}$, we have $[x]_0=1$. The inverse maps are given by $\mu_{21}^{-1}(x)= ([x]_-,[x]_+)$,  $\mu_{12}^{-1}(x)= ([x^{-1}]_+^{-1},[x^{-1}]_-^{-1})$. They are regular morphisms by~\itemref{KM_gaussian}, which proves~\itemref{KM_U1_U2}.
\end{proof}

\begin{proof}[Proof of~\eqref{eq:KM_affine_charts}.]
The map $  \dg\Ummin\dg^{-1}\xrasim \Caff_g$ is a biregular isomorphism for $g=\id$ by~\cite[Lemma~7.4.10]{Kum}. Since $\Waff$ acts on $\Gmin/\Bmin$ by left multiplication, the case of general $g\in \Waff$ follows as well. Since $\U_1(g)$, $\U_2(g)$ are closed ind-subvarieties of $\dg \Ummin\dg^{-1}$ and $\Xaff^g$, $\Xaff_g$ are closed ind-subvarieties of $\Caff_g$, it suffices to show that the image of $\U_1(g)$ equals $\Xaff^g$ while the image of $\U_2(g)$ equals $\Xaff_g$.  By~\cite[Exercise~7.4.E(9) and 5.2.3(11)]{Kum}, we have
\begin{align*}
\Umin&=(\Umin\cap \dg \Ummin\dg^{-1})\cdot(\Umin\cap \dg \Umin\dg^{-1})=\U_1(g)\cdot (\Umin\cap \dg \Umin\dg^{-1}),\\
\Ummin&=(\Ummin\cap \dg \Ummin\dg^{-1})\cdot(\Ummin\cap \dg \Umin\dg^{-1})=\U_2(g)\cdot (\Ummin\cap \dg \Umin\dg^{-1}).
\end{align*}
Thus 
\begin{align*}
\Bmin \dg\Bmin&=\U_1(g)\cdot(\Umin\cap \dg \Umin\dg^{-1})\cdot \dg\Bmin= \U_1(g)\cdot \dg\cdot \Bmin,\\
\Bmmin \dg\Bmin&=\U_2(g)\cdot(\Ummin\cap \dg \Umin\dg^{-1})\cdot \dg\Bmin= \U_2(g)\cdot \dg\cdot \Bmin.\qedhere
\end{align*}
\end{proof}


\bibliographystyle{alpha}
\bibliography{positroids}

\end{document}